%% file: fichier-amsart.tex
\documentclass[12pt]{amsart}

\usepackage{amsmath}
\usepackage{amsfonts}
\usepackage{latexsym}
\usepackage{graphicx}
\usepackage{amssymb}
\usepackage{amsthm}
\usepackage[margin=2cm]{geometry}
\usepackage{url}
\usepackage{color}
\usepackage{enumerate}
\usepackage[shortlabels]{enumitem} %
\usepackage[small]{caption}
\usepackage{cite}       %

\usepackage{bm}

\usepackage{todonotes}

\usepackage{tikz}
\usepackage{tikz-cd}

\usepackage{hyperref}

\usepackage[T1]{fontenc}     %
\usepackage{lmodern}         %
\usepackage[utf8]{inputenc}  %

\newtheorem{definition}{Definition}[section]
\newtheorem{lemma}[definition]{Lemma}
\newtheorem{proposition}[definition]{Proposition}

\newtheorem{corollary}[definition]{Corollary}

\newtheorem{theorem}[definition]{Theorem}
\newtheorem{remark}[definition]{Remark}

\newtheorem{maintheorem}{Theorem}

\newtheorem*{corollary*}{Corollary}

\newcommand{\Proofof}{Proof of }

\title{Almost everywhere balanced sequences of complexity $2n+1$}
\author[J.~Cassaigne]{Julien Cassaigne}
\address[J.~Cassaigne]{Aix-Marseille Université, CNRS, Centrale Marseille, Institut de mathématiques de Marseille, I2M - UMR 7373, 13453 Marseille, France}
\email{julien.cassaigne@math.cnrs.fr}

\author[S.~Labb\'e]{S\'ebastien Labb\'e}
\address[S.~Labb\'e]{Univ. Bordeaux, CNRS,  Bordeaux INP, LaBRI, UMR 5800, F-33400, Talence, France}
\email{sebastien.labbe@labri.fr}

\author[J.~Leroy]{Julien Leroy}
\address[J.~Leroy]{Département de mathématique,  Université de Liège,  12 Allée de la
découverte (B37), 4000 Liège, Belgique}
\email{j.leroy@uliege.be}

\date{\today}

\begin{document}



\begin{abstract}
    \input{abstract.tex}
\end{abstract}

\maketitle

\input{fichier-content}

\bibliographystyle{myalpha} %
\bibliography{biblio}

\end{document}

%% file: abstract.tex
We study ternary sequences associated with a multidimensional continued fraction algorithm introduced by the first author.
The algorithm is defined by two matrices and we show that it is measurably isomorphic to the shift on the set $\{1,2\}^\mathbb{N}$ of directive sequences.
For a given set $\mathcal{C}$ of two substitutions, we show that there exists a $\mathcal{C}$-adic sequence for every vector of letter frequencies or, equivalently, for every directive sequence.
We show that their factor complexity is at most $2n+1$ and is $2n+1$ if and only if the letter frequencies are rationally independent if and only if the $\mathcal{C}$-adic representation is primitive.
It turns out that in this case, the sequences are dendric.
We also prove that $\mu$-almost every $\mathcal{C}$-adic sequence is balanced, where $\mu$ is any shift-invariant ergodic Borel
probability measure on $\{1,2\}^\mathbb{N}$ giving a positive measure to the cylinder $[12121212]$. 
We also prove that the second Lyapunov exponent of the matrix cocycle associated with the measure $\mu$ is negative.

%% file: fichier-content.tex
\newcommand{\N}{\mathbb{N}}
\newcommand{\Z}{\mathbb{Z}}
\newcommand{\Q}{\mathbb{Q}}
\newcommand{\R}{\mathbb{R}}
\newcommand{\bx}{\mathbf{x}}
\newcommand{\A}{\mathcal{A}}
\newcommand{\C}{\mathcal{C}}
\newcommand{\E}{\mathcal{E}}
\renewcommand{\P}{\mathbb{P}}
\newcommand{\MONE}{{\arraycolsep=2pt\left(\begin{array}{rrr}
1 & 1 & 0 \\
0 & 0 & 1 \\
0 & 1 & 0
\end{array}\right)}}
\newcommand{\MTWO}{{\arraycolsep=2pt\left(\begin{array}{rrr}
0 & 1 & 0 \\
1 & 0 & 0 \\
0 & 1 & 1
\end{array}\right)}}
\newcommand{\AONE}{{\arraycolsep=2pt\left(\begin{array}{rr}
1 & 1 \\
0 & 1
\end{array}\right)}}
\newcommand{\ATWO}{{\arraycolsep=2pt\left(\begin{array}{rr}
1 & 0 \\
1 & 1
\end{array}\right)}}

\def\ba{\mathbf{a}}
\def\bc{\mathbf{c}}
\def\bf{\mathbf{f}}
\def\bm{\mathbf{m}}
\def\bs{\mathbf{s}} 
\def\bu{\mathbf{u}}
\def\bv{\mathbf{v}}
\def\bw{\mathbf{w}}
\def\by{\mathbf{y}}
\def\bz{\mathbf{z}}
\newcommand\transpose[1]{\vphantom{#1}^{\mathsf{T}}\!#1}
\newcommand\transposeENV[1]{\strut^{\mathsf{T}}\!#1}

\def\Bcal{\mathcal{B}}
\def\Ical{\mathcal{I}}
\def\Mcal{\mathcal{M}}
\def\Pcal{\mathcal{P}}
\def\Xcal{\mathcal{X}}
\def\Lcal{\mathcal{L}}
\def\Scal{\mathcal{S}}

\def\alph{\mathrm{alph}}
\def\dim{\mathrm{dim}}
\def\Lexless{\mathrm{Lexless}}

\def\Msf{\mathsf{M}}

\def\mapC{\mathsf{c}}
\def\mapS{\mathsf{s}}

\setcounter{tocdepth}{1}
\tableofcontents

\section{Introduction}

A theorem of Dirichlet says that every positive irrational number $\alpha$
has infinitely many rational approximations $\frac{p}{q}\in\Q$
such that $|\alpha-\frac{p}{q}|<\frac{1}{q^2}$.
Such approximations can be computed from the continued
fraction expansion of $\alpha$ 
\[
    \alpha=[a_0;a_1,a_2,\dots] 
    =
    a_0
+ \frac{\displaystyle 1}{\displaystyle a_1
+ \frac{\displaystyle 1}{\displaystyle a_2
+ \frac{\displaystyle 1}{\displaystyle \dots}}}
\]
where $a_0\in\N$ and $a_1,a_2,\ldots\in\N\setminus\{0\}$.
Indeed, for all $n\in\N$, the truncation $\frac{p_n}{q_n}=[a_0;a_1,\ldots,a_n]$
provides a sequence $(p_n/q_n)_{n\in\N}$ of rational approximations of $\alpha$
called convergents satisfying Dirichlet's theorem.
Equivalently, the convergents $p_n/q_n$ can be computed from a product of the matrices
$A_1=\left(\begin{smallmatrix}1&1\\0&1\end{smallmatrix}\right)$
    and
$A_2=\left(\begin{smallmatrix}1&0\\1&1\end{smallmatrix}\right)$
    involving the above sequence of partial quotients:
\[
    \left(\begin{array}{cc} p_{2n+1} & p_{2n} \\ q_{2n+1} & q_{2n} \end{array}\right)
	= A_1^{a_0} A_2^{a_1} A_1^{a_2} \cdots A_2^{a_{2n+1}}.
\]
The convergence of $p_n/q_n$ to $\alpha$ then implies that 
\begin{equation}
\label{eq:2x2-matrix-convergence}
	\left(\begin{array}{c} \alpha \\ 1 \end{array}\right) \R_{\geq 0}
	=
	\bigcap_{k \geq 0} A_{i_0} A_{i_1} \cdots A_{i_k} \R^2_{\geq 0} 
\end{equation}
where the sequence $(i_n)_{n \in \N} \in \{1,2\}^\N$ 
is $1^{a_0}2^{a_1}1^{a_2}\cdots1^{a_{2k}}2^{a_{2k+1}}\cdots$.
Equation~\eqref{eq:2x2-matrix-convergence} holds even
if $1$ and $2$ do not both occur infinitely many times 
in $(i_n)_{n \in \N}$, in which case $\alpha$ is rational.
	If $\Delta = \{(x,y) \in \mathbb{R}_{\geq 0}^2 \mid x+y=1\}$ denotes
    the projection of the positive cone,
    Equation~\eqref{eq:2x2-matrix-convergence} defines a continuous and onto map
	$\pi: \{1,2\}^\N \to \Delta$.
This map is almost one-to-one and its (almost everywhere) inverse is obtained by iterating the normalized Euclid
algorithm $f_E$ which successively applies either $\bx\mapsto A_1^{-1}\bx/\|A_1^{-1}\bx\|_1$
or $\bx\mapsto A_2^{-1}\bx/\|A_2^{-1}\bx\|_1$,
according to whether $\bx\in A_1\R^2_{\geq 0}$ or $\bx\in A_2\R^2_{\geq 0}$.
Thus the shift map on $\{1,2\}^\N$ defines a symbolic representation of the dynamical system $(\Delta,f_E)$.

Sturmian words give a combinatorial flavor to
Equation~\eqref{eq:2x2-matrix-convergence}.
With the matrices $A_1$ and $A_2$ are respectively associated the substitutions
$ s_{1}:
1 \mapsto 1,
2 \mapsto 12$
and
$s_{2}:
1 \mapsto 21,
2 \mapsto 2.$
With the directive sequence $(i_n)_{n \in \N} \in \{1,2\}^\N$ is then associated the $\{s_1,s_2\}$-adic word $\bw \in \{1,2\}^\N$:
\begin{equation}\label{eq:stumian-s-adic}
    \bw = \lim_{n\to\infty} s_{i_0}s_{i_1}\cdots s_{i_n}(1^\omega)
\end{equation}
which is a Sturmian word \cite{MR1970391} if both letters $1$ and $2$ appear infinitely often in the directive sequence.
Since $A_j$ is the incidence matrix of the substitution $s_i$ for $i\in\{1,2\}$, Equation~\eqref{eq:2x2-matrix-convergence} ensures that
the vector of frequencies of letters in $\bw$ exists and is equal to $\pi((i_n)_{n \in \N}) = \frac{1}{1+\alpha}(\alpha,1)$.
Recall that the incidence matrix of a substitution $\sigma:A^* \to A^*$ is the matrix $M_\sigma = (|\sigma(a)|_b)_{b,a \in A}$, where $|u|_v$ stands for the number of occurrences of a word $v$ in a word $u$.
It is easily seen that for any word $w \in A^*$, $M_\sigma (|w|_a)_{a \in A} = (|\sigma(w)|_a)_{a \in A}$.

Sturmian words form a deeply studied class of binary words with lots of equivalent definitions~\cite{MR1905123}.
They are for instance the aperiodic words with minimal factor complexity $\#\Lcal_\bw(n)=n+1$~\cite{MR0322838}, where $\Lcal_\bw(n)$ denotes the language of words of length $n$ of $\bw \in A^\N$, i.e., 
$\Lcal_\bw(n)=\{u \in A^n \mid u \text{ occurs in } \bw\}$.
Sturmian words are also the aperiodic 1-balanced binary words~\cite{MR0000745}, where an infinite word $\bw \in A^\N$ is $K$-balanced if any two finite words of the same length occurring in $\bw$ have, up to $K$, the same number of occurrences of each letter.
The balance property allows to prove that for any Sturmian word $\bw$, the frequencies of $1$ and $2$ exist and are irrational. 
More than that, any Sturmian word $\bw$ has uniform word frequencies, that is, for all finite word $u$ occurring in $\bw$, the ratio $\frac{|w_k w_{k+1} \cdots w_{k+n}|_u}{n+1}$ has a limit $f_u$ when $n$ goes to infinity, uniformly in $k$.

\subsection*{Results}
We consider an extension of Equation~\eqref{eq:2x2-matrix-convergence} to a set of two $3\times3$ matrices.
Using a generalization of Euclid's algorithm, which is thus named a Multidimensional Continued Fraction Algorithm (MCFA), we show that these matrices allow to represent any direction in the positive cone $\R_{\geq0}^3$.
Doing so, we generalize Sturmian words on a three-letter alphabet by extending Equation~\eqref{eq:stumian-s-adic} to two well-chosen substitutions.
We obtain words $\bw$ of complexity $2n+1$ that are balanced for
almost every given vector of letter frequencies.
This article extends our previous work~\cite{MR3703620} presented
during the conference WORDS 2017.

The two matrices are
\[
C_{1}=\MONE
\qquad\text{and}\qquad
C_{2}=\MTWO
\]
and we show that for each sequence $(i_n)_{n\in\N}\in\{1,2\}^\N$, the set
$    \bigcap_{n\geq0} C_{i_0}C_{i_1}\cdots C_{i_n} \R^3_{\geq0}
$
is one-dimensional.
This property, sometimes called \emph{weak convergence}, is not satisfied by all choices of $3\times 3$ matrices. 
For instance, Nogueira proved that the Poincaré MCFA is not convergent~\cite{MR1336331}.
In our case, convergence allows to define a continuous map
$\pi:\{1,2\}^\N \to \Delta =\{\bx\in\mathbb{R}^3_{\geq 0}\mid\Vert\bx\Vert_1=1\}$ by 
\begin{equation}\label{eq:def-pi-C1-C2}
	\pi((i_n)_{n \in \N})\, \R_{\geq0}
	=
	\bigcap_{n\geq0} C_{i_0}C_{i_1}\cdots C_{i_n} \R^3_{\geq0}.
\end{equation}
This map is not injective, as for example $\pi(1222\dots)=(0,1,0)=\pi(2111\dots)$,
but it is onto. 
We also show that $\pi$ is injective exactly on the set $\mathcal{P}$ of primitive sequences, i.e., sequences $(C_{i_n})_{n \in \N}$ such that for all $m$ and all large enough $n>m$, $C_{i_m} \cdots C_{i_n}$ has only positive entries.
Furthermore, the image $\pi(\mathcal{P})$ is the set $\mathcal{I}$ of normalized vectors with rationally independent entries.
The inverse of $\pi:\mathcal{P} \to \mathcal{I}$ is given by the MCFA introduced by the first author~\cite{cassaigne_algorithme_2015} that consists in iterating the map $f_C$ on $\bx \in \mathcal{I}$ that applies either $\bx\mapsto C_1^{-1}\bx/\|C_1^{-1}\bx\|_1$
or $\bx\mapsto C_2^{-1}\bx/\|C_2^{-1}\bx\|_1$
according to whether 
$\bx\in C_1\R^3_{\geq 0}$ or $\bx\in C_2\R^3_{\geq 0}$.
Thus we obtain a similar symbolic representation as for the classical Euclid algorithm. 

\begin{maintheorem}
\label{maintheorem:conjugacy to shift}
	The symbolic dynamical system $(\{1,2\}^\N,S)$ is a symbolic representation of $(\Delta,f_C)$.
More precisely, 
\begin{itemize}
\item
for any shift-invariant Borel probability measure $\mu$ on $\{1,2\}^\N$ such that $\mu(\mathcal{P})=1$, the map $\pi:(\{1,2\}^\N,S,\mu) \to (\Delta,f_C,\pi_*\mu)$ is a measure-preserving isomorphism; 
\item
for any $f_C$-invariant Borel probability measure $\nu$ on $\Delta$ such that $\nu(\mathcal{I})=1$, the map $\pi:(\{1,2\}^\N,S,\pi^{-1}_*\nu) \to (\Delta,f_C,\nu)$ is a measure-preserving isomorphism.
\end{itemize} 
\end{maintheorem}

This result in particular applies to any positive Bernoulli measure $\beta$ on $\{1,2\}^\N$ and to the $f_C$-invariant probability measure $\xi$ defined by the density function $6/(\pi^2(1-x_1)(1-x_3))$~\cite{arnoux_labbe_2017}.
Observe that any Bernoulli measure on $\{1,2\}^\N$ is ergodic and that the measure $\xi$ is also ergodic~\cite{fougeron_simplicity_2021}.
Thus the pointwise ergodic theorem may be applied to obtain properties for Bernoulli-almost every directive sequence $(i_n)_{n \in \N}$ or for Lebesgue-almost every vector $\bx$.
Theorem~\ref{maintheorem:combinatoire almost always} below is an example of such a result. 

We pursue the analogy with Euclid's algorithm by giving a combinatorial flavor to the symbolic representations $(i_n)_{n \in \N} \in \{1,2\}^\N$.
We consider the substitutions
\[
c_{1}:
\begin{cases}
1 \mapsto 1 	\\
2 \mapsto 13	\\
3 \mapsto 2
\end{cases}
\qquad\text{and}\qquad
c_{2}:
\begin{cases}
1 \mapsto 2\\
2 \mapsto 13\\
3 \mapsto 3
\end{cases}
\]
whose incidence matrices are respectively $C_1$ and $C_2$ and we show that the
class of $\C$-adic words with $\C=\{c_1,c_2\}$
provides a nice generalization of Sturmian words over
a three-letter alphabet.
We indeed have the following interpretations of the previous discussion:
\begin{itemize}
\item
by weak convergence, the frequencies of letters exist in every $\C$-adic word;
\item
by surjectivity of $\pi$, every $\bx \in \Delta$ is the vector of letter frequencies of a $\C$-adic word;
\item
the bijection $\pi:\mathcal{P} \to \mathcal{I}$ induces a bijection between primitive $\C$-adic words and vectors of letter frequencies with rationally independent entries.
\end{itemize}

We give another equivalence of primitive $\C$-adic words in terms of their factor complexity, generalizing the Sturmian case.
We also show that the primitive $\C$-adic words are exactly the $\C$-adic words that are dendric, a property recently introduced under the name of ``tree sets''~\cite{MR3320917} (see Section~\ref{sec:factor complexity} for the definition).

\begin{maintheorem}
\label{maintheorem:caracterization 2n+1}
Let $\bw$ be a $\C$-adic word with directive sequence $(i_n)_{n \in \N}$.
The following are equivalent.
\begin{enumerate}[\rm (i)]
\item
$\bw$ has factor complexity $p(n)=2n+1$ for all $n \in \mathbb{N}$;
\item
the frequencies of letters in $\bw$ are rationally independent;
\item
    $(C_{i_n})_{n \in \N}$ is primitive;
\item
$\bw$ is a uniformly recurrent dendric word.
\end{enumerate}
\end{maintheorem}

The last property of Sturmian words that we consider is their balancedness.
Not all primitive $\C$-adic words are balanced~\cite{CRMATH_2021__359_4_399_0}, but we prove that almost all of them are (for many measures).
Our proof is based on the method proposed by Avila and
Delecroix~\cite{MR4043208} for Brun and Fully Subtractive MCFA.
It consists in applying the pointwise ergodic theorem to show that some fixed contracting matrix appears sufficiently often in almost every sequence $(C_{i_n})_{n \in \N}$.
The same method allows to show that the second Lyapunov exponent
is negative. The definition of Lyapunov exponents can be found
in Section~\ref{sec:lyapunov}.

An application of multidimensional continued fraction algorithms is
to provide simultaneous Diophantine approximation of a vector of real numbers
\cite{MR568710}.
The quality of the approximations can be evaluated in terms of the first two Lyapunov exponents of the MCFA~\cite{MR1156412,MR1230366}.
In particular, if the second Lyapunov exponent is negative, this implies that
the algorithm is strongly convergent \cite{MR1960307,MR1804954,MR1960306}.

\begin{maintheorem} 
    \label{maintheorem:combinatoire almost always}
    Let $\mu$ be a shift-invariant ergodic 
    Borel probability measure on $\{1,2\}^\N$.
    If 
    \[
    	\mu([12121212])>0,
    \] 
    then for $\mu$-almost every directive sequence 
    $(i_n)_{n\in\N} \in \{1,2\}^\N$,
    the word $\bw=\lim_{n\to\infty}c_{i_0}\dots c_{i_n}(1^\omega)$ 
    is balanced
    and the second Lyapunov exponent $\theta_2^\mu$ 
    of the cocycle with matrices $\{C_1,C_2\}$ is negative.
\end{maintheorem}

This result in particular applies to any positive Bernoulli measure and to the measure $\pi^{-1}_*(\xi)$.
Thus it extends a result of Berthé, Steiner and Thuswaldner~\cite{MR4194166} who proved that the second Lyapunov
exponent is negative for the measure $\pi_*^{-1}(\xi)$.
Observe that $\pi^{-1}_*(\xi)$ is not a Bernoulli measure (see Remark~\ref{rem:not-bernouilli}) so all these measures are pairwise mutually singular.

It turns out that the map $f_C$ is conjugate 
with a semi-sorted version of 
another MCFA, the Selmer algorithm \cite{MR0130852,schweiger}
(see Section~\ref{sec:selmer}).
Also note that Selmer algorithm is conjugate on the absorbing simplex to
M\"onkemeyer's algorithm \cite{MR64084} (see \cite{MR2413304}).

\subsection*{Example and applications}

Consider the periodic sequence $121212\cdots$.
We have that
\[
\pi(121212\cdots)=
\frac{1}{\beta^2+1}
\left(\begin{array}{c}\beta\\\beta^2-\beta\\1\end{array}\right)
\approx
\left(\begin{array}{c}
0.4302\\ 0.3247\\ 0.2451
\end{array}\right)
\]
is a positive right eigenvector of the primitive matrix $C_1C_2$
associated with the Perron-Frobenius eigenvalue
$\beta\approx 1.7548$ of $C_1C_2$. 
It is the positive root of the characteristic polynomial $x^3-2x^2+x-1$
of $C_1C_2$.
The infinite word on the alphabet $\{1,2,3\}$ 
obtained by applying our MCFA to the above vector
is the
$\C$-adic word which is 
the unique fixed point of the substitution $c_1c_2:1\mapsto 13, 2\mapsto 12, 3\mapsto 2$:
\[
    \bw = (w_n)_{n\geq0} = \lim_{k\to\infty}(c_1c_2)^k(1) = 
1321213121321312132121321312132121312132\cdots
\]
whose set of factors of lengths 0, 1, 2, 3 and 4 are listed in the following table:
\[
\begin{array}{l|l}
n & 2n+1 \text{ factors of length } n\\
\hline
0 & \{\varepsilon\}\\
1 & \{1, 2, 3\}\\
2 & \{12, 13, 21, 31, 32\}\\
3 & \{121, 131, 132, 212, 213, 312, 321\}\\
4 & \{1213, 1312, 1321, 2121, 2131, 2132,
		3121, 3212, 3213\}
\end{array}
\]
The left eigenvector of $C_1C_2$ 
associated with the dominant eigenvalue $\beta$
is $u=(1, \beta^2-\beta, \beta-1)$.
We define the map
$h:\{1,2,3\}\to\mathbb{C}$
by
$h(1)=1$, $h(2)={\beta^*}^2-\beta^*$ and $h(3)=\beta^*-1$
where
$\beta^*\approx 0.12256 + 0.74486i$
is one of the two complex Galois conjugates of $\beta$.
Observe that the vector $u^*=(h(1), h(2), h(3))$ is 
the image of $u$ under
the automorphism of the field $\Q(\beta)$ defined by $\beta\mapsto\beta^*$.
The scalar product of $u^*$ with $\pi(121212\cdots)$ is zero.
Thus, as $w$ is balanced,
the partial sums
$S^h(N)=\sum_{i=0}^{N-1} h(w_i)$
are bounded. The set $\{S^h(N)\colon N\in\N\}$, shown in
Figure~\ref{fig:rauzy-fractal-12}, is a well-known construction of the Rauzy
fractal associated with a substitution \cite{MR667748,MR1020484,MR2721985,MR2759108}.
Theorem~\ref{maintheorem:combinatoire almost always}
implies that the Rauzy fractal is bounded for almost every $\C$-adic word.
As shown recently, this is not true for all $\C$-adic words \cite{CRMATH_2021__359_4_399_0}.

\begin{figure}[h]
    \begin{center}
        \includegraphics[width=.45\linewidth]{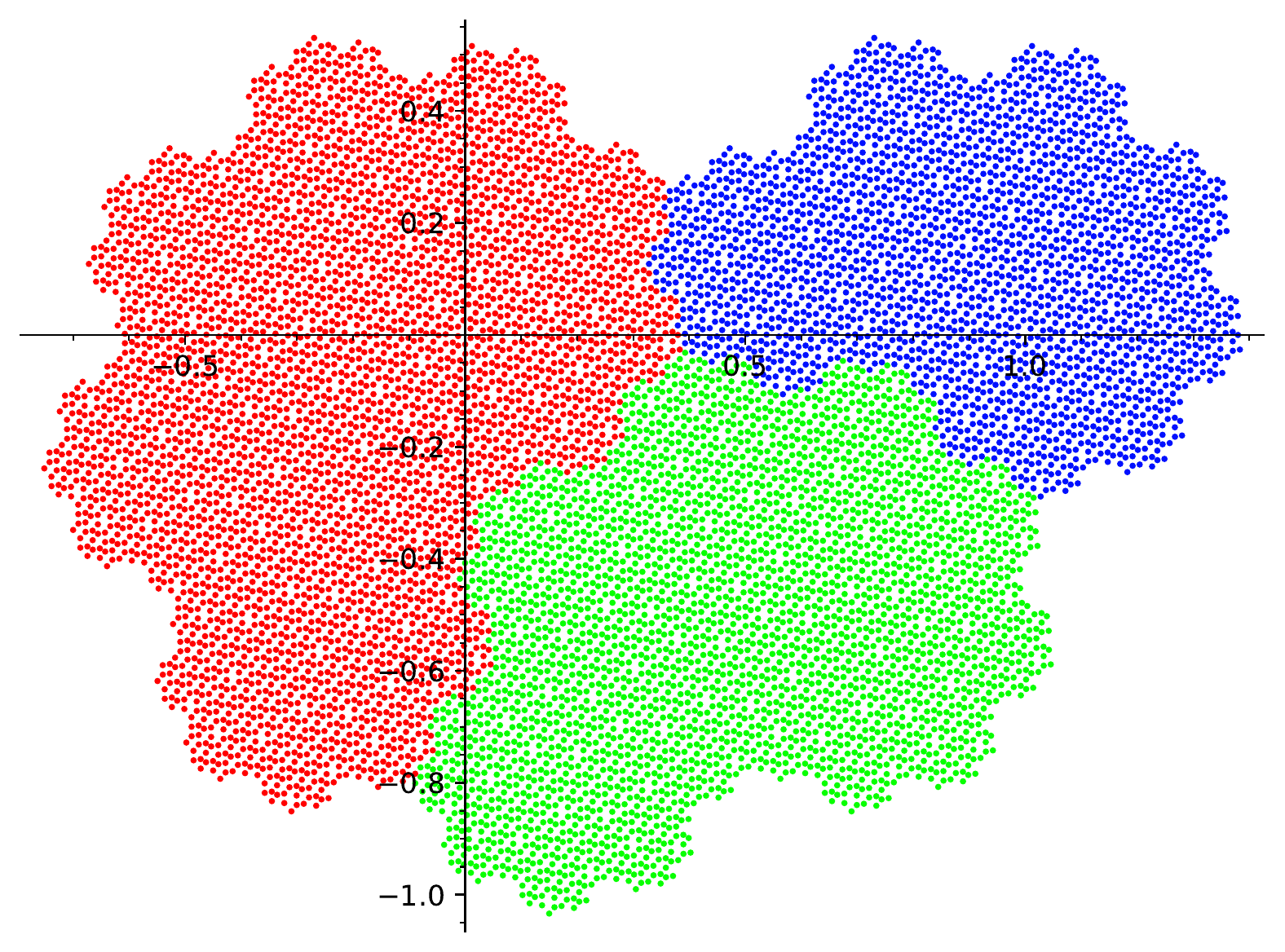}
        \includegraphics[width=.45\linewidth]{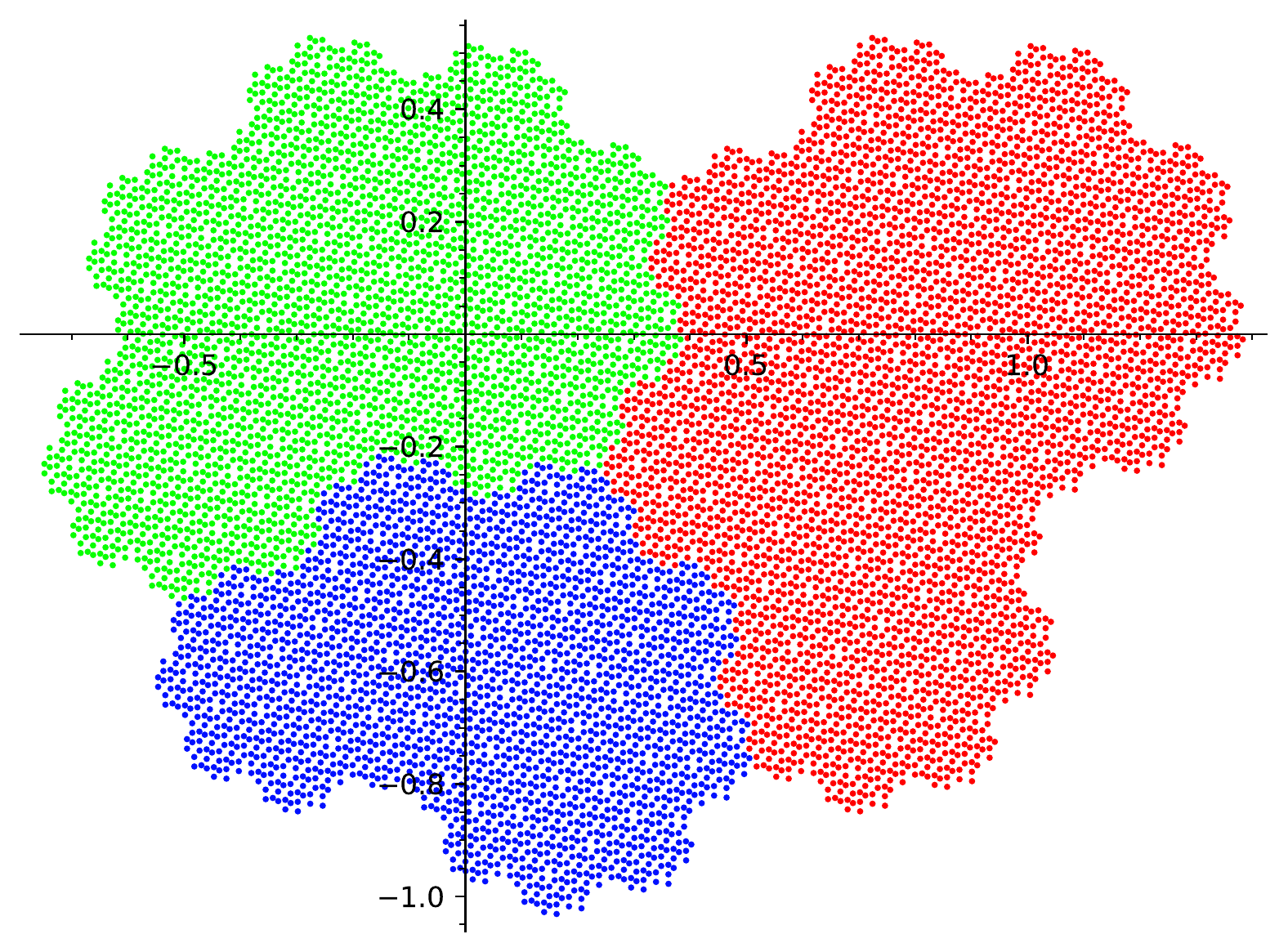}
    \end{center}
    \caption{The Rauzy fractal associated with the fixed point $\bw$ of $c_1c_2$.
            On the left (right resp.) the color of the point $S^h(N)\in\mathbb{C}$ is chosen
            according to the letter $w_N$ ($w_{N-1}$ resp.).}
\label{fig:rauzy-fractal-12}
\end{figure}

Figure~\ref{fig:rauzy-fractal-12} can be reproduced in SageMath
in few lines:
{\footnotesize
\begin{verbatim}
    sage: c1 = WordMorphism("1->1,2->13,3->2")
    sage: c2 = WordMorphism("1->2,2->13,3->3")
    sage: c12 = c1*c2
    sage: c12.rauzy_fractal_plot()
    sage: c12.rauzy_fractal_plot(exchange=True)
\end{verbatim}}
\noindent
In Figure~\ref{fig:rauzy-fractal-12}, we observe that the fractal can
be decomposed into three parts in two distinct ways, which defines an exchange
of pieces inside the fractal. 
Theorem~\ref{maintheorem:combinatoire almost always}
has important consequences.
Recent progresses \cite{berthe_multidimensional_2020,fogg_symbolic_2020}, which
build on our preliminary work \cite{MR3703620}, prove that the exchange
of pieces is almost surely equivalent to a rotation on a two-dimensional torus, and more
importantly that almost every rotation on the 2-dimensional torus admits a
coding of complexity $2n+1$ through such a fractal partition of the 2-torus.

\subsection*{Comparison with other generalizations of Sturmian words over larger alphabets}
There exist many other generalizations of Sturmian words over larger alphabets, each focusing on particular properties satisfied by Sturmian words. 

Words of complexity $2n+1$ were for instance considered by Arnoux and Rauzy~\cite{MR1116845}
with the condition that, like Sturmian words, there is exactly one left and one right special factor of each length; 
these words are now called Arnoux-Rauzy words. 
It is known that the frequencies of any Arnoux-Rauzy word are well defined and belong to the Rauzy gasket~\cite{MR3184185}, a fractal set of Lebesgue measure zero.
Thus the above condition on the number of special factors is very restrictive for the possible letter frequencies.

Words of complexity $p(n)\leq 2n+1$ include Arnoux-Rauzy words, 
codings of interval exchange transformations and more~\cite{MR3214265}. 
For any given letter frequencies one can construct words of factor
complexity $2n+1$ by the coding of a 3-interval exchange transformation. It is however known that these words are almost always unbalanced~\cite{MR1488330}.

In recent years, multidimensional continued fraction algorithms were used to obtain ternary balanced words with low factor complexity for any given vector of letter frequencies.  
Indeed the Brun algorithm leads to balanced words~\cite{MR3124516} and it was shown that the Arnoux-Rauzy-Poincaré algorithm leads to words of factor complexity $p(n)\leq\frac{5}{2}n+1$~\cite{MR3283831}.

Thus the words that we consider in this paper provide the first class of words which simultaneously generalize the three Sturmian properties of having factor complexity $(\# A-1)n+1$, having any vector of rationally independent letter frequencies and being almost always balanced.
The problem of finding an analogue of $f_C$ in dimension $d\geq 4$, generating balanced $\Scal$-adic sequences with complexity $(d-1)n+1$ for almost every vector of letter frequencies, is still open.

\subsection*{Structure of the article}
In Section~\ref{sec:bidimensional CFA}, we define the MCFA used in this article as well as the associated matrices $C_1$ and $C_2$, the substitutions $c_1$ and $c_2$ and the adic words.

Since we are dealing with convergence of cones $C_{i_0}C_{i_1}\cdots C_{i_n} \R^3_{\geq0}$, an important part of the paper deals with products of matrices.
In Section~\ref{sec:seminorm}, we define a semi-norm $\|\cdot\|_D$ on $\mathbb{R}^d$ which is well-suited for the matrices $C_1$ and $C_2$ and, using it, we give sufficient conditions so that a sequence $(M_n)_{n \in \N}$ of non-negative $d\times d$ matrices is weakly convergent (Proposition~\ref{prop:convergence-sufficient-conditions}).
We then apply our results in Section~\ref{section:primitivity and convergence} to sequences $(C_{i_n})_{n \in \N} \in \{C_1,C_2\}^\N$ and show that any such sequence is weakly convergent (Proposition~\ref{prop:convergence-for-C1-C2}).
In particular, this defines the map $\pi$ of Equation~\eqref{eq:def-pi-C1-C2}.

In Section~\ref{sec:rational dependencies}, we characterize the rational dependencies of $\pi((i_n)_{n \in \N})$.
In particular, we show that $\pi(\mathcal{P}) = \mathcal{I}$ (Theorem~\ref{thm:dim-sur-Q}) and that the restriction of $\pi$ to $\mathcal{P}$ is a bijection (Corollaries~\ref{cor:nonuniqueimpliesnonprimitive} and~\ref{cor:l'algo rend la meme suite directrice}).
In particular, this implies Theorem~\ref{maintheorem:conjugacy to shift}, as detailed in Section~\ref{sec:symbolic repr}.

We show in Section~\ref{sec:word frequencies} that all $\C$-adic words have uniform word frequencies (Proposition~\ref{prop:unif frequencies}) and in Section~\ref{sec:balance} that almost all of them are balanced (part 1 of Theorem~\ref{maintheorem:combinatoire almost always}).
We show that the Lyapunov exponent is negative in Section~\ref{sec:lyapunov}, completing the proof of Theorem~\ref{maintheorem:combinatoire almost always}.
The factor complexity of $\C$-adic words is studied in Section~\ref{sec:factor complexity}, completing Theorem~\ref{maintheorem:caracterization 2n+1}.
The link with Selmer algorithm is studied in Section~\ref{sec:selmer}.

\subsection*{Acknowledgments}

We are thankful to Valérie Berthé for her enthusiasm toward this project
and Vincent Delecroix for helping discussions.
We also thank the referee for their thorough reading and pertinent suggestions
improving the quality of the article.

\section{A bidimensional continued fraction algorithm}
\label{sec:bidimensional CFA}
On $\mathbb{R}^3_{\geq 0}$, the bidimensional continued fraction algorithm
introduced by the first author~\cite{cassaigne_algorithme_2015} is
\[
F_C (x_1,x_2,x_3) = 
\begin{cases}
    (x_1-x_3, x_3, x_2), & \mbox{if } x_1 \geq x_3;\\
    (x_2, x_1, x_3-x_1), & \mbox{if } x_1 < x_3.
\end{cases}
\]
More information on multidimensional continued fraction algorithms
can be found in \cite{BRENTJES,schweiger}.

Alternatively, the map $F_C$ can be defined by associating nonnegative matrices
to each part of a partition of $\mathbb{R}^3_{\geq 0}$ into $\Lambda_1\cup\Lambda_2$ where
\begin{align*}
	\Lambda_1 &= \{(x_1,x_2,x_3)\in\mathbb{R}^3_{\geq 0} \mid 
	x_1 \geq x_3\}, \\
    \Lambda_2 &= \{(x_1,x_2,x_3)\in\mathbb{R}^3_{\geq 0} \mid 
	x_1 < x_3\}.
\end{align*}
The matrices are given by the rule
$\Msf(\bx) = C_i$
if and only if
$\bx\in\Lambda_i$ where
\[
C_{1}=\MONE
\qquad\text{and}\qquad
C_{2}=\MTWO.
\]
The map $F_C$ on $\mathbb{R}^3_{\geq 0}$ and
the projective map $f_C$ on
$\Delta=\{\bx\in\mathbb{R}^3_{\geq 0}\mid\Vert\bx\Vert_1=1\}$ are then defined as:
\[
    F_C(\bx) = \Msf(\bx)^{-1}\bx
    \qquad\text{and}\qquad
    f_C(\bx) = \frac{F_C(\bx)}{\Vert F_C(\bx)\Vert_1}.
\]
Thus, we have
\[
f_C (x_1,x_2,x_3) = 
\begin{cases}
    (\frac{x_1-x_3}{1-x_3}, \frac{x_3}{1-x_3}, \frac{x_2}{1-x_3}), & \mbox{if } x_1 \geq x_3;\\
    (\frac{x_2}{1-x_1}, \frac{x_1}{1-x_1}, \frac{x_3-x_1}{1-x_1}), & \mbox{if } x_1 < x_3.
\end{cases}
\]
Many of their properties can be found in \cite{labbe_3-dimensional_2015}.
Since $\{\Lambda_1,\Lambda_2\}$ is a partition of $\mathbb{R}^3_{\geq 0}$, any vector $\bx \in \mathbb{R}^3_{\geq 0}$ defines a sequence of matrices $(C_{i_n})_{n \in \N}$ by $C_{i_n} = \Msf(F_C^n(\bx))$ and we have
\begin{equation}
\label{eq:definition of sequence of matrices}
	\bx \in \bigcap_{n \geq 0} C_{i_0} C_{i_1} \cdots C_{i_n} \mathbb{R}^3_{\geq 0}.
\end{equation}
The $n$-cylinders induced by $f_C$ on $\Delta$ are illustrated in Figure~\ref{fig:cylinders}.

\begin{figure}
    \includegraphics[width=4cm]{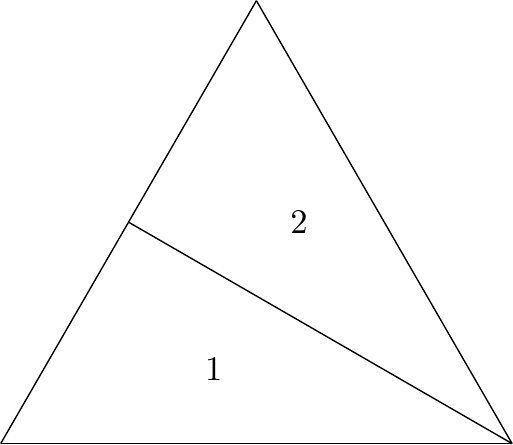}
    \includegraphics[width=4cm]{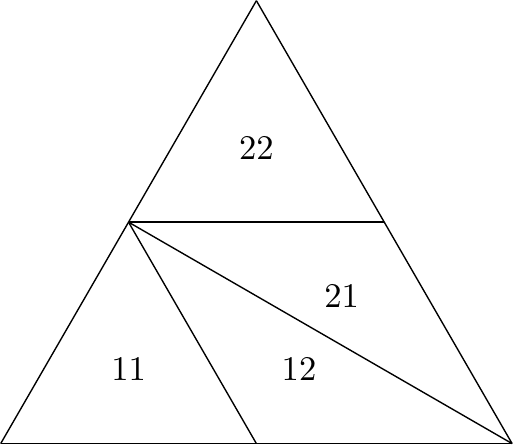}
    \includegraphics[width=4cm]{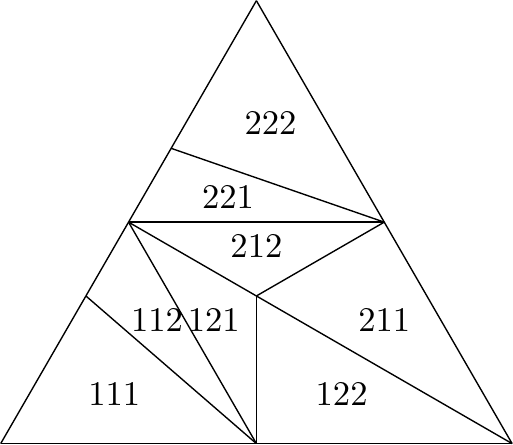}\\
    \includegraphics[width=4cm]{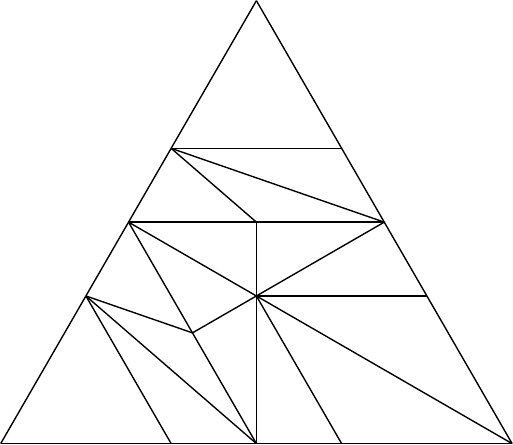}
    \includegraphics[width=4cm]{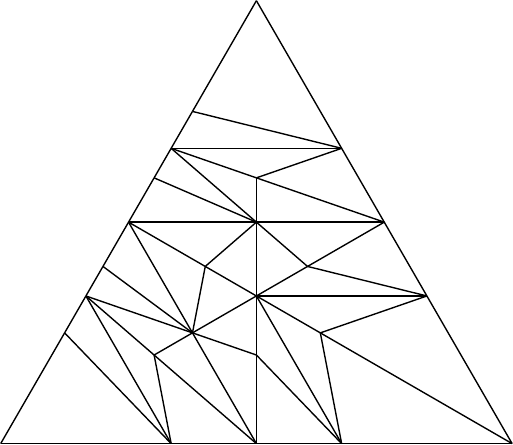}
    \includegraphics[width=4cm]{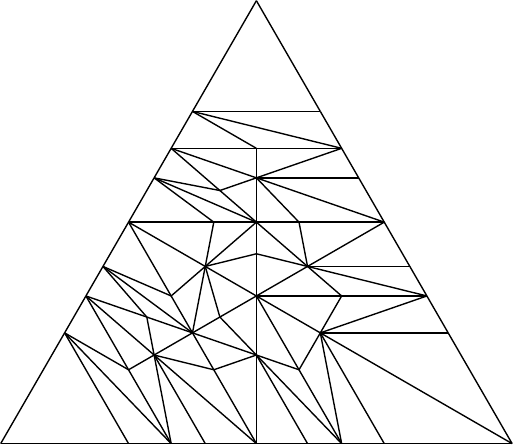}
    \caption{The $n$-cylinders of $f_C$ on $\Delta$ for each
    $n\in\{1,2,3,4,5,6\}$.
    Any $n$-cylinder is represented by a word $u_0 u_1 \cdots u_{n-1}$ over $\{1,2\}^*$ and is the set of points $x \in \Delta$ such that for all $k \in \{0,1,\dots,n-1\}$, $\Msf(f_C^k(x)) = C_{u_k}$.}
    \label{fig:cylinders}
\end{figure}

\subsection{Background on substitutions and $\Scal$-adic words}

Let $A$ be an alphabet, i.e., a finite set.
By {\em substitution} over $A$ we mean an endomorphism $\sigma$ of the free monoid $A^*$ which is non-erasing, i.e. $\sigma(a) \neq \varepsilon$ for all $a$, where $\varepsilon$ is the empty word.
If $\Scal$ is a set of substitutions over $A^*$,
a word $\bw \in A^\N$ is said to be {\em $\Scal$-adic} if there is a sequence 
$\boldsymbol{\sigma}=(\sigma_n)_{n \in \N} \in \Scal^\N$ and a sequence $\ba = (a_n)_{n \in \N} \in A^\N$ such that the limit $\lim_{n \to +\infty} \sigma_0 \sigma_1 \cdots \sigma_{n-1}(a_{n})$ exists and is equal to $\bw$.
The 2-tuple $(\boldsymbol{\sigma},\ba)$ is called an {\em $\Scal$-adic representation} of $\bw$ and the sequence $\boldsymbol{\sigma}$ a {\em directive sequence} of $\bw$.

A sequence of substitutions $(\sigma_n)_{n \in \N} \in \Scal^\N$ is said to be {\em everywhere growing} if $\min_{a \in A} |\sigma_{[0,n)}(a)|$ goes to infinity as $n$ goes to infinity.

With an substitution $\sigma:A^* \to A^*$, we associate its {\em incidence matrix} $M_\sigma \in \N^{A \times A}$ defined by $(M_\sigma)_{a,b} = |\sigma(b)|_a$.
Thus, for any word $w\in A^*$, we have $\overrightarrow{\sigma(w)} = M_\sigma \overrightarrow{w}$, where $\overrightarrow{w} \in \N^A$ is defined by $\overrightarrow{w}_a = |w|_a$.

\subsection{Substitutions and $\Scal$-adic words associated with the matrices $C_1$ and $C_2$}

We consider the alphabet $\A = \{1,2,3\}$ and the two substitutions
\[
c_{1}:
\begin{cases}
1 \mapsto 1 	\\
2 \mapsto 13	\\
3 \mapsto 2
\end{cases}
\qquad\text{and}\qquad
c_{2}:
\begin{cases}
1 \mapsto 2\\
2 \mapsto 13\\
3 \mapsto 3
\end{cases}
\]
and $\C$-adic words over the set $\C=\{c_1,c_2\}$.
One may check that $C_i$ is the incidence matrix of $c_i$ for $i=1,2$.
Note that the choice of the above substitutions $c_1$ and $c_2$ is less trivial than one
may first think.
Indeed, not all choices for the image of the letter $2$ allow the complexity
to be $2n+1$ and obtain Theorem~\ref{maintheorem:caracterization 2n+1}.
In particular, changing
$c_1$ to be
$1\mapsto1,2\mapsto31,3\mapsto2$ may seem interesting since it makes both $c_1$
and $c_2$ left-marked (the first letter of the images are all distinct),
but this choice does not work as it increases the factor complexity for the
associated $\C$-adic words.

Like for matrices, any vector $\bx \in \R^3_{\geq 0}$ defines a sequence of substitutions $(c_{i_n})_{n \in \N}$, where $c_{i_n} = \mapC(F_C^{n}(\bx))$ and $\mapC(\by) = c_i$ if and only if $\by \in \Lambda_i$. 
For example, using vector $\bx=\left(1, e, \pi\right)$, we have
\[
\mapC(\bx) \mapC(F_C\bx) \mapC(F_C^2\bx) \mapC(F_C^3\bx) \mapC(F_C^4\bx)
    = c_{2} c_{1} c_{2} c_{1} c_{1}
    = \begin{cases}
	1 \mapsto 23\\
	2 \mapsto 23213\\
	3 \mapsto 2313
    \end{cases}
\]
whose incidence matrix is $C_{2} C_{1} C_{2} C_{1} C_{1}$. 

The next lemma shows that not every 2-tuple $((\sigma_n)_{n \in \N},(a_n)_{n \in \N}) \in \C^\N \times \A^\N$ can be a $\C$-adic representation of a word.
In what follows, we use the notations
$\sigma_{[m,n]}=\sigma_{m} \sigma_{m+1} \cdots \sigma_{n}$
and
$\sigma_{[m,n)}=\sigma_m \sigma_{m+1} \cdots \sigma_{n-1}$ when $m\leq n$.

\begin{lemma}
\label{lemma:existence of limit}
For every directive sequence $\boldsymbol{\sigma} = (\sigma_n)_{n \in \N} \in \C^\N$, there exists a sequence of letters $(a_n)_{n \in \N} \in \A^\N$
such that
$\bw = \lim_{n \to +\infty} \sigma_{[0,n)}(a_n)$ exists
    and is an infinite word. Moreover, $\bw$ is independent of the choice of
    $(a_n)_{n\in\N}$.
More precisely,
\begin{enumerate}
\item
    If $\boldsymbol{\sigma}$ contains infinitely many occurrences of both $c_1$ and $c_2$, then the limit exists and $\bw = \lim_{n \to +\infty} \sigma_{[0,n)}(1)$. 
\item
If there is some integer $N$ such that $\sigma_n = c_1$ for all $n \geq N$, then the limit exists and is an infinite word if and only if $(a_n)_{n \in \N} \in \A^*\{2,3\}^\N$.
In that case, we have $\bw = (\sigma_{[0,N)}(1))^\omega = \sigma_{[0,n)}(1^\omega)$.
\item
If there is some integer $N$ such that $\sigma_n = c_2$ for all $n \geq N$, then the limit exists and is an infinite word if and only if there is some integer $N' \geq N$ such that $(a_{N'+2n},a_{N'+2n+1}) = (1,2)$ for all $n$.
In that case, we have $\bw = \sigma_{[0,N')}(13^\omega)$.
\end{enumerate}
\end{lemma}

\begin{proof}
    For all $m$, we set $w_m = \sigma_{[0,m)}(a_m)$.
We also let $p_m$ denote the longest prefix of $w_m$ which is a prefix of $w_n$ for all $n \geq m$.
The limit $\lim_{n \to +\infty} w_n$ exists and is an infinite word if and only if the length of $p_n$ tends to infinity as $n$ increases.
Furthermore, in that case $\lim_{n \to +\infty} w_n = \lim_{n \to +\infty} p_n$.

Let us prove (1).
Since both $c_1$ and $c_2$ occur infinitely many times in $(\sigma_n)_{n\in \N}$, there is a sequence of integers $(l_m)_{m \in \N}$ such that $(\sigma_{n})_{n \geq l_m}$ has a prefix of the form $c_1 c_2^k c_1$ for some $k \geq 1$ and $l_{m+1} \geq l_m+k+2$.
Furthermore, for all $k \geq 1$ and all $a \in \A$, $1$ is a prefix of $c_1 c_2^k c_1(a)$.
    Thus for all $n \geq l_{m+1}$ and all $a \in \A$, $1$ is a prefix of $\sigma_{[l_m,n)}(a)$, hence $\sigma_{[0,l_m)}(1)$ is a prefix of 
    $\sigma_{[0,n)}(a)$.
    As the length of $\sigma_{[0,l_m)}(1)$ tends to infinity as $n$ increases, this shows that 
\[
    \lim_{n \to +\infty} \sigma_{[0,n)}(a_n)
	=
	\lim_{n \to +\infty} \sigma_{[0,n)}(1)
\]
for all sequences $(a_n)_{n \in \N}$, which ends the proof.

Let us prove (2).
    As $c_1(1) = 1$, the sequence of letters $(a_n)_{n \in \N}$ cannot contain infinitely many ones, otherwise the sequence $(\sigma_{[0,n)}(a_n))_{n \in \N} \in (\A^*)^\N$ would have a constant subsequence and the limit, if it exists, would be a finite word.
Thus the sequence $(a_n)_{n \in \N}$ has to be in $\mathcal{A}^* \{2,3\}^\mathbb{N}$.
As for all $m$ and all $n \geq 2m$, $1^m$ is a proper prefix of both $c_1^n(2)$ and $c_1^n(3)$, the limit $\lim_{n \to +\infty} \sigma_{[0,n)}(a_n)$ is the periodic word $(\sigma_{[0,N)}(1))^\omega$.
The proof of (3) is obtained in a similar way.
\end{proof}

The next result is a direct consequence of Lemma~\ref{lemma:existence of limit}.
One could actually show that the converse also holds.

\begin{corollary}
\label{cor:everywhere growing}
If a $\mathcal{C}$-adic word $\bw$ is aperiodic, then it admits an everywhere growing directive sequence $(\sigma_n)_{n \in \mathbb{N}} \in \mathcal{C}^\mathbb{N}$.
\end{corollary}

By Lemma~\ref{lemma:existence of limit}, when the sequence $(c_{i_n})_{n \in \N} = (\mapC(F_C^n\bx))_{n \in \mathbb{N}}$ contains infinitely many occurrences of $c_1$ and $c_2$, it defines a unique $\mathcal{C}$-adic word
\[
    W(\bx) = \lim_{n\to\infty} c_{i_0} c_{i_1} \cdots c_{i_n} (1).
\]

For example, using vector $\bx=\left(1, e, \pi\right)$, it is a consequence of Proposition~\ref{prop:primitiveness for C adic}, Theorem~\ref{thm:dim-sur-Q} and Corollary~\ref{cor:l'algo rend la meme suite directrice} that the sequence $(\mapC(F_C^n\bx))_{n \in \mathbb{N}}$ contains infinitely many occurrences of $c_1$ and $c_2$ and the associated infinite $\mathcal{C}$-adic word is
\[
    W(\bx) 
    = 2323213232323132323213232321323231323232 \cdots.
\]

\section{Semi-norm of matrices and convergence}
\label{sec:seminorm}

Equation~\eqref{eq:definition of sequence of matrices} shows that the iteration of the map $F_C$ on $\bx \in \mathbb{R}_{\geq 0}^3$ defines a sequence of matrices $(M_n)_{n \geq 0} \in \{C_1,C_2\}^\N$ such that
\[
	\bx \in \bigcap_{n \geq 0} M_{[0,n)} \mathbb{R}_{\geq 0}^3.
\]
In this section, we give sufficient conditions for a sequence of
$d$-dimensional nonnegative matrices $(M_n)_{n \in \N}$ to be {\em weakly
convergent}, i.e., to be such that the cone
\[	
	\bigcap_{n \geq 0} M_{[0,n)} \R_{\geq 0}^d
\]
is one-dimensional.

The following result states that the notion of weak convergence is related to the existence of (uniform) frequencies in $\Scal$-adic words.
Let $\bw = (w_n)_{n \in \N} \in A^\N$ be an infinite word and let $u \in A^*$ be a word occurring in $\bw$.
The {\em frequency} of $u$ in $\bw$ is the limit, whenever it exists, $\lim_{n \to \infty} \frac{|w_{[0,n)}|_u}{n}$, where $w_{[m,n)} = w_m w_{m+1} \cdots w_{n-1}$ and $|v|_u$ stands for the number of occurrences of $u$ in the word $v$.
The word $\bw$ has {\em uniform word frequencies} if for every $u \in A^*$, the ratio $\frac{|w_{[k,k+n)}|_u}{n}$ converges when $n$ goes to infinity, uniformly in $k$.

\begin{theorem}{\rm\cite{MR3330561}}\label{thm:berthe-delecroix_convergence}
    Let $A$ be an alphabet of size $d$.
Let $\bw \in A^\N$ be a word that admits an everywhere growing directive sequence $(\sigma_n)_{n \in \N}$ and let $(M_{n})_{n \in \N}$ be the associated sequence of incidence matrices.
If for all $k \in \N$, the cone
\begin{equation} \label{eq:cone converge}
	\bigcap_{n \geq k} M_{[k,n)} \R_{\geq 0}^d
\end{equation}
is one-dimensional, then $\bw$ has uniform word frequencies.
In particular, if $\bf \in \R_{\geq 0}^d$ is such that $\|\bf\|_1 = 1$ and 
\begin{equation}
	\bigcap_{n \geq 0} M_{[0,n)} \R_{\geq 0}^d = \R_{\geq 0} \bf,
\end{equation}
then $\bf$ is the vector of letter frequencies of $\bw$.
\end{theorem}

\subsection{A semimetric on the projective space}

Recall that the {\em Hilbert metric} is defined as
\begin{equation*}
    d_H(\R_{>0} \bv, \R_{>0}\bw) = \max_{1\leq i,j\leq d} \log\frac{v_iw_j}{v_jw_i}
\end{equation*}
where $\bv=(v_1,\ldots,v_d)\in\R^d_{>0}$ and $\bw=(w_1,\ldots,w_d)\in\R^d_{>0}$.
Here, we define another closely related function as
\begin{equation}
    d_M(\R_{>0} \bv, \R_{>0}\bw) = \frac{1}{\Vert\bv\Vert_2\cdot \Vert\bw\Vert_2}
                           \cdot\max_{1\leq i,j\leq d} |v_iw_j-v_jw_i|
\end{equation}
where $\bv=(v_1,\ldots,v_d)\in\R^d_{\geq 0}\setminus\{0\}$ and
$\bw=(w_1,\ldots,w_d)\in\R^d_{\geq 0}\setminus\{0\}$.
It is not a distance as it does not satisfy the triangle inequality, but it is
a semimetric, that is, it satisfies the first three axioms of a distance as
shown below.

\begin{lemma}\label{lem:semimetric}
    $d_M$ is a semimetric, i.e.,
\begin{enumerate}[\rm (i)]
    \item $d_M(\R_{>0} \bv, \R_{>0}\bw) \geq 0$,
    \item $d_M(\R_{>0} \bv, \R_{>0}\bw) = 0$ if and only if $\R_{>0} \bv=\R_{>0}\bw$,
    \item $d_M(\R_{>0} \bv, \R_{>0}\bw)=d_M(\R_{>0} \bw, \R_{>0}\bv)$.
\end{enumerate}
\end{lemma}

\begin{proof}
Let $\bv=(v_1,\ldots,v_d)\in\R^d_{\geq 0}\setminus\{0\}$ and $\bw=(w_1,\ldots,w_d)\in\R^d_{\geq 0}\setminus\{0\}$.

    (i)
    We have $d_M(\R_{>0} \bv, \R_{>0}\bw) \geq 0$ by definition.

    (ii)
    If $\R_{>0} \bv=\R_{>0}\bw$, then there exists $k>0$ such that
    $\bw=k\bv$.
    Then
    \[
       \max_{1\leq i,j\leq d} |v_iw_j-v_jw_i|
       = \max_{1\leq i,j\leq d} |v_i(kv_j)-v_j(kv_i)|
       = k \max_{1\leq i,j\leq d} |v_iv_j-v_jv_i|
       = 0.
    \]
    Thus $d_M(\R_{>0} \bv, \R_{>0}\bw) = 0$.
    Conversely, if $d_M(\R_{>0} \bv, \R_{>0}\bw) = 0$, then
    \[
        \max_{1\leq i,j\leq d} |v_iw_j-v_jw_i|
        = {\Vert\bv\Vert_2\cdot \Vert\bw\Vert_2} \cdot
          d_M(\R_{>0} \bv, \R_{>0}\bw) 
        = 0.
    \]
Therefore, for every $i,j$ such that
    $1\leq i,j\leq d$, we have $v_iw_j=v_jw_i$. 
Choose $j$ such that $w_j \neq 0$ and set $k = \frac{v_j}{w_j}$.
Thus, for all $i$, we have $v_i=  k w_i$.
As $\bv \neq 0$, we have $k > 0$ and we conclude that $\R_{>0} \bv=\R_{>0}\bw$.

    (iii) We have
    \begin{align*}
        d_M(\R_{>0} \bv, \R_{>0}\bw) 
        &= \frac{1}{\Vert\bv\Vert_2\cdot \Vert\bw\Vert_2}
                           \cdot\max_{1\leq i,j\leq d} |v_iw_j-v_jw_i|\\
        &= \frac{1}{\Vert\bw\Vert_2\cdot \Vert\bv\Vert_2}
                           \cdot\max_{1\leq i,j\leq d} |w_iv_j-w_jv_i|
         = d_M(\R_{>0} \bw, \R_{>0}\bv).
    \end{align*}
\end{proof}

Using the semimetric $d_M$, we define the \emph{diameter} of a cone $\Lambda \subseteq \R_{\geq 0}^d$ as
\begin{equation}
        \mathrm{diam}(\Lambda) = \sup_{\bv,\bw\in \Lambda\setminus \{0\}} d_M(\R_{>0}\bv,\R_{>0}\bw).
\end{equation}
The fact that the diameter is defined from a semimetric is enough for our
needs since the following lemma proves that a cone of diameter zero is reduced to a single line.

\begin{lemma}\label{lem:diam-0=line}
Let $\Lambda\subseteq \R_{\geq 0}^d$ be a cone. 
If $\mathrm{diam}(\Lambda)=0$, then
there exists $\bu\in\R_{\geq 0}^d\setminus\{0\}$ satisfying $\Lambda=\R_{\geq 0}\bu$.
\end{lemma}

\begin{proof}
    Let $\bu\in \Lambda\setminus\{0\}$.
    By definition, we have $\R_{\geq 0}\bu\subseteq \Lambda$.
    Now let $\bv\in \Lambda \setminus \{0\}$.
    Since $\mathrm{diam}(\Lambda)=0$, we have $d_M(\R_{>0}\bu,\R_{>0}\bv)=0$.
    From Lemma~\ref{lem:semimetric} (ii), there exists $k\in\R_{>0}$ such that
    $\bv=k\bu$. Therefore $\bv\in\R_{> 0}\bu$.
    We have proved 
    $\Lambda\subseteq\R_{\geq 0}\bu$.
\end{proof}

Our aim is now to study the diameter of $\bigcap_{n \geq 0} M_{[0,n)} \R_{\geq 0}^d$ for a given sequence of matrices $(M_n)_{n \in \N}$.
To that aim, we provide an upper bound for the diameter of a cone defined by the image of the nonnegative orthant under the application of a nonnegative matrix.
It is defined in terms of the entries of the matrix and in terms of a matrix semi-norm that we define below.

If $V$ is a non-trivial vector subspace of $\R^d$ and $\Vert\cdot\Vert$ is a 
semi-norm on $\R^d$ which is a norm on $V$,
then the matrix semi-norm $\Vert\cdot|_V\Vert$ is defined for any $d\times d$
matrix $M$ as
\begin{equation}
    \left\Vert M\middle|_V\right\Vert := \sup_{v\in V\setminus\{0\}}\frac{\Vert M v\Vert}{\Vert v\Vert}.
\end{equation}
For any vector $\bf \in \mathbb{R}^d_{\geq 0} \setminus\{0\}$, $\bf^\perp$ stands for the vector space of codimension 1 orthogonal to $\bf$.

Let $\Lambda \subset\R^d_{\geq0}$ be some cone.
As done in \cite{MR4043208},
if $\Vert\cdot\Vert$ is a norm on $\bf^\perp$ for all $\bf \in \Lambda \setminus\{0\}$,
we define a matrix semi-norm 
on $\R^{d\times d}$ as
\begin{equation}\label{eq:semi-norm-perp-to-cone}
    \left\Vert M\right\Vert^{\Lambda}
    = \sup_{\bf\in \Lambda\setminus\{0\}}\left\Vert M\middle|_{\bf^\perp}\right\Vert.
\end{equation}
Thus, we have
\[
    \left\Vert M\right\Vert^{\Lambda}
    = \sup_{\bf\in \Lambda\setminus\{0\}}
	\sup_{v \in \bf^\perp\setminus\{0\}}    
    \frac{\left\Vert Mv\right\Vert}{\left\Vert v\right\Vert}.
\]
The diameter of the image of the nonnegative orthant under a positive matrix can be bounded
by the semi-norm of its transpose matrix.

\begin{lemma}\label{lem:diam-upper-bound}
    Let $A=(a_{ij})\in\R_{>0}^{d\times d}$ be a positive matrix.
    Then 
    \begin{equation}
        \mathrm{diam}(A\,\R^d_{\geq0}) 
        \leq 
           \frac{1}{\min_{1\leq i,j\leq d} a_{ij}}
           \cdot
           \left\Vert\transpose{A}\right\Vert_2^{A\R^d_{\geq0}}.
    \end{equation}
\end{lemma}

\begin{proof}
    From the definition of $d_M$ and of the diameter of a cone, and using the fact that for all $v,w \in \mathbb{R}^d_{\geq 0}$, $|\transpose{w}v| \leq \Vert w \Vert_2 \Vert v \Vert_2$, we compute
    \begin{align*}
        \mathrm{diam}(A\R^d_{\geq0}) 
        &= \sup_{v,w\in \R^d_{\geq0}\setminus\{0\}} d_M(\R_{>0} Av, \R_{>0} Aw) \\
        &= \sup_{v,w\in \R^d_{\geq0}\setminus\{0\}} \frac{1}{\Vert Av\Vert_2\Vert Aw\Vert_2}
           \max_{1\leq i,j\leq d} 
           \left|(\transpose{e_i}  Av)(\transpose{e_j}  Aw) - (\transpose{e_j}  Av)(\transpose{e_i}  Aw)\right|\\
        &= \sup_{v,w\in \R^d_{\geq0}\setminus\{0\}} \frac{1}{\Vert Av\Vert_2\Vert Aw\Vert_2}
           \max_{1\leq i,j\leq d} 
           \left|(\transpose{e_i}  Av)(\transpose{w}  \transpose{A}  e_j) - (\transpose{e_j}  Av)(\transpose{w}  \transpose{A}  e_i)\right|\\
        &= \sup_{v,w\in \R^d_{\geq0}\setminus\{0\}} \frac{1}{\Vert Av\Vert_2\Vert Aw\Vert_2}
           \max_{1\leq i,j\leq d} 
           \left|\transpose{w}  \transpose{A}  \left((\transpose{e_i}  Av)e_j - (\transpose{e_j}  Av)e_i\right)\right|\\
        &\leq \sup_{v,w\in \R^d_{\geq0}\setminus\{0\}} \frac{1}{\Vert Av\Vert_2\Vert Aw\Vert_2}
           \cdot\left\Vert w  \right\Vert_2\cdot
           \max_{1\leq i,j\leq d} 
           \left\Vert \transpose{A}  \left((\transpose{e_i}  Av)e_j - (\transpose{e_j}  Av)e_i\right)\right\Vert_2\\
        &=
        \sup_{w\in \R^d_{\geq0}\setminus\{0\}} 
           \frac{\left\Vert w \right\Vert_2}{\Vert Aw\Vert_2}
           \cdot
           \sup_{v\in \R^d_{\geq0}\setminus\{0\}} 
           \frac{1}{\Vert Av\Vert_2}
           \cdot
           \max_{1\leq i,j\leq d} 
           \left\Vert \transpose{A}  \left((\transpose{e_i}  Av)e_j - (\transpose{e_j}  Av)e_i\right)\right\Vert_2.
\end{align*}
Now observe that
\[
\sup_{w\in \R^d_{\geq0}\setminus\{0\}} 
\dfrac{\left\Vert w \right\Vert_2}{\Vert Aw\Vert_2}
\leq \dfrac{1}{\min_{1\leq i,j\leq d} a_{ij}}.
\]
Furthermore, for all $v \in \mathbb{R}^d_{\geq 0} \setminus \{0\}$, we have $\bf = Av \in A \mathbb{R}^d_{\geq 0}\setminus\{0\}$ and, for all $1 \leq i,j \leq d$, 
\[
    (\transpose{e_i}  Av)e_j - (\transpose{e_j}  Av)e_i \in \bf^\perp \setminus\{0\},
\]
and, as a consequence,
\[
    \left\Vert \transpose{A}  \left((\transpose{e_i}  Av)e_j - (\transpose{e_j}  Av)e_i\right)\right\Vert_2
	\leq 
    \left\Vert \transpose{A} \right\Vert_2^{A \mathbb{R}^d_{\geq 0}} 
    \left\Vert (\transpose{e_i}  Av)e_j - (\transpose{e_j}  Av)e_i\right\Vert_2.
\]
We finally get
\begin{align*}
        \mathrm{diam}(A\R^d_{\geq0}) 
        &\leq 
           \frac{1}{\min_{1\leq i,j\leq d} a_{ij}}
           \cdot
           \sup_{v\in \R^d_{\geq0}\setminus\{0\}} 
           \frac{1}{\Vert Av\Vert_2}
           \cdot
           \left\Vert \transpose{A} \right\Vert_2^{A\R^d_{\geq0}}
           \cdot
           \max_{1\leq i,j\leq d} 
           \left\Vert \left((\transpose{e_i}  Av)e_j - (\transpose{e_j}  Av)e_i\right)\right\Vert_2\\
        &\leq 
           \frac{1}{\min_{1\leq i,j\leq d} a_{ij}}
           \cdot
           \sup_{v\in \R^d_{\geq0}\setminus\{0\}} 
           \frac{1}{\Vert Av\Vert_2}
           \cdot
           \left\Vert \transpose{A} \right\Vert_2^{A\R^d_{\geq0}}
           \cdot
           \left\Vert Av\right\Vert_2\\
        & = 
           \frac{1}{\min_{1\leq i,j\leq d}a_{ij}}
           \cdot
           \left\Vert \transpose{A} \right\Vert_2^{A\R^d_{\geq0}}.
    \end{align*}
\end{proof}

The next lemma shows that the matrix semi-norm defined in Equation~\eqref{eq:semi-norm-perp-to-cone} behaves well with respect
to product of matrices.

\begin{lemma}\label{lem:semi-norm-on-AB}
    Let $A,B\in\R^{d\times d}_{\geq0}$ such that $AB \neq 0$ and
    let $\Vert\cdot\Vert$ be any semi-norm on $\R^d$ which is a norm on every
    $\bf^\perp$ with $\bf\in (A \R^d_{\geq0} \cup B \R^d_{\geq0} )\setminus\{0\}$.
    We have
    \[
        \left\Vert\transpose{(AB)} \right\Vert^{AB\R^d_{\geq0}}
        \leq
        \left\Vert \transpose{B} \right\Vert^{B\R^d_{\geq0}}
        \left\Vert \transpose{A} \right\Vert^{A\R^d_{\geq0}}.
    \]
\end{lemma}

\begin{proof}
    We have
\[
    \left\Vert\transpose{(AB)} \right\Vert^{AB\R^d_{\geq0}}
        = \sup_{\bf\in AB\R^d_{\geq0}\setminus\{0\}, z\in \bf^\perp\setminus\{0\}}
        \frac{\Vert\transpose{(AB)}  z\Vert}{\Vert z\Vert}.
\]
If for all $\bf\in AB\R^d_{\geq0}\setminus\{0\}$ and $z\in \bf^\perp\setminus\{0\}$,
    we have $\transpose{A} z = 0$, then 
    $\left\Vert\transpose{(AB)} \right\Vert^{AB\R^d_{\geq0}} = 0$
and the result follows.
Otherwise, we have
\[
    \left\Vert\transpose{(AB)} \right\Vert^{AB\R^d_{\geq0}}
        = \sup_{\bf\in AB\R^d_{\geq0}\setminus\{0\}, 
                z\in \bf^\perp\setminus\{0\}, 
                \transpose{A} z\neq0}
            \frac{\Vert\transpose{(AB)}  z\Vert}{\Vert z\Vert}.
\]
Observe that for all $\bf \in AB\R^d_{\geq0}\setminus\{0\}$ and all $z\in \bf^\perp\setminus\{0\}$, we have $\transpose{A} z \in \mathbf{g}^\perp$ for some $\mathbf{g} \in B\R^d_{\geq0}\setminus\{0\}$.
Since $\Vert\cdot\Vert$ is a norm on every
    $\mathbf{g}^\perp$ with $\mathbf{g}\in B \R^d_{\geq0} \setminus\{0\}$, we have $\left\Vert \transpose{A} z \right\Vert \neq 0$ for all $\bf\in AB\R^d_{\geq0}\setminus\{0\}, z\in \bf^\perp\setminus\{0\}$ such that $\transpose{A} z\neq0$. 
Therefore, we get
    \begin{align*}
        \left\Vert\transpose{(AB)} \right\Vert^{AB\R^d_{\geq0}}
        &= \sup_{\bf\in AB\R^d_{\geq0}\setminus\{0\}, z\in \bf^\perp\setminus\{0\}, \transpose{A} z\neq0}
        \frac{\Vert \transpose{B} \transpose{A}  z\Vert}{\Vert \transpose{A}  z\Vert}
        \cdot
        \frac{\Vert \transpose{A}  z\Vert}{\Vert z\Vert}\\
        &\leq 
        \sup_{\bf\in AB\R^d_{\geq0}\setminus\{0\}, z\in \bf^\perp\setminus\{0\}, \transpose{A} z\neq0}
        \frac{\Vert \transpose{B} \transpose{A}  z\Vert}{\Vert \transpose{A}  z\Vert}
        \cdot
        \sup_{\bf\in AB\R^d_{\geq0}\setminus\{0\}, z\in \bf^\perp\setminus\{0\}, \transpose{A} z\neq0}
        \frac{\Vert \transpose{A}  z\Vert}{\Vert z\Vert}\\
        &\leq 
        \sup_{\mathbf{g}\in B\R^d_{\geq0}\setminus\{0\}, z'\in \mathbf{g}^\perp\setminus\{0\}}
        \frac{\Vert \transpose{B}  z'\Vert}{\Vert z'\Vert}
        \cdot
        \sup_{\bf\in A\R^d_{\geq0}\setminus\{0\}, z\in \bf^\perp\setminus\{0\}}
        \frac{\Vert \transpose{A}  z\Vert}{\Vert z\Vert}\\
        &=
        \left\Vert \transpose{B} \right\Vert^{B\R^d_{\geq0}}
        \cdot
        \left\Vert \transpose{A} \right\Vert^{A\R^d_{\geq0}}
    \end{align*}
    where we substituted $z'=\transpose{A}  z$, $\bf=A\mathbf{g}$ and $\mathbf{g}=Bv$ for some $v\in\R^d_{\geq0}$
    since there exists $v\in\R^d_{\geq0}$ such that $\bf = ABv$. 
\end{proof}

\subsection{Primitive sequences and convergence}

Recall that a square matrix $M$ is {\em primitive} if there is some positive integer $k$ such that $M^k$ has only positive entries.
As an analogue, if $\bm = (M_n)_{n \in \N}$  is a sequence of nonnegative square matrices of the same size $d$, we say that $\bm$ is {\em primitive} if for all $r \in \N$, there exists $s > r$ such that $M_{[r,s)}$ has only positive entries.

If $\bm$ is the sequence of incidence matrices associated with a sequence of endomorphisms $(\sigma_n)_{n \in \mathbb{N}}$ of $A^*$, then primitivity of $\bm$ means that for all $r \in \N$, there exists $s > r$ such that for all letters $a,b \in A$, $a$ occurs in $\sigma_{[r,s)}(b)$.

The next result provides sufficient conditions for having weak convergence of a
primitive sequence of matrices without any recurrence hypothesis like in
Proposition 3.5.5 of \cite{akiyama_cirm2017_book_2020}.

\begin{proposition}\label{prop:convergence-sufficient-conditions}
Let $\bm = (M_n)_{n \in \N}$ be a primitive sequence
of nonnegative integer matrices. 
If there exists $K>0$ such that
    $\left\Vert \transpose{M_{[0,n)}} \right\Vert^{M_{[0,n)}\R^d_{\geq0}}\leq K$
               for some norm and for infinitely many $n\in\N$,
    then there exists a vector $\bu\in\R^d_{\geq0} \setminus \{\boldsymbol{0}\}$ satisfying
    \begin{equation}
        \bigcap_{n\geq0} M_{[0,n)} \R^d_{\geq0} = \R_{\geq 0}\bu.
    \end{equation}
\end{proposition}

\begin{proof}
    Since $\bm$ is primitive, there exists an increasing sequence
    $(n_k)_{k\in\N}$ such that $n_0 = 0$ and $M_{[n_k,n_{k+1})}$ has only positive entries for
    every $k\in\N$. 
    Furthermore, the sequence $(n_k)_{k \in \mathbb{N}}$ can be chosen among the indices $n$ for which 
    $\left\Vert \transpose{M_{[0,n)}} \right\Vert^{M_{[0,n)}\R^d_{\geq0}}\leq K$.
    If $A^{(k)}=(a^{(k)}_{ij})=M_{[0,n_k)}$, then for $k \geq 1$, one has
    \[
        \min_{1\leq i,j\leq d} a^{(k)}_{ij} \geq d^{k-1}.
    \]
    Moreover, as the dimension is finite, the chosen norm is equivalent to the
    $2$-norm, which implies that there exists a constant $K'$ such that
    \[
    \left\Vert \transpose{M_{[0,n)}} \right\Vert_2^{M_{[0,n)}\R^d_{\geq0}}
               \leq K'\cdot
    \left\Vert \transpose{M_{[0,n)}} \right\Vert^{M_{[0,n)}\R^d_{\geq0}}.
    \]
    Therefore, as $M_{[0,n_{k-1})}\,\R^d_{\geq0} \subset M_{[0,n)}\,\R^d_{\geq0} \subset M_{[0,n_k)}\,\R^d_{\geq0}$ whenever $n_{k-1} \leq n \leq n_k$, we get, using Lemma~\ref{lem:diam-upper-bound},
    \begin{align*}
        \lim_{n\to\infty}
        \mathrm{diam}(M_{[0,n)}\,\R^d_{\geq0}) 
        &=
        \lim_{k\to\infty}
        \mathrm{diam}(M_{[0,n_k)}\,\R^d_{\geq0}) \\
        &\leq
        \lim_{k\to\infty}
           \frac{1}{\min_{1\leq i,j\leq d} a^{(k)}_{ij}}
           \cdot
           \left\Vert \transpose{M_{[0,n_k)}}
           \right\Vert_2^{M_{[0,n_k)}\R^d_{\geq0}}
        \leq
        \lim_{k\to\infty}
           \frac{1}{d^{k-1}}
           \cdot K'\cdot K = 0.
    \end{align*}
    We conclude from Lemma~\ref{lem:diam-0=line} that the cone $\bigcap_{n\geq0} M_{[0,n)} \R^d_{\geq0}$ is one-dimensional.
\end{proof}

\subsection{A piecewise linear semi-norm}\label{sec:piecewise-linear-semi-norm}

Proposition~\ref{prop:convergence-sufficient-conditions} holds for any norm on $\R^d$.
In this section, we consider the following function $\Vert\cdot\Vert_D:\R^d\to\R$ and show that it is a semi-norm on $\R^d$ and a norm on well-chosen subspaces.
It is defined as
\begin{equation}\label{eq:semi-norm-D}
    \Vert v\Vert_D = \max(v) - \min(v),
\end{equation}
where $\max(v) = \max_{1 \leq i \leq d} v_i$ and $\min(v) = \min_{1 \leq i \leq d} v_i$.
It is invariant under the addition of constant vectors, that is,
\begin{equation}\label{eq:1}
    \Vert v + a\transpose{(1,\dots,1)} \Vert_D  = \Vert v \Vert_D
\end{equation}
for every $v\in\R^d$ and $a\in\R$.
Note that the function $\Vert\cdot\Vert_D$ is not a norm on $\R^d$ as $\Vert v\Vert_D=0$ for some nonzero vector
$v$. But $\Vert\cdot\Vert_D$ is a norm on some well-chosen subspaces.

\begin{lemma}\label{lem:normD-equiv-norminfty}
Let $\bf\in\R^d_{\geq 0}\setminus\{0\}$.
Then
\begin{enumerate}[\rm (i)]
\item $\Vert\cdot\Vert_D$ is a semi-norm on $\R^d$,
\item $\Vert\cdot\Vert_D$ is a norm on $\bf^\perp$,
\item $\Vert\cdot\Vert_D$ and $\Vert\cdot\Vert_\infty$ are equivalent norms on $\bf^\perp$. 
More precisely, 
    $2\Vert v\Vert_\infty \geq \Vert v\Vert_D \geq \Vert v\Vert_\infty$
    for every $v\in \bf^\perp$.
\end{enumerate}
\end{lemma}

\begin{proof}
    (i) We show that it is a semi-norm.
It is \emph{absolutely homogeneous}.
    Let $a\in\R_{\geq 0}$ and $v\in \R^d$. We have $\|av\|_D = \max(av) -
    \min(av)=a\max(v) - a\min(v)=a\|v\|_D$,
    and $\|-v\|_D=\|v\|_D$.
It is \emph{subadditive}. 
    Let $u,v\in\R^d$. We have $\|u+v\|_D = 
    \max(u+v) - \min(u+v)\leq
    \max(u) + \max(v)
    -\min(u) - \min(v)=\|u\|_D+\|v\|_D$.
It is \emph{non-negative}.
    For every $v\in\R^d$, we have 
    $\max(v) \geq \min(v)$ so that $\|v\|_D\geq0$.

(ii) Now we show that it is a norm on $\bf^\perp$.
It is \emph{definite}.
    Let $v\in \bf^\perp$ and suppose that $\|v\|_D=0$.
    We have $\max(v)=\min(v)$ so that $v= a(1,\dots,1)$ for some $a\in\R$.
    By definition of $v$, we have that $0= \transpose{v}\bf = a\Vert
    \bf\Vert_1$ which holds only if $a=0$ since $\Vert \bf\Vert_1\neq 0$.
    Therefore $v=0$. 

(iii) We always have
    \[
    \Vert v\Vert_D = \max(v) - \min(v) 
    \leq |\max(v)| + |\min(v)|
    \leq 2\Vert v\Vert_\infty.
    \]
If $\min(v)>0$, then $\transpose{v} \bf>0$ which contradicts the fact that $v$ is orthogonal
to $\bf$. 
    Similarly, $\max(v)<0$ implies $\transpose{v} \bf<0$ and contradicts the fact that $v$ is orthogonal
to $\bf$.    
Therefore $v\in \bf^\perp$ implies that $\min(v)\leq0\leq\max(v)$.
We conclude that
    \[
    \Vert v\Vert_D 
    = \max(v) - \min(v) 
    = |\max(v)| + |\min(v)| 
    \geq \Vert v\Vert_\infty.
    \]
\end{proof}

If follows from Lemma~\ref{lem:normD-equiv-norminfty} that
$\Vert\cdot|_{\bf^\perp}\Vert_D$ is a matrix semi-norm as soon as
$\bf\in\R^d_{\geq0}\setminus\{0\}$.
Note that it follows from Equation~\eqref{eq:1} that it satisfies
\begin{equation}
    \left\Vert \left(M + \transpose{(1,\dots,1)}  u\right)|_{\bf^\perp}\right\Vert_D  
    = \left\Vert M|_{\bf^\perp} \right\Vert_D
\end{equation}
for every matrix $M\in\R^{d\times d}$ and row vector $u\in\R^d$.

Finally, if $M\in\R^{d\times d}$ and 
$\bf\in\R^d_{\geq0}\setminus\{0\}$, then
it follows from Lemma~\ref{lem:normD-equiv-norminfty}~(iii) that
\begin{equation}\label{eq:infty-vs-D}
    \frac{1}{2}\cdot\left\Vert M\middle|_{\bf^\perp}\right\Vert_D
    \leq
    \left\Vert M\middle|_{\bf^\perp}\right\Vert_\infty
    \leq
    2\cdot\left\Vert M\middle|_{\bf^\perp}\right\Vert_D.
\end{equation}

\subsection{The supremum is attained on the boundaries}

We now state a general result which states that the supremum of
    $\left\Vert \transpose{M} \right\Vert_D^{M\R^d_{\geq0}}$
is attained on the boundaries of a finite number of
subcones forming a partition of $\R^d_{\geq0}$.
It is used in this article for proving the balancedness of almost all $\C$-adic sequences.

\begin{lemma}\label{lem:sup-on-boundaries}
    Let $d\geq 2$ and $M\in\R^{d\times d}_{>0}$ be a positive and invertible matrix.
Consider the set $\mathcal{H}$ of hyperplanes orthogonal to some vector in
    \begin{equation}\label{eq:hyperplanes-vectors}
S = \left(M \E \cup (\E-\E) \cup M(\E-\E)\right)\setminus\{0\}
    \end{equation}
where $\E= \{e_i:1\leq i\leq d\}$ and $\mathcal{D}$ be the finite union of one-dimensional intersections of several hyperplanes of $\mathcal{H}$:
        \begin{equation}\label{eq:droites}
\mathcal{D} = \bigcup_{\substack{h_1,\dots,h_{d-1}\in\mathcal{H}\\ \dim \bigcap_{1 \leq i < d} h_i =1}} \bigcap_{1 \leq i < d} h_i.
        \end{equation}
Then the maximal value of the semi-norm is attained at some vector $z$ in
    $\left(\mathcal{D}\setminus\pm (\transpose{M} )^{-1}\R^d_{> 0} \right) \setminus\{0\}$, i.e.,
\begin{equation}\label{eq:sup-on-vertices}
    \left\Vert \transpose{M} \right\Vert_D^{M\R^d_{\geq 0}} = 
    \max_{z\in \left(\mathcal{D}\setminus\pm (\transpose{M} )^{-1}\R^d_{> 0}\right) \setminus\{0\}}\frac{\Vert \transpose{M}  z\Vert_D}{\Vert z\Vert_D}.
\end{equation}
\end{lemma}

\begin{proof}
    We have
\[
    \left\Vert \transpose{M} \right\Vert_D^{M\R^d_{\geq0}}
    = \sup_{\bf\in M\R^d_{\geq0}\setminus\{0\}}\left\Vert \transpose{M} \middle|_{\bf^\perp}\right\Vert_D
    = \sup_{\bf\in M\R^d_{\geq0}\setminus\{0\}}\sup_{z\in \bf^\perp\setminus\{0\}}\frac{\Vert \transpose{M}  z\Vert_D}{\Vert z\Vert_D}
    = \sup_{z\in Z\setminus\{0\}}\frac{\Vert \transpose{M}  z\Vert_D}{\Vert z\Vert_D},
\]
where 
    \begin{align*}
        Z&=\{z\in\R^d\mid z\perp \bf \text{ for some } \bf\in M\R^d_{\geq0}\setminus\{0\}\}\\
         &=\{z\in\R^d\mid \transpose{z}  Mu=0 \text{ for some } u\in\R^d_{\geq0}\setminus\{0\}\}\\
         &=\{z\in\R^d\mid \text{ entries of } \transpose{M}  z \text{ are not all positive or all negative}\}\\
         &=\R^d\setminus \pm (\transpose{M} )^{-1}\R^d_{>0}
    \end{align*}
    The vectors of $Z$ correspond to $2^d-2$ of the $2^d$ cones delimited by
    the hyperplanes orthogonal to the vectors of $M\E$.

    The norm $\Vert z\Vert_D$ is a piecewise linear form which is linear on
    each of the cones delimited by the hyperplanes orthogonal to the nonzero vectors
    of $\E-\E$.
    There are $d!$ such cones.
    Similarly, the norm $\Vert \transpose{M}  z\Vert_D$ is a piecewise linear form which is linear on
    every of the cones delimited by the hyperplanes orthogonal to the nonzero vectors
    of $M(\E-\E)$.

    We consider any of the subcones $\Lambda$ delimited by hyperplanes orthogonal to
    some vectors in $S$ defined in Equation \eqref{eq:hyperplanes-vectors} that
    are inside of $Z$. 
    Remark that by construction both $\Vert z\Vert_D$ and $\Vert \transpose{M}z\Vert_D$
    are linear on $\Lambda$.
    The intersection of $\Lambda$ with the euclidean sphere
    of radius~1 is compact. 
    Therefore, the maximum $m$ of the function
    $\frac{\Vert \transpose{M}  z\Vert_D}{\Vert z\Vert_D}$ restricted to
    $\Lambda \setminus \{0\}$ is attained at some point $z_0$ with $\Vert z_0
    \Vert_2 =1$:
    \[
        \frac{\Vert \transpose{M}  z_0\Vert_D}{\Vert z_0\Vert_D}= m.
    \]

    Let $L$ be a linear form of $\R^d$ such that $L(z)=\Vert \transpose{M}  z\Vert_D - m\Vert z\Vert_D$ for $z\in\Lambda$. If $L=0$, then 
    $ \frac{\Vert \transpose{M}  z\Vert_D}{\Vert z\Vert_D}= m$
    is constant on $\Lambda$, so its maximum is attained on an edge of $\Lambda$.
    Otherwise,
    the equation $L(z)=0$ defines a hyperplane $H$ containing
    the origin. 
    By definition of the maximum we have
    $L(z)\leq 0$ for any $z\in \Lambda$.
    Therefore $\Lambda$ is contained in one of the halfspaces delimited by $H$.
    The set $H\cap\Lambda$ is either an edge, or contains an edge of $\Lambda$.
    Therefore, the maximum of $\Vert \transpose{M}  z\Vert_D - m\Vert z\Vert_D$ must the attained on an edge of $\Lambda$, that is,
    at some point in $\mathcal{D}\setminus\pm (\transpose{M} )^{-1}\R^d_{>0}$.
\end{proof}

\section{Convergence in the monoid generated by $C_1$ and $C_2$}
\label{section:primitivity and convergence}

In this section, we consider the monoid of $3\times3$ matrices generated by $C_1$ and $C_2$ and we study the weak convergence of sequences in $\{C_1,C_2\}^\N$, notably using Proposition~\ref{prop:convergence-sufficient-conditions}.

\subsection{Primitiveness}

Proposition~\ref{prop:convergence-sufficient-conditions} provides sufficient conditions for the convergence when the sequence is primitive. 
Our first task is to characterize primitive sequences of matrices in $\{C_1,C_2\}^\N$.

\begin{proposition}
\label{prop:primitiveness for C adic}
A sequence $\bm = (M_n)_{n \in \N} \in \{C_1,C_2\}^\N$ is not primitive if and only if there is some integer $N \geq 0$ such that for all $i \in \N$, $M_{N+2i}=M_{N+2i+1}$.
\end{proposition}

\begin{proof}
Given a matrix $M$, we associate with it a boolean matrix $B(M)$ of the same size defined by 
\[
(B(M))_{ij} = 
\begin{cases}
1, &	\text{if } M_{ij}>0; \\
0, & 	\text{otherwise}.
\end{cases}
\]
Thus a matrix $M$ has only positive entries if and only if $B(M)$ contains only 1's.

Assume first that there is some integer $N \geq 0$ such that for all $i \in \N$, $M_{N+2i}=M_{N+2i+1}$.
The graph in Figure~\ref{figure: graph for non-primitive} represents the possible boolean matrices associated with products of the form $C_1^{2k_0} C_2^{2k_1} C_1^{2k_2} \cdots C_2^{2k_{2n-1}}$.
The vertices are boolean matrices and there is an edge from $B_1$ to $B_2$ with label $C_i^2$ if $B_2 = B(B_1\cdot C_i^2)$.
We immediately check that this implies that the sequence $\bm$ is not primitive.
\begin{figure}[h]
\centering
\includegraphics{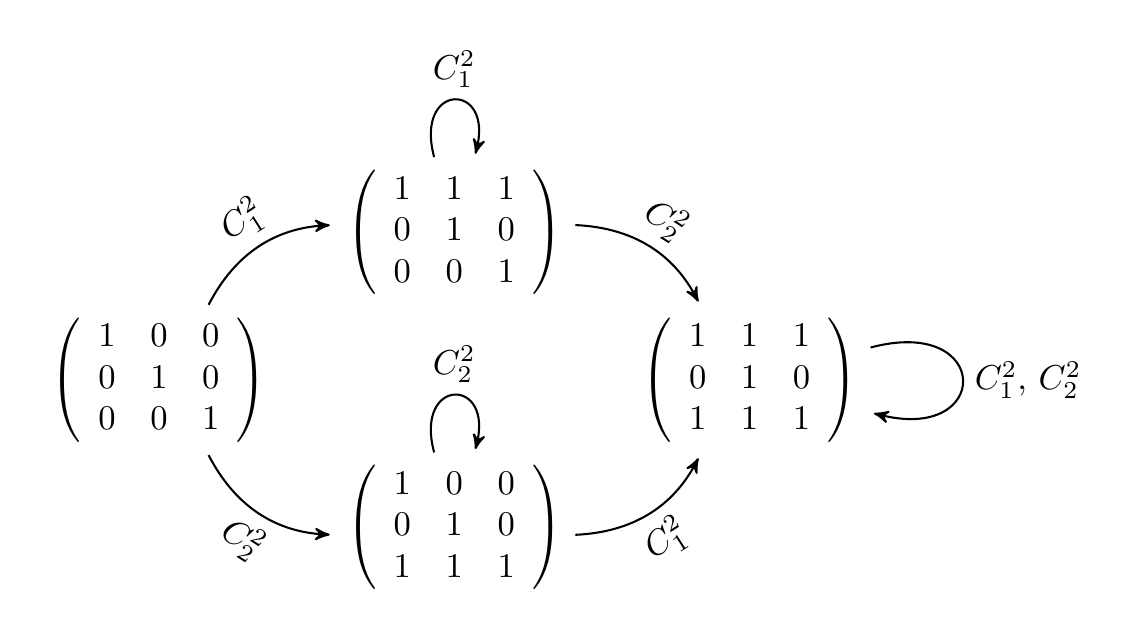}
\caption{If $M_{N+2i}=M_{N+2i+1}$ for all $i \in \N$, then $\bm$ is not primitive.}
\label{figure: graph for non-primitive}
\end{figure}

Now assume that there is no integer $N \geq 0$ such that for all $i \in \N$, $M_{N+2i}=M_{N+2i+1}$.
This implies that there are infinitely many $N \in \N$ such that the sequence $(M_n)_{n \geq N}$ starts with a product of the form $C_1C_2^{2k+1}C_1$ or $C_2C_1^{2k+1}C_2$.
Observe that we have
\[
B(C_1C_2C_1) = 
\left( \begin{array}{ccc}
1&1&1 \\0&1&1 \\1&1&0
\end{array} \right),
\quad
B(C_2C_1C_2) = 
\left( \begin{array}{ccc}
0&1&1 \\1&1&0 \\1&1&1
\end{array} \right)
\]
and, for $k \geq 1$,
\[
B(C_1C_2^{2k+1}C_1) = 
\left( \begin{array}{ccc}
1&1&1 \\1&1&1 \\1&1&0
\end{array} \right),
\quad
B(C_2C_1^{2k+1}C_2) = 
\left( \begin{array}{ccc}
0&1&1 \\1&1&1 \\1&1&1
\end{array} \right).
\]
We build graphs analogously to the one in Figure~\ref{figure: graph for non-primitive} but with starting vertex one of the 4 matrices above.
These graphs are represented in Figure~\ref{figure: graph for primitive} and we immediately check that $\bm$ is primitive.
\begin{figure}[h]
\centering
\includegraphics[width=\linewidth]{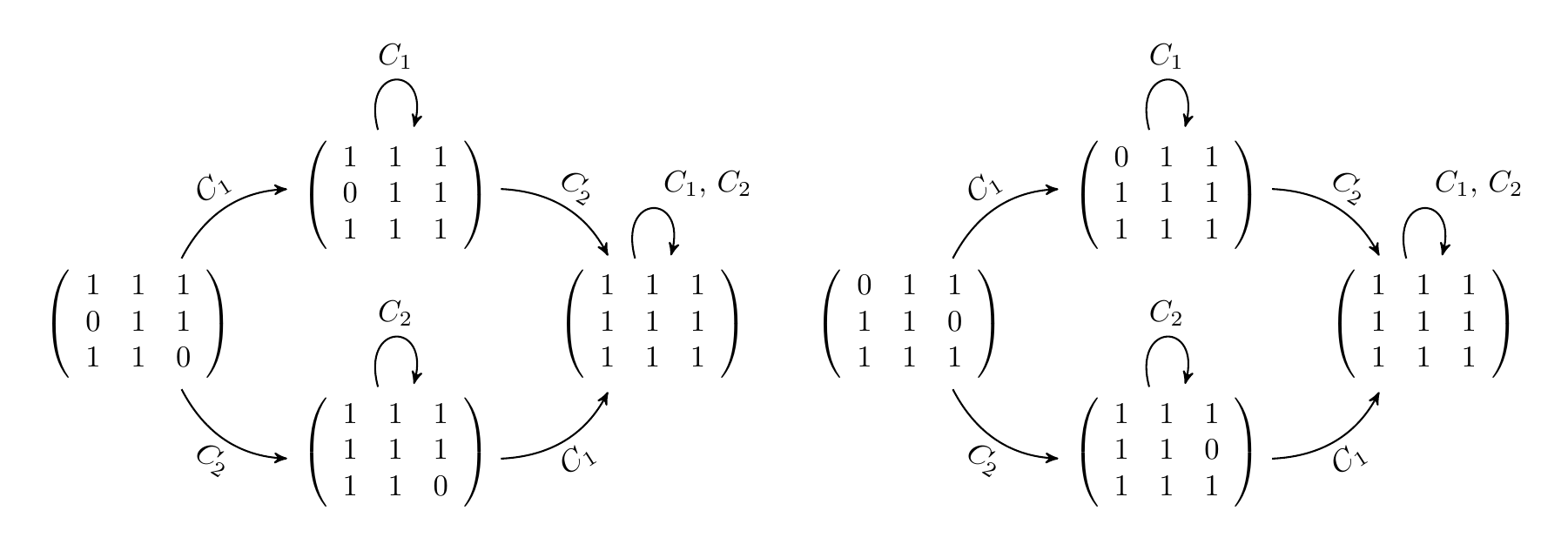}
\caption{If $\bm$ contains infinitely many occurrences of products of the form $C_1C_2^{2k+1}C_1$ or $C_2C_1^{2k+1}C_2$, then $\bm$ is primitive.}
\label{figure: graph for primitive}
\end{figure}
\end{proof}

\begin{lemma}\label{lem:primitive-implies-nice-factorization}
    Let $\bm = (M_n)_{n \in \N} \in \{C_1,C_2\}^\N$ be a sequence of matrices.
    If $\bm$ contains infinitely many occurrences of both $C_1$ and $C_2$, then
    there exists an increasing sequence of integers 
$(n_m)_{m \in \mathbb{N}}$ such that
$n_0=0$ and
\begin{equation}\label{eq:def_n_m}
M_{[n_m,n_{m+1)}} \in \{C_1 C_2^k C_1, C_2 C_1^k C_2 \mid k \in \mathbb{N}\}
\end{equation}
for all $m\in\N$.
\end{lemma}

\begin{proof}
By induction, suppose that there exist $M\in\N$ and an increasing sequence of integers 
$(n_m)_{0\leq m \leq M}$ such that $n_0=0$ and satisfying Equation~\eqref{eq:def_n_m} for every $m\in\N$ such that $0\leq m< M$.
Then, the next value $n_{M+1}$ of the sequence is defined recursively as
    \[
    	n_{M+1} = 
 		\min\{k > n_M \mid M_{k} = M_{n_M}\} + 1.
    \]
The existence of $n_{M+1}$ is obvious since both $C_1$ and $C_2$ occur infinitely often in $\bm$.
\end{proof}

\subsection{The semi-norm $\|\cdot\|_D$ in the monoid generated by $C_1$, $C_2$}

The next lemma presents a nice property of the matrices of the form
$C_1C_2^{n}C_1$ or $C_2C_1^{n}C_2$ for $n\geq0$ in terms of the norm $\Vert\cdot\Vert_D$.
Note that when $d=3$, $\Vert(v_1,v_2,v_3)\Vert_D=\max\{|v_1-v_2|,|v_1-v_3|,|v_2-v_3|\}$.

\begin{lemma}\label{lem:C1C2kC1-upperbound-1}
    For every $n\in\N$, we have
    \[
        \left\Vert \transpose{(C_1C_2^{n}C_1)}  \right\Vert_D^{\R^3_{\geq 0}} = 1
        \quad
        \text{ and }
        \quad
        \left\Vert \transpose{(C_2C_1^{n}C_2)}  \right\Vert_D^{\R^3_{\geq0} } = 1.
    \]
\end{lemma}

\begin{proof}
    Let $\bf\in\R^3_{\geq0}\setminus\{0\}$ be a nonzero nonnegative vector
    and let $z\in \bf^\perp$.
	We only prove it for $C_1C_2^{n}C_1$, the other one being symmetric.
    We separate the odd and even cases.
    Let $k\in\N$. Using Equation~\eqref{eq:1}, we have
    \[
        \Vert \transpose{(C_1C_2^{2k}C_1)}  z\Vert_D 
        =
        \left\Vert 
\left(\begin{array}{ccc}
1 & k & 0 \\
1 & k+1 & 0 \\
1 & k & 1
\end{array}\right)
        z\right\Vert_D 
        =
        \left\Vert 
\left(\begin{array}{ccc}
0 & 0 & 0 \\
0 & 1 & 0 \\
0 & 0 & 1
\end{array}\right)
        z\right\Vert_D 
    \]
and
    \[
        \Vert \transpose{(C_1C_2^{2k+1}C_1)}  z\Vert_D 
        =
        \left\Vert 
\left(\begin{array}{ccc}
1 & k & 1 \\
1 & k+1 & 1 \\
1 & k+1 & 0
\end{array}\right)
        z\right\Vert_D 
        =
        \left\Vert 
\left(\begin{array}{ccc}
0 & -1 & 0 \\
0 & 0 & 0 \\
0 & 0 & -1
\end{array}\right)
        z\right\Vert_D 
    \]
Since $z$ is orthogonal to $\bf \in \R^3_{\geq0}\setminus\{0\}$, we have $\min(z) \leq 0 \leq \max(z)$.
Thus, for 
\[
z_1 = 
\left(\begin{array}{ccc}
0 & 0 & 0 \\
0 & 1 & 0 \\
0 & 0 & 1
\end{array}\right)
        z
\quad \text{and} \quad 
z_2 =  
\left(\begin{array}{ccc}
0 & -1 & 0 \\
0 & 0 & 0 \\
0 & 0 & -1
\end{array}\right)
        z,
\]
we have
\begin{align*}
	\min(z) \leq \min (z_1)  \leq & 0 \leq \max(z_1) \leq \max(z)		\\	
	-\max(z) \leq \min (z_2) \leq & 0 \leq \max(z_2) \leq -\min(z),
\end{align*}
which implies that $\Vert z_1 \Vert_D, \Vert z_2 \Vert_D \leq \Vert z \Vert_D$.

This shows that for all $z\in \bf^\perp\setminus\{0\}$ and all $n \in \mathbb{N}$,
\[
    \frac{\Vert \transpose{(C_1C_2^nC_1)}  z \Vert_D}{\Vert z \Vert_D} \leq 1.
\]
Furthermore, $\bf$ being nonzero nonnegative, there exist $a,b \in \mathbb{R}$ with $a\leq 0\leq b$ such that $z=(0,a,b) \in \bf^\perp\setminus\{0\}$.
For this vector $z$, we have 
\begin{align*}
    \Vert \transpose{(C_1C_2^{2k}C_1)}  z\Vert_D &= \Vert (0,a,b)\Vert_D = \Vert z \Vert_D;		\\
    \Vert \transpose{(C_1C_2^{2k+1}C_1)}  z\Vert_D &= \Vert (-a,0,-b)\Vert_D = \Vert z\Vert_D,
\end{align*}
showing that $\sup_{z \in \bf^\perp\setminus\{0\}}\frac{\Vert \transpose{(C_1C_2^nC_1)}  z \Vert_D}{\Vert z \Vert_D} = 1$.
\end{proof}

Observe that Lemma~\ref{lem:C1C2kC1-upperbound-1} does not hold in general.
Indeed some matrices $M$ obtained as the product of matrices $C_1$ and
$C_2$ are such that $\Vert \transpose{M}  z\Vert_D > \Vert z\Vert_D$. For example,
it is the case for $M=C_1C_2$. For $z=(8, -5, 13)$ we 
compute
\[
    \Vert \transpose{(C_1C_2)}  z\Vert_D = 
    \left\Vert 
\left(\begin{array}{ccc}
1 & 0 & 1 \\
1 & 1 & 0 \\
0 & 1 & 0
\end{array}\right)\cdot
\left(\begin{array}{c}
8 \\
-5 \\
13
\end{array}\right)
    \right\Vert_D = 
    \left\Vert 
\left(\begin{array}{c}
21 \\
3 \\
-5
\end{array}\right)
    \right\Vert_D = 26
\]
which is larger than $\Vert z\Vert_D=18$. 

\subsection{Convergence in the monoid generated by $C_1$ and $C_2$}

The next result shows that any sequence $(M_n)_{n \in \N} \in \{C_1,C_2\}^\N$ is weakly convergent.

Let $\bu \in \mathbf{R}_{\geq 0}^3$ be a vector belonging to the cone $\bigcap_{n\geq0} M_{[0,n)} \R^3_{\geq0}$.
For all $k \in \N$, we define the vector
\[
	\bu^{(k)} = M_{[0,k)}^{-1} \bu.
\]   

\begin{proposition}\label{prop:convergence-for-C1-C2}
For any sequence $\bm = (M_n)_{n \in \N} \in \{C_1,C_2\}^\N$, there exists a vector $\bu\in\R^3_{\geq0} \setminus \{0\}$ satisfying
    \begin{equation}
        \bigcap_{n\geq0} M_{[0,n)} \R^3_{\geq0} = \R_{\geq0}\bu.
    \end{equation}
\end{proposition}

\begin{proof}
We split the proof into two cases, depending on whether $\bm$ is primitive or not.

Assume first that $\bm = (M_n)_{n \in \N} \in \{C_1,C_2\}^\N$ is a primitive sequence.
    From Proposition~\ref{prop:primitiveness for C adic} and Lemma~\ref{lem:primitive-implies-nice-factorization}
    there exists an increasing sequence of integers
    $(n_m)_{m \in \mathbb{N}}$ such that
    $n_0=0$ and
    \begin{equation}\label{eq:def_n_m_2}
    M_{[n_m,n_{m+1)}} \in \{C_1 C_2^k C_1, C_2 C_1^k C_2 \mid k \in \mathbb{N}\}
    \end{equation}
    for all $m\in\N$.
    We compute using Equation~\eqref{eq:infty-vs-D}, Lemma~\ref{lem:semi-norm-on-AB}
    and Lemma~\ref{lem:C1C2kC1-upperbound-1} that
    \begin{align*}
    	\left\Vert \transpose{M_{[0,n_\ell)}} \right\Vert_\infty^{M_{[0,n_\ell)}\R^3_{\geq0}}
    	& \leq 
        2 \left\Vert \transpose{M_{[0,n_\ell)}} \right\Vert_D^{M_{[0,n_\ell)}\R^3_{\geq0}}	\\
               &=
        2 \left\Vert \prod_{m=0}^{\ell-1} \transpose{M_{[n_m,n_{m+1})}} 
               \right\Vert_D^{\prod_{m=0}^{\ell-1} M_{[n_m,n_{m+1})}\R^3_{\geq0}}\\
               &\leq
        2 \prod_{m=0}^{\ell-1} \left\Vert \transpose{M_{[n_m,n_{m+1})}} 
               \right\Vert_D^{M_{[n_m,n_{m+1})}\R^3_{\geq0}}\\
               &\leq
        2 \prod_{m=0}^{\ell-1} \left\Vert \transpose{M_{[n_m,n_{m+1})}} 
               \right\Vert_D^{\R^3_{\geq0}}
               =2.
    \end{align*}
    Therefore, from Proposition~\ref{prop:convergence-sufficient-conditions},
    there exists a vector $\bu\in\R^3_{\geq0} \setminus\{0\}$ satisfying
    \begin{equation}
        \bigcap_{n\geq0} M_{[0,n)} \R^3_{\geq0} =
        \bigcap_{m\geq0} M_{[0,n_m)} \R^3_{\geq0} = \R_{\geq0}\bu
    \end{equation}
    and the conclusion follows.

Assume now that $\bm$ is not primitive.
From Proposition~\ref{prop:primitiveness for C adic}, there is an integer $N \geq 0$ such that $M_{N+2i} = M_{N+2i+1}$ for all $i \in \N$.
Let us show that $\bigcap_{n\geq N} M_{[N,n)} \R^3_{\geq0}$ is one-dimensional.

Let $\bu \in \bigcap_{n\geq0} M_{[0,n)} \R^3_{\geq0}$.
    For all $k \in \N$, let us write $\bu^{(k)} = (u_1^{(k)},u_2^{(k)},u_3^{(k)})$ and let us show that $u_2^{(N)}=0$.
Indeed, for all $k \in \N$, the vector $\bu^{(N+2k+2)}$ is equal to one of the following two vectors:
    \begin{eqnarray*}
    C_1^{-2}\bu^{(N+2k)}
    &=&
    (u_1^{(N+2k)}-u_2^{(N+2k)}-u_3^{(N+2k)},u_2^{(N+2k)},u_3^{(N+2k)});
    \\
    C_2^{-2}\bu^{(N+2k)}
    &=&
    (u_1^{(N+2k)},u_2^{(N+2k)},u_3^{(N+2k)}-u_1^{(N+2k)}-u_2^{(N+2k)}).
\end{eqnarray*}
In both cases, the middle entry is unchanged.
Thus, by induction, for every $k\in\N$, we have $u^{(N+2k)}_2=u_2^{(N)}$.
Also the sum of the two other entries decreases by at least $u_2^{(N)}$. 
Therefore, for every $k\in\N$, we have 
    \[
    0\leq u^{(N+2k)}_1+u^{(N+2k)}_3 \leq u_1^{(N)}+u_3^{(N)}-ku_2^{(N)},
    \]
which implies that $u_2^{(N)}=0$.

    To end the proof, it suffices to observe that if $u_2^{(N)}=0$, the action of $C_1^{-2}$ and $C_2^{-2}$ on the vectors $\bu^{(N+k)}$, $k \in \N$, corresponds to the well-known additive Euclidean algorithm applied to the first and third components.
This shows that $\bigcap_{n\geq N} M_{[N,n)} \R^3_{\geq0}$ is one-dimensional and thus that so is $\bigcap_{n\geq 0} M_{[0,n)} \R^3_{\geq0}$.
\end{proof}

\section{Rational dependencies of the limit cone}
\label{sec:rational dependencies}

Proposition~\ref{prop:convergence-for-C1-C2} states that for any sequence of matrices $(M_n)_{n \geq 0} \in \{C_1,C_2\}^{\mathbb{N}}$, the cone $\bigcap_{n \geq 0} M_{[0,n)}\mathbb{R}^3_{\geq 0}$ converges to a half-line $\R_{\geq0}\bf$ with $\|\mathbf{f}\|_1=1$.
In this section, we give more insight on the properties of $(M_n)_{n \geq 0}$ in terms of the rational dependencies of the entries of the vector $\mathbf{f}$.
We define the dimension of a vector $\bf\in\R^3$ as the dimension of the $\Q$-vector space spanned by its entries, denoted
$\dim_\Q(\bf)$.
If $\dim_\Q(\bf)<3$, then there exists a rational dependency between its
entries.
If $\dim_\Q(\bf)=3$, we say that $\bf$ is totally irrational.
In this section, we show that $(M_n)_{n \geq 0}$ is primitive if and only if 
the vector $\mathbf{f}$ is totally irrational.
More precisely, we prove the following result.
Recall that $\Delta$ denotes the simplex 
$\{\bx\in\R^3_{\geq0} \mid \Vert\bx\Vert_1=1\}$.

\begin{theorem} \label{thm:dim-sur-Q}
Let $(M_n)_{n\in\N}\in\{C_1,C_2\}^\N$ and let $\bf\in\Delta$ such that
    $\bigcap_{n \in \mathbb{N}} M_{[0,n)}\R^3_{\geq 0} = \R_{\geq 0}\bf$.
\begin{enumerate}[\rm (i)]
    \item $\dim_\Q(\bf)=1$
        if and only if 
        $(M_n)_{n\in\N}\in\{C_1,C_2\}^*\{C_1^\N,C_2^\N\}$.
    \item $\dim_\Q(\bf)=2$
        if and only if
        $(M_n)_{n\in\N}\in \left(\{C_1,C_2\}^*\{C_1^2,C_2^2\}^\N\right) 
        \setminus \{C_1,C_2\}^*\{C_1^\N,C_2^\N\}$.
    \item $\dim_\Q(\bf)=3$ 
        if and only 
        if $(M_n)_{n\in\N}$ is primitive.
\end{enumerate}
\end{theorem}

Note that the three conditions are mutually exclusive since 
we proved in Proposition~\ref{prop:primitiveness for C adic}
that $(M_n)_{n\in\N}$ is primitive
if and only if
$(M_n)_{n\in\N}\notin \{C_1,C_2\}^*\{C_1^2,C_2^2\}^\N$.
The proofs of the first two cases of Theorem~\ref{thm:dim-sur-Q} 
are done separately in
Lemma~\ref{lem:quand-dimQ=1} and
Lemma~\ref{lem:quand-dimQ=2}.

\begin{lemma}\label{lem:quand-dimQ=1}
    Let $(M_n)_{n\in\N}\in\{C_1,C_2\}^\N$
    and let $\bf \in\Delta$ such that
    $\bigcap_{n \in \mathbb{N}} M_{[0,n)}\R^3_{\geq 0} = \R_{\geq 0}\bf$.
    We have
\begin{enumerate}[\rm (i)]
\item $M_n=C_1$ for every $n\in\N$ if and only if 
    $\bf=(1,0,0)$, 
\item $M_n=C_2$ for every $n\in\N$ if and only if 
    $\bf=(0,0,1)$, 
\item 
    $(M_n)_{n\in\N}\in\{C_1,C_2\}^*\{C_1^\N,C_2^\N\}$
    if and only if $\dim_\Q(\bf)=1$.
\end{enumerate}
\end{lemma}

\begin{proof}
    For every $k\in\N$,
    let $\bf^{(k)}=(f^{(k)}_1,f^{(k)}_2,f^{(k)}_3)=M_{[0,k)}^{-1}\bf$.

    (i)
    If $\bf=(1,0,0)$, then $M_0=C_1$, since $(1,0,0)\notin C_2\R_{\geq 0}^3$.
    Moreover, $\bf^{(1)}=\bf$.
    Therefore, by induction, $M_n=C_1$ for every $n\in\N$.
    Conversely,
    $\bigcap_{n \in \mathbb{N}} C_1^n \R^3_{\geq 0} = \R_{\geq 0}(1,0,0)$.

    (ii) The proof is done similarly to the proof of (i).

    (iii)
    Suppose that
    $(M_n)_{n\in\N}\in\{C_1,C_2\}^*\{C_1^\N,C_2^\N\}$.
    Then, there exists $k\in\N$ such that
    $(M_n)_{n\geq k}\in\{C_1^\N,C_2^\N\}$.
    From (i) and (ii), 
    $\bf^{(k)}\in\{(1,0,0), (0,0,1)\}$.
    Thus $\bf\in\Q^3\setminus\{0\}$ so that $\dim_\Q(\bf)=1$.
    Conversely, if 
    $\dim_\Q(\bf)=1$ then 
    $\bf\in\Q^3\setminus\{0\}$.
    We may suppose $b\bf\in\Z^3\setminus\{0\}$ for some $b\in\N_{>0}$.
    If $\min(\bf^{(k)}) \neq 0$, then $\|\bf^{(k+1)}\|_1 \leq \|\bf^{(k)}\|_1 -1/b$. 
    Thus there exists $k\in\N$
    such that $\min(\bf^{(k)})=0$.
Then, if $\bf^{(k')}$ is not of the form $(0,0,a)$, $(0,a,0)$ or $(a,0,0)$ for some $a$, then $\|\bf^{(k'+2)}\|_1 \leq \|\bf^{(k')}\|_1 -1/b$.    
    Thus there exists $k'\in\N$
    such that $\bf^{(k')} \in \{(a,0,0),(0,a,0),(0,0,a)\}$ for some $a>0$.
    If $\bf^{(k')} = (0,a,0)$, then $\bf^{(k'+1)} \in \{(a,0,0),(0,0,a)\}$.
    Like for the cases (i) and (ii), we deduce that $(M_n)_{n \geq k'+1}$ is in $\{C_1^\N,C_2^\N\}$, which ends the proof.
\end{proof}

\begin{lemma}\label{lem:quand-dimQ=2}
    Let $(M_n)_{n\in\N}\in\{C_1,C_2\}^\N$
    and let $\bf=(f_1,f_2,f_3)\in\Delta$ such that
    $\bigcap_{n \in \mathbb{N}} M_{[0,n)}\R^3_{\geq 0} = \R_{\geq 0}\bf$.
    We have
\begin{enumerate}[\rm (i)]
\item if $M_{2n}=M_{2n+1}$ for every $n\in\N$
    then $f_2=0$ and $\dim_\Q(\bf)\leq 2$,
\item 
    if $\dim_\Q(\bf)=2$ and $f_2=0$,
    then $M_{2n}=M_{2n+1}$ for every $n\in\N$.
\end{enumerate}
\end{lemma}

\begin{proof}
    For every $k\in\N$,
    let $\bf^{(k)}=(f^{(k)}_1,f^{(k)}_2,f^{(k)}_3)=M_{[0,k)}^{-1}\bf$.
    The proof of the first item uses the same arguments as in the proof of Proposition~\ref{prop:convergence-for-C1-C2}.

    (i) 
Note that,
for every $k\in\N$,
    $\bf^{(2k+2)}$ is equal to one of the following two vectors:
    \begin{align}
\label{eq-f2k+2-first}
    C_1^{-2}\bf^{(2k)}&=(f^{(2k)}_1-f^{(2k)}_2-f^{(2k)}_3,f^{(2k)}_2,f^{(2k)}_3) 
        \text{ or }
        \\
\label{eq-f2k+2-second}
    C_2^{-2}\bf^{(2k)}&=(f^{(2k)}_1,f^{(2k)}_2,f^{(2k)}_3-f^{(2k)}_1-f^{(2k)}_2).
\end{align}
In both cases, the middle entry is unchanged.
Thus, by induction, for every $k\in\N$, we have $f^{(2k)}_2=f_2$.
Also the sum of the two other entries decreases by at least $f_2$. 
Therefore, for every $k\in\N$, we have 
    \[
    0\leq f^{(2k)}_1+f^{(2k)}_3 \leq f_1+f_3-kf_2
    \]
which implies that $f_2=0$.

    (ii)
    We do the proof by induction.
    Suppose that $f_2^{(2k)}=0$.
    We have that $\bf^{(2k)}\in M_{2k}M_{2k+1}\R^3_{\geq 0}$.
    If $M_{2k}M_{2k+1}=C_1C_2$ and $\bf^{(2k+2)}=(\alpha,\beta,\gamma)$ for some $\alpha,\beta,\gamma\geq 0$, 
    then $\bf^{(2k)}=(\alpha+\beta,\beta+\gamma,\alpha)$ 
    so that $\beta=\gamma=0$ and $f_1^{(2k)}=\alpha=f_3^{(2k)}$.
	Thus we have $\bf^{(2k+1)} = (0,\alpha,0)$, so that $\dim_\Q(\bf) = \dim_\Q(\bf^{(2k+1)}) = 1$, which is a contradiction.
	We similarly reach a contradiction when supposing $M_{2k}M_{2k+1}=C_2C_1$.
    We thus obtain that
    $M_{2k}M_{2k+1}=C_1C_1$ or $M_{2k}M_{2k+1}=C_2C_2$.
    As in both cases we get, using \eqref{eq-f2k+2-first} or \eqref{eq-f2k+2-second}, $f^{(2k+2)}_2=0$, this ends the proof.
\end{proof}

We now give the description of primitive sequences.

\begin{proof}[\Proofof Theorem~\ref{thm:dim-sur-Q}]
    Statement (i) 
    follows from
    Lemma~\ref{lem:quand-dimQ=1}.
Statement (ii) follows from Statements (i) and (iii) and from Proposition~\ref{prop:primitiveness for C adic}.
    Let us thus show Statement (iii).

    For every $k\in\N$,
    let $\bf^{(k)}=(f^{(k)}_1,f^{(k)}_2,f^{(k)}_3)=M_{[0,k)}^{-1}\bf$.
Let us assume that $\dim_\Q(\bf)=3$.
If $(M_n)_{n \in \N}$ is not primitive, then by Proposition~\ref{prop:primitiveness for C adic} and Lemma~\ref{lem:quand-dimQ=2}(i), there exists $N\in \N$ such that $\dim_\Q(\bf^{(N)})\leq 2$.
As $\bf = M_{[0,N)}\bf^{(N)}$, we obtain $\dim_\Q(\bf)\leq 2$, which contradicts our hypothesis.

Let us now assume that $(M_n)_{n \in \N}$ is primitive.
Using Proposition~\ref{prop:primitiveness for C adic} and Lemma~\ref{lem:quand-dimQ=1}(iii), we cannot have $\dim_\Q(\bf)=1$.
    So we assume $\dim_\Q(\bf)=2$.
    Observe first that, if $\bf^{(N)}$ has a zero entry
for some~$N \in \N$, then either $f_2^{(N)}=0$ or $f_2^{(N+1)}=0$.
Then, since $\dim_\Q(\bf^{(N)}) = \dim_\Q(\bf^{(N+1)}) = \dim_\Q(\bf)$, this would imply by Proposition~\ref{prop:primitiveness for C adic} and Lemma~\ref{lem:quand-dimQ=2}(ii) that $(M_n)_{n \in \N}$ is not primitive, which is a contradiction.
    From now on we assume that all entries of $\bf^{(n)}$
are positive for all $n$, and we show that we again reach a contradiction.

Since $\dim_\Q(\bf)<3$, there exists some integer vector $\bv \in \bf^\perp\setminus\{0\}$.
The sequence $(M_n)_{n \in \N}$ can be factored over
$\{C_1C_2^kC_1, C_2C_1^kC_2 \mid k \in \N \}$ by Lemma~\ref{lem:primitive-implies-nice-factorization}.
Let us consider the sequence $(n_m)$ such that $n_0=0$ and $M_{n_m}\ldots M_{n_{m+1}-1}$
is in this set for all $m\in\mathbb N$.
Since $C_1$ and $C_2$ are unimodular, for all $m \in \mathbb N$, $\transpose{M_{[0,n_m)}} \bv$ is an integer vector and so $(\Vert \transpose{M_{[0,n_m)}} \bv \Vert_D)_{m \in \mathbb N}$ is a nonnegative integer sequence.
In what follows, we reach a contradiction by showing that $(\Vert \transpose{M_{[0,n_m)}} \bv \Vert_D)_{m \in \mathbb N}$ is non-increasing and decreases infinitely often.

The proof that $(\Vert \transpose{M_{[0,n_m)}} \bv \Vert_D)_{m \in \mathbb N}$ is non-increasing is already done in the first part of the proof of Lemma~\ref{lem:C1C2kC1-upperbound-1}, observing that for all $m$, $\bf^{(n_m)}$ has positive entries and $\transpose{M_{[0,n_m)}} \bv \in (\bf^{(n_m)})^\perp \setminus\{0\}$.
Furthermore, if $\transpose{M_{[0,n_m)}} \bv = (a,b,c)$, then $\Vert \transpose{M_{[0,n_m)}} \bv \Vert_D = \max(|b-a|, |c-b|, |c-a|)$ and, using Equation~\eqref{eq:1} and considering separately the even and odd cases (like in the proof of Lemma~\ref{lem:C1C2kC1-upperbound-1}), we get
\[	\Vert \transpose{M_{[0,n_{m+1})}} \bv \Vert_D
	=
	\begin{cases}
		\max\{|b|,|c|,|c-b|\}, 
			& \text{if } M_{[n_m,n_{m+1})} \in \{C_1 C_2^k C_1 \mid k \in \mathbb N\}; \\
		\max\{|a|,|b|,|a-b|\}, 
			& \text{if } M_{[n_m,n_{m+1})} \in \{C_2 C_1^k C_2 \mid k \in \mathbb N\}.
	\end{cases}
\]
Since $\bf^{(n_m)}$ has positive entries and $\transpose{M_{[0,n_m)}} \bv \in (\bf^{(n_m)})^\perp \setminus\{0\}$, we have $\min \{a,b,c\}<0<\max\{a,b,c\}$, hence
\begin{itemize}
\item
if $M_{[n_m,n_{m+1})} \in \{C_1 C_2^k C_1 \mid k \in \mathbb N\}$, 
$\Vert \transpose{M_{[0,n_{m+1})}} \bv \Vert_D = \Vert \transpose{M_{[0,n_{m})}} \bv \Vert_D$ 
if and only if 
$a$ is (inclusively) between $b$ and $c$, in which case $\Vert \transpose{M_{[0,n_{m+1})}} \bv \Vert_D = \Vert \transpose{M_{[0,n_{m})}} \bv \Vert_D = |c-b|$;
\item
if $M_{[n_m,n_{m+1})} \in \{C_2 C_1^k C_2 \mid k \in \mathbb N\}$, 
$\Vert \transpose{M_{[0,n_{m+1})}} \bv \Vert_D = \Vert \transpose{M_{[0,n_{m})}} \bv \Vert_D$ 
if and only if 
$c$ is (inclusively) between $a$ and $b$, in which case $\Vert \transpose{M_{[0,n_{m+1})}} \bv \Vert_D = \Vert \transpose{M_{[0,n_{m})}} \bv \Vert_D = |a-b|$.
\end{itemize}

Using Proposition~\ref{prop:primitiveness for C adic}, there are infinitely many integers $m$ such that either $M_{[n_m,n_{m+1})}$ is in $\{C_i C_j^{2k+1}C_i \mid k \in \mathbb N, \{i,j\} = \{1,2\}\}$, or $M_{[n_m,n_{m+1})}M_{[n_{m+1},n_{m+2})}$ is in $\{C_i C_j^{2k} C_i C_j C_i^{2\ell} C_j \mid k,\ell \in \mathbb N,\{i,j\} = \{1,2\}\}$.
We show that for any such $m$, we have $\Vert \transpose{M_{[0,n_{m+2})}} \bv \Vert_D < \Vert \transpose{M_{[0,n_m)}} \bv \Vert_D$.
Let us write $\transpose{M_{[0,n_m)}} \bv = (a,b,c)$.

Assume first that $M_{[n_m,n_{m+1})} = C_1C_2^{2k+1}C_1$ for some $k \in \mathbb N$; the case $M_{[n_m,n_{m+1})} = C_2C_1^{2k+1}C_2$ is symmetric.
Writing 
\[
\transpose{M_{[0,n_{m+1})}} \bv 
=
\left( 
\begin{array}{c}
a' \\ b' \\ c'
\end{array} 
\right) 
=
\left( 
\begin{array}{l}
a+kb+c \\ a+(k+1)b+c \\ a+(k+1)b
\end{array} 
\right),
\]
we have 
\begin{align*}
& \Vert \transpose{M_{[0,n_{m})}} \bv \Vert_D
= \Vert \transpose{M_{[0,n_{m+1})}} \bv \Vert_D 
= \Vert \transpose{M_{[0,n_{m+2})}} \bv \Vert_D \\
\Leftrightarrow
& 
\begin{cases}
\Vert \transpose{M_{[0,n_{m})}} \bv \Vert_D
 = \Vert \transpose{M_{[0,n_{m+1})}} \bv \Vert_D
= |c-b|	\\
\Vert \transpose{M_{[0,n_{m+1})}} \bv \Vert_D 
 = \Vert \transpose{M_{[0,n_{m+2})}} \bv \Vert_D
\in \{|c'-b'|,|a'-b'|\}
\end{cases}
\end{align*}
Observing that $|c'-b'|=|c|$ and $|a'-b'|=|b|$, we deduce that
$\Vert \transpose{M_{[0,n_{m+2})}} \bv \Vert_D
<
 \Vert \transpose{M_{[0,n_{m})}} \bv \Vert_D$.

Now assume that $M_{[n_m,n_{m+1})} = C_1C_2^{2k}C_1$ and that $M_{[n_m,n_{m+1})} = C_2C_1^{2\ell}C_2$; the case $M_{[n_m,n_{m+1})} = C_2C_1^{2k}C_2$ and $M_{[n_m,n_{m+1})} = C_1C_2^{2\ell}C_1$ is symmetric.

Writing 
\[
\transpose{M_{[0,n_{m+1})}} \bv 
=
\left( 
\begin{array}{c}
a' \\ b' \\ c'
\end{array} 
\right) 
=
\left( 
\begin{array}{l}
a+kb \\ a+(k+1) b \\ a+k b + c
\end{array} 
\right),
\]
we have 
\begin{align*}
& \Vert \transpose{M_{[0,n_{m})}} \bv \Vert_D
= \Vert \transpose{M_{[0,n_{m+1})}} \bv \Vert_D 
= \Vert \transpose{M_{[0,n_{m+2})}} \bv \Vert_D \\
\Leftrightarrow
& 
\begin{cases}
\Vert \transpose{M_{[0,n_{m})}} \bv \Vert_D
 = \Vert \transpose{M_{[0,n_{m+1})}} \bv \Vert_D
= |c-b|	\\
\Vert \transpose{M_{[0,n_{m+1})}} \bv \Vert_D 
 = \Vert \transpose{M_{[0,n_{m+2})}} \bv \Vert_D
= |a'-b'|
\end{cases}
\end{align*}
Observing that $|a'-b'|=|b|$, we deduce that
$\Vert \transpose{M_{[0,n_{m+2})}} \bv \Vert_D
<
 \Vert \transpose{M_{[0,n_{m})}} \bv \Vert_D$. 
 \end{proof}

\begin{corollary} \label{cor:nonuniqueimpliesnonprimitive}
Let $\bf \in \Delta$.
If there exist two different sequences $(M_n)_{n \in \N}, (M_n')_{n \in \N} \in \{C_1,C_2\}^\N$ such that 
\[
\bigcap_{n \in \mathbb{N}} M_{[0,n)}\R^3_{\geq 0} 
= \R_{\geq 0}\bf
= \bigcap_{n \in \mathbb{N}} M_{[0,n)}'\R^3_{\geq 0}, 
\]
then $\dim_\Q(\bf)\leq 2$ and both sequences $(M_n)_{n \in \N}, (M_n')_{n \in \N}$ are not primitive.
\end{corollary}

\begin{proof}
    For every $k\in\N$,
    let 
    \[
    \bf^{(k)}=(f^{(k)}_1,f^{(k)}_2,f^{(k)}_3)=M_{[0,k)}^{-1}\bf 
    \qquad \text{and} \qquad
    \bf'^{(k)}=(f'^{(k)}_1,f'^{(k)}_2,f'^{(k)}_3)=M_{[0,k)}'^{-1}\bf.
    \]
    There exists $n\in\N$ such that
    $M_k=M'_k$ for every $k\in\N$ with $0\leq k<n$ and
    $M_n\neq M'_n$.
    Since $\bf^{(n)} \in M_n \R^3_{\geq 0}$, $\bf'^{(n)} \in M'_n \R^3_{\geq 0}$ and $\bf^{(n)} = \bf'^{(n)}$, this implies that $f^{(n)}_1=f^{(n)}_3$ so that $\dim_\Q(\bf)\leq 2$.
From Theorem~\ref{thm:dim-sur-Q}
$(M_n)_{n\in\N}$ and $(M'_n)_{n\in\N}$ are not primitive.
\end{proof}

As a consequence, any primitive sequence of matrices
$(M_n)_{n\in\N}\in\{C_1,C_2\}^\N$ can be recovered from the
vector it contracts the positive cone to
by applying the algorithm $F_C$.

\begin{corollary}
\label{cor:l'algo rend la meme suite directrice}
Let $(M_n)_{n\in\N}\in\{C_1,C_2\}^\N$ be a primitive directive sequence.
Let $\bf=(f_1,f_2,f_3)\in\Delta$ be such that
    $\bigcap_{n \in \mathbb{N}} M_{[0,n)}\R^3_{\geq 0} = \R_{\geq 0}\bf$.
    Then for all $n \in \N$, $M_n = \Msf(F_C^n(\bf))$.
\end{corollary}

\begin{proof}
    Suppose on the contrary that
    $(M_n)_{n \in \mathbb{N}} \neq (\Msf(F_C^n(\mathbf{f})))_{n \in \mathbb{N}}$.
    From Corollary~\ref{cor:nonuniqueimpliesnonprimitive},
    $\dim_\Q(\bf)\leq2$ and 
$(M_n)_{n\in\N}$ is not primitive which is a contradiction.
\end{proof}

\section{Symbolic representation of $(\Delta,f_C)$}
\label{sec:symbolic repr}

In this section, we prove Theorem~\ref{maintheorem:conjugacy to shift}.
Let us first define the measure-preserving dynamical systems we are dealing with.

\subsection{Background on dynamical systems}

Let $(X,\mathcal{B}_X,\mu)$, $(Y,\mathcal{B}_Y,\nu)$ be two measured spaces.
A map $f:X \to Y$ is {\em measure-preserving} if it is measurable and satisfies $\mu(f^{-1}(B)) = \nu(B)$ for all $B \in \mathcal{B}_Y$.
If furthermore, $f$ is a bijection and $f^{-1}$ is measurable, then $f^{-1}$ is also measure-preserving.
In that case we say that $f$ is an {\em invertible measure-preserving} map.

Let $(X,\mathcal{B}_X,\mu)$ be a measured space and $(Y,\mathcal{B}_Y)$ be a measurable space.
If $\pi: X \to Y$ is measurable, then the {\em pushforward measure} on $Y$ is the measure $\pi_*\mu$ defined by $\pi_*\mu(B) = \mu(\pi^{-1}(B))$ for all $B \in \mathcal{B}_Y$.
The map $\pi: (X,\mathcal{B}_X,\mu) \to (Y,\mathcal{B}_Y,\pi_*\mu)$ is then measure-preserving.

A \emph{measure-preserving dynamical system} is a tuple $(X,T,\mathcal{B},\mu)$, where $(X,\mathcal{B},\mu)$ is a probability space and $T:X\to X$ is measure-preserving. 
We also say that the measure $\mu$ is \emph{$T$-invariant}.
It is said to be \emph{ergodic} if for every set
$B\in\mathcal{B}$, $T^{-1}B=B$ implies that $\mu(B) \in \{0,1\}$.

Two measure-preserving dynamical systems $(X,T,\mathcal{B}_X,\mu)$, $(Y,S,\mathcal{B}_Y,\nu)$ are said to be {\em isomorphic} if there exist sets $\tilde{X} \in \mathcal{B}_X,\tilde{Y} \in \mathcal{B}_Y$ of measure 1 such that $T(\tilde{X})\subset\tilde{X}$, $S(\tilde{Y})\subset\tilde{Y}$ and an invertible measure-preserving map $\pi:\tilde{X} \to \tilde{Y}$ such that $\pi \circ T(x) = S \circ \pi(x)$ for all $x \in \tilde{X}$.
Such a map $\pi$ is called an {\em isomorphism}.

In this paper, the measure-preserving dynamical system are always on a topological space $X$ and we always consider the Borel $\sigma$-algebra on it, so we simply denote them by $(X,T,\mu)$.

\subsection{Dynamical systems associated with $f_C$}

Equipping $\Delta$ with its natural Borel $\sigma$-algebra, $f_C: \Delta \to \Delta$ is measurable.
Furthermore, the measure $\xi$ defined 
for any measurable set 
$A\subset\Delta=\{(x_1,x_2,x_3)\mid x_1+x_2+x_3=1, x_1,x_2,x_3\geq0\}$
by the density function
\[
    \xi(A)=
	\frac{6}{\pi^2}
    \int_A
	\frac{1}{(1-x_1)(1-x_3)}
    dx_1dx_3
\]
is 
a $f_C$-invariant
Borel probability measure~\cite{arnoux_labbe_2017}, which makes $(\Delta,f_C,\xi)$ a measure-preserving dynamical system. The reader may confirm that it is a probability measure by computing the following integral:
    \begin{align*}
        \xi(\Delta)
            &= 
            \frac{6}{\pi^2}
            \int_0^{1}
                \int_{0}^{1-x_1}
                \frac{1}{(1-x_1) (1-x_3)} dx_3 dx_1
        = 1.
    \end{align*}
The measure $\xi$ is furthermore ergodic~\cite{fougeron_simplicity_2021} so it is the unique $f_C$-invariant probability measure which is equivalent to the Lebesgue measure on the simplex $\Delta$.

The set $\{1,2\}^\N$ is equipped with the product topology of the discrete topology on $\{1,2\}$ and we consider the associated Borel $\sigma$-algebra.
The shift map $S:\{1,2\}^\N \to \{1,2\}^\N$ defined by $S((i_n)_{n \in \N}) = (i_{n+1})_{n \in \N}$ is continuous, hence measurable.
For every $0 \leq p \leq 1$, the vector $(p_1,p_2) = (p,1-p)$ uniquely defines a Borel probability measure $\beta_p$ by 
$\beta_p([i_1 i_2\cdots i_n]) = p_{i_1} p_{i_2} \cdots p_{i_n}$, 
where 
$[i_1 i_2 \cdots i_n] = \{(j_k)_{k \in \N} \mid j_k = i_k \text{ for } 0 \leq k \leq n\}$.
This measure is shift-invariant and is called a {\em Bernoulli measure}.
It is {\em positive} whenever $0 < p < 1$.
For any $p$, $(\{1,2\}^\N,S,\beta_p)$ is thus a measure-preserving dynamical system.
It is classical to show that any Bernoulli measure is ergodic. 

We now show that $(\Delta,f_C)$ and $(\{1,2\}^\N,S)$ are isomorphic (for many measures).
We consider the sets $\Delta_1, \Delta_2$ that are the restriction to $\Delta$ of $\Lambda_1$ and $\Lambda_2$, i.e.
\begin{align*}
	\Delta_1 &= \{(x_1,x_2,x_3) \in \Delta \mid x_1 \geq x_3\};	\\
	\Delta_2 &= \{(x_1,x_2,x_3) \in \Delta \mid x_1 < x_3\}.		
\end{align*}
We respectively define the maps 
$\pi:\{1,2\}^\N \to \Delta$
and
$\delta: \Delta \to \{1,2\}^\N$ 
by
\begin{align*}
	\pi((i_n)_{n \in \N}) = \bf, 
	\quad
	&
	\text{where }
    \bigcap_{n\geq0} C_{i_0}C_{i_1}\cdots C_{i_n} \R^3_{\geq0} = 				\R_{\geq0}\bf	\\
    \delta(\bf) = (i_n)_{n \in \N}, 
    \quad
    &
    \text{where }
    f_C^n(\bf) \in \Delta_{i_n} 
    \text{ for every } n.
\end{align*}
The map $\pi$ is well defined and continuous by Proposition~\ref{prop:convergence-for-C1-C2}.
The map $\delta$ is well defined because $\{\Delta_1,\Delta_2\}$ is a partition of $\Delta$.
We finally let $\Pcal$ denote the set of sequences $(i_n)_{n \in \N} \in \{1,2\}^\N$ such that $(C_{i_n})_{n \geq 0}$ is primitive and we let $\Ical$ denote the set of vectors in $\Delta$ with rationally independent entries.

\begin{proof}[\Proofof Theorem~\ref{maintheorem:conjugacy to shift}]
The map $\pi$ is measurable because it is continuous.
The map $\delta$ is also measurable because so is $f_C$ and for all $i_0 i_1 \cdots i_{n-1} \in \{1,2\}^n$,
\[
	\delta^{-1}([i_0 i_1 \cdots i_{n-1}]) 
	= 
	\bigcap_{0 \leq k < n} f_C^{-k}(\Delta_{i_k}) 
\] 
is a measurable set.
Therefore, for every measure $\mu$ on $\{1,2\}^\N$ and every measure $\nu$ on $\Delta$, the maps $\pi:(\{1,2\}^\N,\mu) \to (\Delta,\pi_*\mu)$ and $\delta: (\Delta,\nu) \to (\{1,2\}^\N,\delta_*\nu)$ are measure-preserving.

Iterating the map $f_C$ shows that $\pi$ is surjective: any $\bf \in \Delta$ satisfies
\[
	\bf \in \bigcap_{n \geq 0} C_{i_0} C_{i_1} \cdots C_{i_n} \mathbb{R}^3_{\geq 0},
\]
where $(i_n)_{n \in \N} = \delta(\bf)$.
In other words, we have $\bf = \pi (\delta(\bf))$.
Theorem~\ref{thm:dim-sur-Q} and Corollaries~\ref{cor:nonuniqueimpliesnonprimitive} and~\ref{cor:l'algo rend la meme suite directrice} then imply that respectively restricted to $\Pcal$ and $\Ical$, the maps $\pi$ and $\delta$ are bijections that are the inverse of each other.
To conclude the proof, it suffices to observe that for all $\bf \in \Delta$, we have $\delta \circ f_C(\bf) = S \circ \delta(\bf)$.
\end{proof}

\begin{corollary}
The systems $(\Delta,f_C,\xi)$ and $(\{1,2\}^\N,S,\delta_*\xi)$ are isomorphic.
For any positive Bernoulli measure $\beta$, the systems $(\{1,2\}^\N,S,\beta)$ and $(\Delta,f_C,\pi_*\beta)$ are isomorphic. 
\end{corollary}
\begin{proof}
The first part follows from the fact that, $\xi$ being equivalent to the Lebesgue measure, $\xi(\Ical) = 1$.
For the second part, it is well known that any Bernoulli measure $\beta$ is ergodic. 
If $\beta$ is positive, then $\beta([121])$ is positive and, by ergodicity, we get for all $m \in \N$, $\beta(\bigcup_{n \geq m} S^{-n}[121]) = 1$ and so $\beta(\bigcap_{m \in \N} \bigcup_{n \geq m} S^{-n}[121]) = 1$.
By Proposition~\ref{prop:primitiveness for C adic}, we have $\bigcap_{m \in \N}\bigcup_{n \geq m} S^{-n}[121] \subset \Pcal$, hence $\beta(\Pcal)=1$.
\end{proof}

In what follows, we consider measures $\mu$ on $\{1,2\}^\N$ to obtain results for $\mu$-almost directive sequences $(i_n)_{n \in \N}$.
However our main goal is to deal with sequences of matrices $(C_{i_n})_{n \in \N}$.
To alleviate notation, we will transfer the measures $\mu$ on $\{C_1,C_2\}^\N$ and speak about $\mu$-almost every sequences $(M_n)_{n \in \N} \in \{C_1,C_2\}^\N$.

\begin{remark}\label{rem:not-bernouilli}
    Observe that $\delta_*\xi=\pi^{-1}_*\xi$ is not a Bernoulli measure
    since $\delta_*\xi([11])\neq\delta_*\xi([1])^2$. Indeed,
    \begin{align*}
        \delta_*\xi([1]) = \delta_*\xi([2]) 
            &= 
            \frac{6}{\pi^2}
            \int_0^{\frac{1}{2}}
                \int_{x_1}^{1-x_1}
                \frac{1}{(1-x_1) (1-x_3)} dx_3 dx_1
        = \frac{1}{2}\\
        \delta_*\xi([11]) = \delta_*\xi([22]) 
            &= 
            \frac{6}{\pi^2}
            \int_{\frac{1}{2}}^1
                \int_{0}^{1-x_1}
                \frac{1}{(1-x_1) (1-x_3)} dx_3 dx_1
            = \frac{1}{2}-\frac{3 \log^2(2)}{\pi^2}
            \approx 0.3540\\
        \delta_*\xi([12]) = \delta_*\xi([21])
            &= 
            \frac{6}{\pi^2}
            \int_0^{\frac{1}{2}}
                \int_{0}^{x_1}
                \frac{1}{(1-x_1) (1-x_3)} dx_3 dx_1
            = \frac{3 \log^2(2)}{\pi^2}
            \approx 0.1460
    \end{align*}
    Since the measure $\delta_*\xi$ and Bernoulli measures on $\{1,2\}^\N$ are
    ergodic and shift-invariant, they are pairwise mutually singular. 
\end{remark}

\section{Word frequencies}
\label{sec:word frequencies}

In this section, we come back to Theorem~\ref{thm:berthe-delecroix_convergence} that motivated the study made in the previous sections and we prove the following result.

\begin{proposition}\label{prop:unif frequencies}
    Every $\C$-adic word $\bw = \lim_{n \to +\infty} \sigma_{[0,n)}(1^\omega)$, $(\sigma_n)_{n \in \N} \in \C^\N$, has uniform word frequencies.
    In particular, if $\bf \in \R_{\geq 0}^d$ is such that $\|\bf\|_1 = 1$ and 
\begin{equation}
	\bigcap_{n \geq 0} M_{\sigma_{[0,n)}} \R_{\geq 0}^3 = \R_{\geq 0} \bf,
\end{equation}
then $\bf$ is the vector of letter frequencies of $\bw$.
\end{proposition}

\begin{proof}
Let $(M_n)_{n \geq 0}$ be the sequence of incidence matrices associated with $(\sigma_n)_{n \in \N}$.
By Proposition~\ref{prop:convergence-for-C1-C2}, there is a vector $\bf$ such that $\|\bf\|_1=1$ and 
\[
	\bigcap_{n \geq 0} M_{[0,n)} \R_{\geq 0}^3 = \R_{\geq 0} \bf.
\]
By Corollary~\ref{cor:everywhere growing}, every $\C$-adic word either is ultimately periodic, or has an everywhere growing directive sequence $(\sigma_n)_{n \in \N} \in \C^\N$.
In the latter case, it directly follows from Theorem~\ref{thm:berthe-delecroix_convergence} that $\bf$ is the vector of letter frequencies.

Let us now assume that $\bw$ is ultimately periodic, hence that $(\sigma_n)_{n \in \N}$ is not everywhere growing.
Then $\bw$ has uniform word frequencies and it remains to show that $\bf$ is the vector of letter frequencies.
By Lemma~\ref{lemma:existence of limit}, there is an integer $N \geq 0$ such that one of the following situation happens:
\begin{enumerate}
\item
$\sigma_n=c_1$ for all $n \geq N$ and $\bw = \sigma_{[0,N)}(1^\omega)$.
\item
$\sigma_n=c_2$ for all $n \geq N$ and $\bw = \sigma_{[0,N')}(13^\omega)$ for some integer $N' \geq N$.
\end{enumerate}
Using Lemma~\ref{lem:quand-dimQ=1}, we deduce that in the first case (resp., second case), $\bf$ is the normed vector proportional to $M_{[0,N)}e_1$ (resp., to $M_{[0,N')}e_3$) and this indeed corresponds to the vector of letter frequencies of $\bw$.
\end{proof}

\section{Balance property}
\label{sec:balance}

A word $\bw \in A^\N$ is said to be \emph{$C$-balanced} or \emph{balanced} if there exists a constant $C>0$ such that for all words $u, v \in A^*$ of the same length and occurring in $\bw$, and for every letter $i\in A$, $||u|_i-|v|_i|\leq C$.
In this article, the notion of balance is thus more inclusive than what is
normally used in the context of Sturmian sequences \cite{MR0000745} where
balanced sequences refer here to $1$-balanced sequences.
Assuming that an $\Scal$-adic word has uniform word frequencies, a sufficient condition for finite balance can be expressed using the incidence matrices of the directive sequence.

\begin{theorem} {\rm\cite[Theorem 5.8]{MR3330561}}\label{thm:berthe-delecroix_balance}
    Let $(\sigma_n)_{n\in\N}$ be the directive sequence of an $\Scal$-adic representation of a word $\bw$. For each $n$, let $M_n$ be the incidence matrix of $\sigma_n$. Assume that $\bw$ has uniform letter frequencies and let $\bf$ be the letter frequencies vector. If
\begin{equation}\label{eq:conditionforbeingbalance}
    \sum_{n\geq 0}\left\Vert\transpose{M_{[0,n)}}\middle|_{\bf^\perp}\right\Vert\cdot\Vert M_n\Vert < \infty
\end{equation}
for some norm $\Vert\cdot\Vert$, then the word $\bw$ is balanced.
\end{theorem}
Note that if the substitutions $\sigma_n$ belong to a finite set, then the
norms $\Vert M_n\Vert$ are uniformly bounded and can be removed from the sum.

Therefore we want to show that
\[
    \left\Vert\transpose{M_{[0,n)}}\middle|_{\bf^\perp}\right\Vert
    =
    \sup_{\bv\in \bf^\perp\setminus\{0\}}\frac{\Vert \transpose{M_{[0,n)}} \bv\Vert}{\Vert \bv\Vert}.
\]
converges to 0 as $n$ goes to infinity fast enough so that the sum at
Equation~\eqref{eq:conditionforbeingbalance} converges.
We achieve this in the current section using the semi-norm $\Vert\cdot\Vert_D$ defined earlier.

The strategy that we use is inspired by Lemma 6 from Avila and
Delecroix \cite{MR4043208} which provides sufficient conditions
so that the second Lyapunov exponent is negative and so that the associated words are balanced~\cite[Theorem 6.4]{MR3330561}.
We already proved in Lemma~\ref{lem:C1C2kC1-upperbound-1}
that
$\left\Vert \transpose{(C_1C_2^{n}C_1)}  \right\Vert_D^{\R^3_{\geq 0}} = 1$
and
$\left\Vert \transpose{(C_2C_1^{n}C_2)}  \right\Vert_D^{\R^3_{\geq0} } = 1$
for every $n\in\N$.
Below, we prove the existence of a matrix in the monoid generated by $C_1$ and
$C_2$ that is contracting for the semi-norm $\Vert\cdot\Vert_D$
(Lemma~\ref{lem:C1C2C1C2C1C2-contracts-4over5}).
This allows to provide an upper-bound for
$\left\Vert\transpose{M_{[0,n)}}\middle|_{\bf^\perp}\right\Vert$ in
Lemma~\ref{lem:upper-bound-4over5}.
and prove part 1 of Theorem~\ref{maintheorem:combinatoire almost always}.

\begin{lemma}\label{lem:C1C2C1C2C1C2-contracts-4over5}
    If $M = (C_1C_2)^3$ or if $M=(C_2C_1)^3$, then
    \[
        \left\Vert \transpose{M} \right\Vert_D^{M\R^3_{\geq0}} \leq \frac{4}{5}.
    \]
\end{lemma}

\begin{figure}
\begin{center}
    \includegraphics[width=\linewidth]{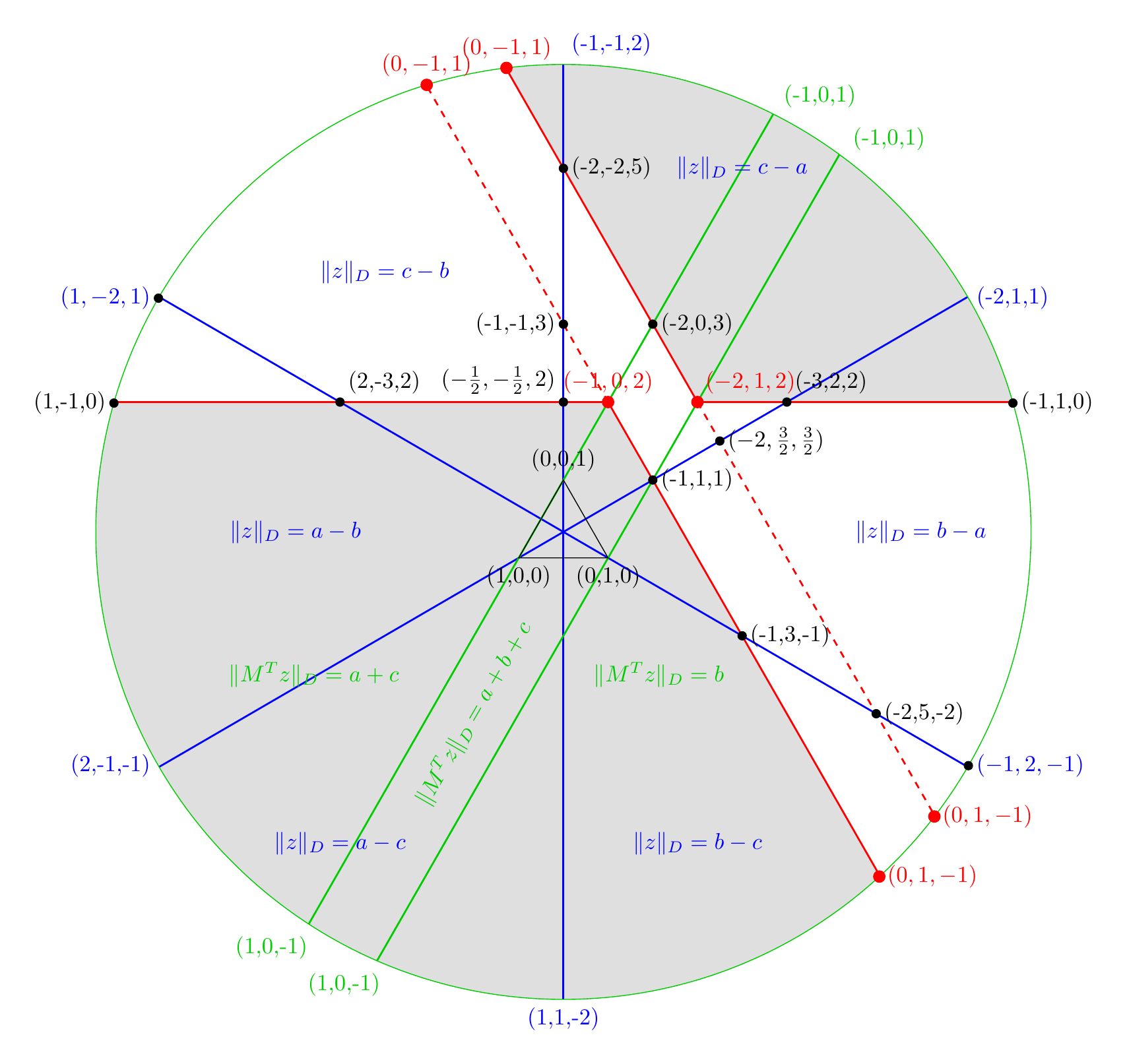}
\end{center}
\caption{An illustration on the plane $a+b+c=1$ 
of the proof that $M=(C_1C_2)^3$ is contracting on each subcone.
The hyperplanes associated with vectors in $M\E$, $\E-\E$ and $M(\E-\E)$ are
respectively drawn in red, blue and green. The green circle illustrates the points at infinity on the projective space.
    The grey regions represent the vectors $z\in
    \pm(\transposeENV{M} )^{-1}\,\R^3_{>0}$ that we do not need to consider for the maximum.
    The points $z$ in 
    $\mathcal{D}\setminus\pm (\transposeENV{M} )^{-1}\R^3_{>0}$
listed in Table~\ref{table:all-36-candidates}
are shown outside and on the boundary of the grey region.
    The maximum of $\frac{\Vert \transposeENV{M} z\Vert_D}{\Vert z\Vert_D}$ is
    attained at $z=(2,-3,2)$ with a value of $4/5$.}
\label{fig:proof-contracting}
\end{figure}

\begin{proof}
    Assume that $M= (C_1C_2)^3$ (the other case is symmetric).
    We have
    \[
    M=(C_1C_2)^3 =
\left(\begin{array}{ccc}
2 & 3 & 2 \\
2 & 2 & 1 \\
1 & 2 & 1
\end{array}\right)
\quad
\text{ and }
\quad
    M^{-1} =
\left(\begin{array}{ccc}
0 & 1 & -1 \\
-1 & 0 & 2 \\
2 & -1 & -2
\end{array}\right).
    \]
If $z=(a,b,c)$, we have
    \[
        \left\Vert
        \transpose{M} z
        \right\Vert_D 
        = \left\Vert
{\arraycolsep=2pt
\left(\begin{array}{ccc}
2 & 2 & 1 \\
3 & 2 & 2 \\
2 & 1 & 1
\end{array}\right)}z
        \right\Vert_D = 
        \left\Vert
{\arraycolsep=2pt
\left(\begin{array}{ccc}
0 & 0 & 0 \\
1 & 0 & 1 \\
0 & -1 & 0
\end{array}\right)}
\left(\begin{array}{c}
a \\
b \\
c 
\end{array}\right)
        \right\Vert_D = 
        \left\Vert
{\arraycolsep=2pt
\left(\begin{array}{c}
0 \\
a+c \\
-b 
\end{array}\right)}
        \right\Vert_D.
        \]

    Recall from Equation~\eqref{eq:hyperplanes-vectors}
    that $\mathcal{H}$ is the set of hyperplanes orthogonal to some vectors in
$S = \left(M \E \cup (\E-\E) \cup M(\E-\E)\right)\setminus\{0\}$ where $\E=
\{e_i\mid1\leq i\leq 3\}$ 
We compute:
	\[
    \begin{array}{lll}
    Me_1 =      (2,2,1),\qquad& 
    (e_1-e_3)=  (1,0,-1),& 
    M(e_1-e_2)= (-1,0,-1),\\
    Me_2 =      (3,2,2),& 
    (e_1-e_2)=  (1,-1,0),\qquad& 
    M(e_2-e_3)= (1,1,1),\\
    Me_3 =      (2,1,1),& 
    (e_2-e_3)=  (0,1,-1),& 
    M(e_1-e_3)= (0,1,0).\\
    \end{array}
    \]

    Thus $\mathcal{H}$ consists of 9 distinct hyperplanes.
    The hyperplanes and cones delimited by them are illustrated in
    Figure~\ref{fig:proof-contracting} on the plane $a+b+c=1$
    where the excluded region $\pm(\transpose{M} )^{-1}\R^3_{>0}$ is shown in grey.
    The set $\mathcal{D}$ is the finite union of intersections of two distinct
    hyperplanes of $\mathcal{H}$. 
    It contains at most $\frac{9\cdot 8}{2}=36$ distinct
    lines (actually, 26 distinct lines) 
    passing through the origin whose directions $z$ are listed in
	Table~\ref{table:all-36-candidates}.

\begin{table}
    \tabcolsep=6pt
    \input{36_cases_table.tex}
\caption{Table of values for each 
    of the 36 vectors in $\mathcal{D}$.
    The vector in
    $\mathcal{D}\setminus\pm (\transposeENV{M})^{-1}\R^3_{>0}$
    which maximizes the ratio
    $\frac{\Vert \transposeENV{M} z\Vert_D}{\Vert z\Vert_D}$
    is $z=(-2,3,-2)$ with a value of $4/5$.}
\label{table:all-36-candidates}
\end{table}

    From Lemma~\ref{lem:sup-on-boundaries}, it is sufficient to consider
    vectors $z\in\mathcal{D}\setminus\pm\transpose{M}^{-1}\R^3_{>0}$.
    Exactly 13 of those lines belong to
    $\pm \transpose{M}^{-1}\R^3_{>0}$ (the grey region) and are excluded from
    the search of the optimal value.
    We have that
    $\mathcal{D}\setminus\pm (\transpose{M} )^{-1}\R^3_{>0}$
    is defined by 23 vectors 
    (some of them defining the same line).
    For each of them, we compute the respective values and norm in
	Table~\ref{table:all-36-candidates}.
    The maximum of $\frac{\Vert \transpose{M} z\Vert_D}{\Vert z\Vert_D}$ is
    attained at $z=(-2,3,-2)$ with a value of $4/5$. The conclusion follows.
    An alternative representation where the subcone $\transpose{M}^{-1}\R^3_{>0}$ is
    bounded (the central region of a Venn diagram) is shown in
    Figure~\ref{fig:venn-diagram}.
\end{proof}

In Lemma~\ref{lem:C1C2kC1-upperbound-1}, we proved that $C_1C_2^nC_1$ and
$C_2C_1^nC_2$ are neutral for the semi-norm $\Vert\cdot\Vert_D$ and in
Lemma~\ref{lem:C1C2C1C2C1C2-contracts-4over5}, we proved that
$(C_1C_2C_1)(C_2C_1C_2)$ and $(C_2C_1C_2)(C_1C_2C_1)$ are contracting for the
semi-norm $\Vert\cdot\Vert_D$. Therefore, it is natural to consider the
acceleration of the algorithm on the monoid generated by $C_2C_1^nC_2$ and
$C_1C_2^nC_1$ (see Figure~\ref{fig:cassaigne_accelerated}).

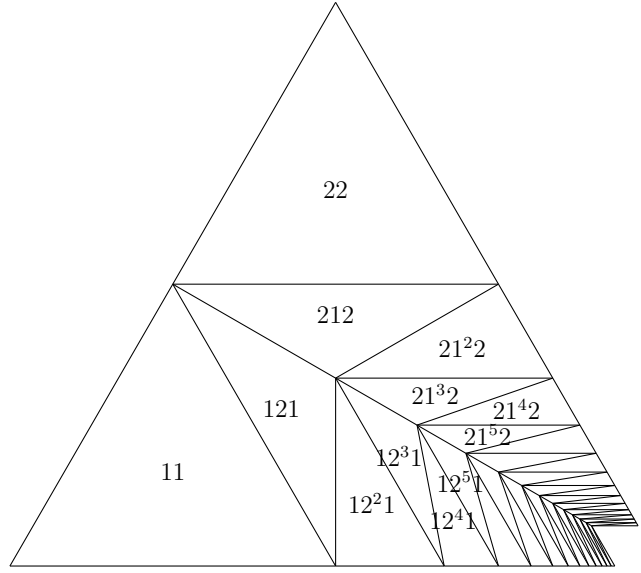
\begin{figure}[h]
\begin{center}
\input{cassaigne_accelerated_ordre25.tikz}
\end{center}
    \caption{The partition associated with the acceleration of the
    algorithm.}
    \label{fig:cassaigne_accelerated}
\end{figure}

Note that
\[
C_1 C_2^{2k}C_1=
\left(\begin{array}{ccc}
1 & 1 & 1 \\
k & k+1 & k \\
0 & 0 & 1
\end{array}\right)
\quad
\text{ and }
\quad
C_1 C_2^{2k+1}C_1=
\left(\begin{array}{ccc}
1 & 1 & 1 \\
k & k+1 & k+1 \\
1 & 1 & 0
\end{array}\right)
\]
and
\[
C_2 C_1^{2k}C_2=
\left(\begin{array}{ccc}
1 & 0 & 0 \\
k & k+1 & k \\
1 & 1 & 1
\end{array}\right)
\quad
\text{ and }
\quad
C_2 C_1^{2k+1}C_2=
\left(\begin{array}{ccc}
0 & 1 & 1 \\
k+1 & k+1 & k \\
1 & 1 & 1
\end{array}\right).
\]

The next lemma gives an upper-bound for the norm restricted to the
complementary plane. Its proof follows the line of the proof of Lemma 6
from Avila and Delecroix \cite{MR4043208} that they applied
for {Brun} and fully subtractive algorithms.

\begin{lemma}\label{lem:upper-bound-4over5}
Let $\mu$ be a shift-invariant ergodic measure on $\{1,2\}^\mathbb{N}$.
For every $\varepsilon>0$, there exists $N$ such that for every $n>N$
and $\mu$-almost all sequences $(M_n)_{n \in \mathbb{N}} \in \{C_1,C_2\}^\mathbb{N}$, 
we have
    \[
        \left\Vert 
        \transpose{M_{[0,n)}} \middle|_{\bf^\perp} 
        \right\Vert_\infty
        \leq (n+3)
        \left(\frac{4}{5}\right)^{
            \frac{1}{8} n (\mu([12121212])-\varepsilon)-\frac{1}{8}},
     \]
where $\bigcap_{n \in \mathbb{N}} M_{[0,n)} \mathbb{R}_{\geq0}^3 = \mathbb{R}_{\geq0} \bf$. 
\end{lemma}

\begin{proof}
First consider the case where $\mu([2]) = 0$, the case $\mu([1]) = 0$ is symmetric.
Then the measure $\mu$ is the Dirac measure concentrated on the sequence $1^\omega$, hence $\mu([12121212])=0$ and 
$
    (n+3) \left(\frac{4}{5}\right)^{
    \frac{1}{8} n (\mu([12121212])-\varepsilon)-\frac{1}{8}} \geq 1
$ 
for all $n \geq 1$ and all $\varepsilon >0$.
Then $\mu$-almost surely $M_n=C_1$ for all $n$, and by Lemma~\ref{lem:quand-dimQ=1}, we then have $\bf = (1,0,0)$.
For all $n$, we have 
\[
\transpose{C_1^{2n}}
=
\left( 
\begin{array}{ccc}
1 & 0 & 0 	\\ 
n & 1 & 0	\\
n & 0 & 1
\end{array}
\right)
\quad \text{and} \quad
\transpose{C_1^{2n+1}}
=
\left( 
\begin{array}{ccc}
1 & 0 & 0 	\\ 
n+1 & 0 & 1	\\
n & 1 & 0
\end{array}
\right),
\]
which implies that 
$
	\left\Vert 
    \transpose{C_1^n} \middle|_{\bf^\perp} 
    \right\Vert_\infty
    = 1
$
for all $n$.

Assume now that $\mu([1])$ and $\mu([2])$ are positive.
By ergodicity of $\mu$, $\mu$-almost every sequence $(M_n)_{n \in \N}$ contains infinitely many occurrences of $C_1$ and of $C_2$.
By Lemma~\ref{lem:primitive-implies-nice-factorization}, there is an increasing sequence $(n_i)_{i\in\N}$ such that
$n_0=0$ and 
    \[
        A_i = M_{[n_i,n_{i+1})} \in \{C_1C_2^kC_1,C_2C_1^kC_2\mid k\in\N\}
    \]
for all $i$.
    For all $n\in\N_{>0}$, there exists a unique $m\in\N$ such that $n_m\leq n-1 < n_{m+1}$.
    Let $\mathbf{g}={M_{[0,n_m)}^{-1}\bf}$, then 
    using Lemma~\ref{lem:semi-norm-on-AB},
    we get
    \begin{align*}
        \left\Vert \transpose{M_{[0,n)}} \middle|_{\bf^\perp} \right\Vert_\infty
        &\leq\left\Vert \transpose{M_{[0,n)}} \right\Vert_\infty^{M_{[0,n)}\R^3_{\geq0}}\\
        &\leq \left\Vert\transpose{M_{[n_m,n)}} \right\Vert_\infty^{M_{[n_m,n)}\R^3_{\geq0}}
             \cdot \left\Vert\transpose{M_{[0,n_m)}}\right\Vert_\infty^{M_{[0,n_m)}\R^3_{\geq0}} \\
        &\leq \left\Vert\transpose{M_{[n_m,n)}} \right\Vert_\infty 
             \cdot \left\Vert\transpose{M_{[0,n_m)}}\right\Vert_\infty^{M_{[0,n_m)}\R^3_{\geq0}} 
    \end{align*}

Remark that $M_{[n_m,n)}$ is of the form 
\begin{align*}
    &C_1C_2^{2k} =
\left(\begin{array}{ccc}
1 & 1 & 0 \\
k & k & 1 \\
0 & 1 & 0
\end{array}\right)
\quad
\text{ or }
\quad
C_1C_2^{2k+1} =
\left(\begin{array}{ccc}
1 & 1 & 0 \\
k & k+1 & 1 \\
1 & 0 & 0
\end{array}\right)
\quad
\text{ or }\\
    &C_2C_1^{2k} =
\left(\begin{array}{ccc}
0 & 1 & 0 \\
1 & k & k \\
0 & 1 & 1
\end{array}\right)
\quad
\text{ or }
\quad
C_2C_1^{2k+1} =
\left(\begin{array}{ccc}
0 & 0 & 1 \\
1 & k+1 & k \\
0 & 1 & 1
\end{array}\right)
\end{align*}
for some $k\in\N$. 
Moreover
\[
\left\Vert \transpose{(C_1C_2^{2k}  )}\right\Vert_\infty = 
\left\Vert \transpose{(C_1C_2^{2k+1})}\right\Vert_\infty =
\left\Vert \transpose{(C_2C_1^{2k}  )}\right\Vert_\infty = 
\left\Vert \transpose{(C_2C_1^{2k+1})}\right\Vert_\infty = k+2.
\]
Therefore
    \begin{align*}
        \left\Vert \transpose{M_{[n_m,n)}} \right\Vert_\infty
        &\leq \frac{n-n_m-1}{2}+2
        \leq \frac{n-1}{2}+2
        = \frac{n+3}{2}.
    \end{align*}

Let us now focus on the term $\left\Vert\transpose{M_{[0,n_m)}}\right\Vert_\infty^{M_{[0,n_m)}\R_{\geq0}^3}$ where $M_{[0,n_m)} = \prod_{i=0}^{m-1}A_i=A_{[0,m)}$.

Let $J_m$ be the set of indices $j\in\{0,1,\dots,n_m-8\}$
such that $M_{[j,j+8)} = (C_1C_2)^4$.
    Let $J_m'\subseteq J_m$ be a subset of maximal cardinality
    such that 
\begin{equation}\label{eq:J-Jgeq8}
\min\left((J_m'-J_m')\cap\N_{>0}\right)\geq 8.
\end{equation}
Observe that $\# J_m' \geq\frac{1}{8}\# J_m$.
    If $j\in J_m'$, then 
there exists a unique $i(j)\in\N$ such that
    $n_{i(j)} \in \{j,j+1,j+2\}$
and therefore
    $A_{i(j)}A_{i(j)+1}\in \{(C_1C_2)^3,(C_2C_1)^3\}$.
    In particular if $j,j'\in J_m'$ with $j\neq j'$,
    then $|i(j')-i(j)|\geq2$
    using \eqref{eq:J-Jgeq8}.
    Let $I_m=i(J_m')=\{i(j)\mid j\in J_m'\}$.

Using Lemma~\ref{lem:semi-norm-on-AB} recursively, 
Equation~\eqref{eq:infty-vs-D},
Lemma~\ref{lem:C1C2kC1-upperbound-1}
and Lemma~\ref{lem:C1C2C1C2C1C2-contracts-4over5},
    we compute
    \begin{align*}
        \left\Vert \transpose{M_{[0,n_m)}}\right\Vert_\infty^{M_{[0,n_m)}\R_{\geq0}^3}
        &=
        \left\Vert \transpose{A_{[0,m)} } 
                \right\Vert_\infty^{A_{[0,m)} \R_{\geq0}^3}
        \leq
        2\left\Vert \transpose{A_{[0,m)} } 
                \right\Vert_D^{A_{[0,m)} \R_{\geq0}^3}\\
        &\leq
        2\prod_{i\in I_m} \left\Vert \transpose{\left(A_{i}A_{i+1}\right)}\right\Vert_D
                                               ^{A_{i}A_{i+1}\R_{\geq0}^3}
        \cdot
        \prod_{\substack{i\in \{0,1,\dots,m-1\}\\i\notin I_m,\;i\notin I_m+1}}
                           \left\Vert \transpose{A_{i}}\right\Vert_D
                                               ^{A_{i}\R_{\geq0}^3}\\
        &\leq
        2\left(\frac{4}{5}\right)^{\# I_m} \cdot 1
        \leq
        2\left(\frac{4}{5}\right)^{\frac{1}{8}\# J_m}.
    \end{align*}

Let us now conclude the proof.
From the pointwise ergodic theorem, for $\mu$-almost every $x\in\{1,2\}^\N$, we have
\[
	\lim_{n\to\infty}\frac{1}{n}\sum_{k=0}^{n-8}\chi_{[12121212]}\circ S^k(x) = \lim_{n\to\infty}\frac{1}{n}\sum_{k=0}^{n-1}\chi_{[12121212]}\circ S^k(x) = \mu([12121212]).
\]

Therefore, for $\mu$-almost every $x\in\{1,2\}^\N$ and for all $\varepsilon >0$, there exists $N$ such that for all $n>N$ we have
\[
    \left\vert\frac{1}{n}\sum_{k=0}^{n-8}\chi_{[12121212]}\circ S^k(x) - \mu([12121212])\right\vert
    <\varepsilon
\]
and we obtain
\begin{align*}
    \# J_m 
    &= \sum_{k=0}^{n_m-8}\chi_{[12121212]}\circ S^k(x) \\
    &\geq \sum_{k=0}^{n-8}\chi_{[12121212]}\circ S^k(x)-1 \\
    &> n (\mu([12121212])-\varepsilon)-1
\end{align*}
This ends the proof.
\end{proof}

\begin{proof}
    [\Proofof Theorem~\ref{maintheorem:combinatoire almost always} (part 1)]
	From Proposition~\ref{prop:unif frequencies}, for every $(i_n)_{n \in \N} \in \{1,2\}^\N$, the $\C$-adic word $\bw = \lim_{n \to +\infty} c_{i_0} c_{i_1} \cdots c_{i_n}(1)$ has uniform word frequencies and its vector of letter frequencies $\bf$ satisfies
	\[
		\bigcap_{n \in \mathbb{N}} C_{i_0} C_{i_1} \cdots C_{i_n} \mathbb{R}_{\geq 0}^3 = \mathbb{R}_{\geq 0} \bf.	
	\]
    From Lemma~\ref{lem:upper-bound-4over5},
    for every $\varepsilon>0$, 
    there exists $N$ such that 
    for $\mu$-almost all sequences 
    $(M_n)_{n \in \mathbb{N}} \in \{C_1,C_2\}^\mathbb{N}$, 
    we have
    \begin{align*}
    \sum_{n>N}
    \left\Vert
    \transpose{M_{[0,n)}}\middle|_{\bf^\perp}
    \right\Vert_\infty
    \cdot
    \Vert M_n \Vert_\infty
        &\leq\sum_{n>N}
        (n+3)
        \left(\frac{4}{5}\right)^{\frac{1}{8} n (\mu([12121212])-\varepsilon)-\frac{1}{8}}
                  \cdot 2.
    \end{align*}
    In particular, if $0<\varepsilon<\mu([12121212])$, 
    the above series converges.
    Therefore, from Theorem~\ref{thm:berthe-delecroix_balance}
    we conclude that
    for $\mu$-almost every directive sequence 
    in $\{1,2\}^\mathbb{N}$, the word $\bw$ is balanced.
\end{proof}

\section{The second Lyapunov exponent}\label{sec:lyapunov}

In this section, we prove the part of Theorem~\ref{maintheorem:combinatoire
almost always} about the second Lyapunov exponent. 
The proof follows from the
lemmas proved in Section~\ref{sec:balance}.
It is different from the one provided in \cite{MR4194166}
as it is based on the approach proposed by Avila and Delecroix
\cite{MR4043208}.
We furthermore prove the negativity of the second Lyapunov exponent
not only for Lebesgue-almost every vector of letter frequencies,
but also for $\mu$-almost every directive sequence $(i_n)_{n \in \mathbb{N}} \in \{1,2\}^\mathbb{N}$, where $\mu$ is any shift-invariant ergodic Borel probability measure on $\{1,2\}^\N$ giving a positive measure to the cylinder $[12121212]$.
Below, we follow the notations of \cite{MR3330561} and \cite{MR4043208}.
For general references on Lyapunov exponents, we refer to 
\cite{MR0240280} and \cite{MR1928529}.

Given an infinite word $\gamma\in\{1,2\}^\N$, we define
the matrices $A_n$ as
\[
    A_n(\gamma) = C_{\gamma_0} C_{\gamma_1}\dots C_{\gamma_{n-1}},
\]
for every $n\geq 0$
and we have the \emph{cocycle relation}
\[
    A_{m+n}(\gamma) = A_m(\gamma) A_n(S^m\gamma),
\]
where $S:\{1,2\}^\N\to\{1,2\}^\N$ is the shift map.
Let $\mu$ be a shift-invariant ergodic measure on $\{1,2\}^\N$.
Since the matrices $C_1$ and $C_2$ are invertible, the cocycle $A_n$ is \emph{log-integrable}, that is
\[
    \int_{\{1,2\}^\N}\log\max\left(
    \Vert A_1(\gamma)\Vert,
    \Vert A_1(\gamma)^{-1}\Vert
    \right)d\mu(\gamma) < \infty.
\]

Let $\mu$ be an ergodic probability measure on $\{1,2\}^\N$.
Since the cocycle $A_n$ is log-integrable,
the first Lyapunov exponent is the
$\mu$-almost everywhere limit
\[
        \theta_1^\mu
        =\lim_{n\to\infty} \frac{\log\Vert A_n(\gamma) \Vert}{n}.
\]
In particular the first Lyapunov exponent measures the exponential growth rate of product of matrices $C_i$ along a $\mu$-generic sequence.
The other Lyapunov exponents $\theta_2^{\mu} \geq \theta_3^{\mu}$ may be
defined by the almost everywhere limits
\begin{align*}
        \theta_1^\mu +
        \theta_2^\mu
        &=\lim_{n\to\infty} \frac{\log\Vert\wedge^2 A_n(\gamma) \Vert}{n},\\
        \theta_1^\mu +
        \theta_2^\mu +
        \theta_3^\mu
        &=\lim_{n\to\infty} \frac{\log\Vert\wedge^3 A_n(\gamma) \Vert}{n} = 0
\end{align*}
where $\wedge^k$ stands for the $k$-th exterior product.
Since the sequence of nested cones $M_{[0,n)}\R^3_{\geq0}$ converges
to a line $\R_{\geq0}\bf$, there 
is a useful characterization of $\theta_2^\mu$.
The second Lyapunov exponent is the $\mu$-almost everywhere limit
\[
        \theta_2^\mu 
        = \lim_{n\to\infty} 
        \frac{\log\left\Vert\transpose{A_n(\gamma)}
              \middle|_{\bf^\perp}\right\Vert}{n},        
\]
where $\bf$ is the vector $\pi(\gamma)$, see Equation~(6.1) from \cite{MR3330561}.
Observe that the limits do not depend on the chosen norm.

In \cite[p. 1522]{MR1156412} and \cite{labbe_3-dimensional_2015}, 
approximations of
the first and second Lyapunov exponents $\theta_1^\mu$ and $\theta_2^\mu$ of Selmer and
$f_C$ algorithms were computed
where $\mu=\pi_*^{-1}(\xi)$ and $\xi$ is the $f_C$-invariant measure on $\Delta$ which is absolutely continuous with respect to the Lebesgue measure.
The values are summarized in the table below:
    \[
\begin{array}{ccccc}
    \text{Algorithm} & \theta_1 & \theta_2 & 1-\theta_2/\theta_1  & \text{Source} \\
    \hline
    F_S & \log(1.200)\approx 0.182 & \log(0.9318)\approx -0.0706 & \approx 1.387 & \text{Baldwin \cite[p. 1522]{MR1156412}}\\
    F_S & \approx 0.18269 & \approx -0.07072 & \approx 1.38710  & \text{Labbé \cite{labbe_3-dimensional_2015}} \\
    F_C & \approx 0.18268 & \approx -0.07072 & \approx 1.38709  & \text{Labbé \cite{labbe_3-dimensional_2015}} \\
\end{array}
    \]
Therefore, the above experiments suggest that for almost every
$\bx\in\Delta$, the associated second Lyapunov is negative.
We prove the negativity of the second Lyapunov exponent for the
cocycle associated with matrices in $\{C_1, C_2\}$ below.

\begin{proof}
    [\Proofof Theorem~\ref{maintheorem:combinatoire almost always} (part 2)]
    Let $\gamma\in \{1,2\}^\mathbb{N}$.
    According to Proposition~\ref{prop:convergence-for-C1-C2},
    the sequence of nested cones $A_n(\gamma)\R^3_{\geq0}$
    converges to a line $\R_{\geq0}\bf$
    for some $\bf\in\Delta$.
    From Lemma~\ref{lem:upper-bound-4over5},
    for every $\varepsilon>0$
    and $\mu$-almost all sequences 
    $\gamma\in\{1,2\}^\mathbb{N}$, 
    we have
    \begin{align*}
    \theta_2^\mu 
        &= \lim_{n\to\infty} 
        \frac{\log\left\Vert\transpose{A_n(\gamma)}
              \middle|_{\bf^\perp}\right\Vert}{n}
        \leq 
        \lim_{n\to\infty}
        \frac{\log\left(
        (n+3)
        \left(\frac{4}{5}\right)^{
            \frac{1}{8} n (\mu([12121212])-\varepsilon)-\frac{1}{8}}
              \right)}{n}\\
        &=
        \lim_{n\to\infty}
        \frac{\log(n+3)+
        \left(\frac{1}{8} n (\mu([12121212])-\varepsilon)-\frac{1}{8}\right)
            \log\left( \frac{4}{5} \right)}{n}\\
        &=
        \frac{1}{8}  \left(\mu([12121212])-\varepsilon\right)
            \log\left( \frac{4}{5} \right).
    \end{align*}
    Therefore
    \[
    \theta_2^\mu 
        \leq
        \frac{1}{8}  \mu([12121212]) 
        \log\left( \frac{4}{5} \right)
        \approx 
        -0.00002633
        <0.
    \]
\end{proof}

The above upper bound is far from the one provided in \cite[Theorem
5.1]{MR4194166} where they proved using other methods that
$\theta_2^\mu<-0.052435991$.

\section{Factor complexity}
\label{sec:factor complexity}

If $\bw$ is an infinite word over some alphabet $A$, we let
$\Lcal_\bw$ denote the set of its factors, i.e., 
$\Lcal_\bw = \{u \in A^* \mid \exists i \in \N: \bw_i \cdots \bw_{i+|u|-1} = u\}$.
The {\em factor complexity} of $\bw$ is the function
\[
    p_\bw: \N \to \N, n \mapsto \#\Lcal_\bw(n) = \#(\Lcal_\bw \cap A^n).
\]
An infinite word $\bw$ is said to be {\em uniformly recurrent} if for all  $u \in \Lcal_\bw$, $u$ occurs infinitely many times in $\bw$ and the gap between two successive occurrences is bounded.
It is classical to prove that every primitive $\Scal$-adic word is uniformly recurrent.

In this section, we study the factor complexity of $\C$-adic words.
In particular, we prove the following result, which ends the proof of Theorem~\ref{maintheorem:caracterization 2n+1}.
It is worth noticing the analogy with Theorem~\ref{thm:dim-sur-Q}.

\begin{theorem}
\label{thm:factorcomplexity}
Let $\bw$ be a $\C$-adic word with directive sequence $(c_{i_n})_{n \in \N}$.
\begin{enumerate}[\rm (i)]
    \item \label{itemFC1}
    there exists $k \geq 1$ such that $p_\bw(n)=k$ for all large enough $n$
        if and only if 
        $(c_{i_n})_{n\in\N}\in\{c_1,c_2\}^*\{c_1^\N,c_2^\N\}$.
    \item \label{itemFC2}
    there exists $k \geq 1$ such that $p_\bw(n)=n+k$ for all large enough $n$ 
        if and only if
        $(c_{i_n})_{n\in\N}\in \left(\{c_1,c_2\}^*\{c_1^2,c_2^2\}^\N\right) 
        \setminus \{c_1,c_2\}^*\{c_1^\N,c_2^\N\}$.
    \item \label{itemFC3}
    $p_\bw(n) = 2n+1$ for all $n$ 
        if and only 
        if $(c_{i_n})_{n\in\N}$ is primitive.
        In particular, this is also equivalent to the fact that $\bw$ is a uniformly recurrent dendric word (see Section~\ref{subsection:bispecials} for the definition).
\end{enumerate}
\end{theorem}

The proof essentially consists in studying the bispecial factors of $\C$-adic words.

\subsection{Bispecial factors and extension sets}
\label{subsection:bispecials}

For every infinite word $\bw \in A^\N$ and every factor $u \in \Lcal_\bw$, we set
\begin{align*}
    E^-(u,\bw) &= \{a \in A \mid a u \in \Lcal_\bw\};	\\
    E^+(u,\bw) &= \{b \in A \mid u b \in \Lcal_\bw\};	\\
    E  (u,\bw) &= \{(a,b) \in A \times A \mid a u b \in \Lcal_\bw\}.
\end{align*}
The set $E(u,\bw)$ is called the {\em extension set} of $u$ in $\bw$.
We represent it by an array of the form
\[
E(u,\bw)\quad=\quad
\begin{array}{c|ccc}
    & \cdots	 	&  j 	& \cdots			\\[-1pt]
\hline
    \vdots	&  &       &       		\\[-1pt]
i	& &		\times					\\[-1pt]
\vdots	&        &			& 			
\end{array},
\]
where a symbol $\times$ in position $(i,j)$ means that $(i,j)$ belongs to $E(u,\bw)$.

The elements of $E^-(u,\bw)$, $E^+(u,\bw)$ and $E(u,\bw)$ are respectively called the {\em left extensions}, the {\em right extensions} and the {\em biextensions} of $u$ in $\bw$. 
When the context is clear, we will omit the information on $\bw$ and simply write $E^-(u)$, $E^+(u)$ and $E(u)$.
The word $u$ is said to be {\em left special} if $\#E^-(u)>1$, {\em right special} if $\#E^+(u)>1$ and {\em bispecial} if it is both left special and right special.

The factor complexity of an infinite word is completely governed by the biextensions of its bispecial factors~\cite{MR2759107}.
In particular, we have the following result.

\begin{proposition}{\rm\cite[Proposition 4.5.3]{MR2759107}}
Let $\bw \in A^\N$ be an infinite word. 
If for every bispecial factor $u$, one has 
\begin{equation}
\label{eq:bilateral mult}
	\#E(u) - \# E^-(u) - \# E^+(u) + 1 = 0,
\end{equation}
then $p_\bw(n) = (p_\bw(1)-1)n +1$ for every $n$. 
\end{proposition}

Equation~\eqref{eq:bilateral mult} is in particular satisfied when there exists $(a,b) \in E(u)$ such that $E(u) \subset (\{a\} \times A) \cup (A \times \{b\})$.
Such a bispecial factor is said to be {\em ordinary}.
On our tabular representation, this means that the biextensions form a cross as follows:
\[
E(u)\quad=\quad
\begin{array}{c|ccccc}
    & & \cdots	 	&  b 	& \cdots &			\\[-1pt]
\hline
	& & & \times & & \\
\vdots	& & & \vdots & & \\
a	& \times & \cdots &		\times	& \cdots \times				\\[-1pt]
\vdots	& & & \vdots & & \\
	& & & \times & & 
\end{array}.
\]

Another representation of the extension set of a word $u \in \Lcal_\bw$ is given by the {\em extension graph} of $u$.
It is the undirected bipartite graph whose set of vertices is the disjoint union of $E^-(u)$ and $E^+(u)$ and whose set of edges is $E(u)$.
A bispecial factor $u$ is said to be {\em dendric} whenever its extension graph is a tree.
Dendric bispecial factors thus also satisfy Equation~\eqref{eq:bilateral mult}.
Infinite words for which all bispecial factors are dendric are also called {\em dendric} and were recently introduced under the name of tree sets~\cite{MR3320917}.
This family of words contains Arnoux-Rauzy words, codings of regular interval exchange and more~\cite{GLL21}.

\subsection{Bispecial factors in $\C$-adic words}

In this section, we give a detailed description of the extension sets of bispecial factors in $\C$-adic words.
To simplify proofs, we consider 
$\mathcal{C}' = \{c_{11},c_{22},c_{122},c_{211},c_{121},c_{212}\}$, where
\[
\begin{array}{lll}
c_{11} = c_1^2:
\begin{cases}
1 \mapsto 1		\\
2 \mapsto 12 	\\
3 \mapsto 13
\end{cases}
&
c_{122} = c_1 c_2^2:
\begin{cases}
1 \mapsto 12	\\
2 \mapsto 132 	\\
3 \mapsto 2
\end{cases}
&
c_{121} = c_1c_2c_1:
\begin{cases}
1 \mapsto 13		\\
2 \mapsto 132 	\\
3 \mapsto 12
\end{cases}
\\ 
\\
c_{22} = c_2^2:
\begin{cases}
1 \mapsto 13	\\
2 \mapsto 23 	\\
3 \mapsto 3
\end{cases}
&
c_{211} = c_2 c_1^2:
\begin{cases}
1 \mapsto 2	\\
2 \mapsto 213 	\\
3 \mapsto 23
\end{cases}
&
c_{212} = c_2c_1c_2:
\begin{cases}
1 \mapsto 23		\\
2 \mapsto 213 	\\
3 \mapsto 13
\end{cases}
\end{array}.
\]
Every (primitive) $\mathcal{C}$-adic word is a (primitive) $\mathcal{C}'$-adic word and conversely. 
The advantage of considering the substitutions in $\mathcal{C}'$ is that they are injective and strongly (left or right) proper: a substitution $\sigma: A^* \to A^*$ is {\em left proper} (resp. {\em right proper}) {\em for the letter $\ell \in A$} if $\sigma(A) \subset \ell A^*$ (resp., $\sigma(A) \subset A^*\ell$);
it is {\em strongly left proper} (resp. {\em strongly right proper}) {\em for the letter $\ell\in A$} if it is left (resp., right) proper for the letter $\ell$ and if $\ell$ occurs only once in every image $\sigma(b)$, $b \in A$.
The next result is~\cite[Proposition 4.1 and Corollary 4.3]{GLL21} for strongly left proper morphisms, but the proof is symmetric in the strongly right proper case.

\begin{proposition}[\cite{GLL21}]
\label{prop:def antecedent and bsp ext image}
Let $\sigma:A^* \to A^*$ be an injective and strongly left (resp., right) proper substitution for the letter $\ell$ and consider $\bw,\bw' \in A^\mathbb{N}$ such that $\bw = \sigma(\bw')$.
Let finally $u$ be a word in $\mathcal{L}_\bw$ containing an occurrence of $\ell$.
There is a unique triplet $(s,v,p) \in A^* \times \mathcal{L}_{\bw'} \times A^*$ and a pair $(a,b) \in E(v,\bw')$ 
such that $u = s \sigma(v)p$ with
\begin{enumerate}
\item
$s$ a proper suffix of $\sigma(a)$ (resp., a non-empty suffix of $\sigma(a)$);
\item
$p$ a non-empty prefix of $\sigma(b)$ (resp., a proper prefix of $\sigma(a)$). 
\end{enumerate}
The bi-extensions of $u$ are then governed by those of $v$ through the relation
\begin{align}
\label{eq:extensionsofufromv}
	E(u,\bw) &= 
		\{(a',b') \in A \times A 
		\mid 
		\exists (a,b) \in E(v,\bw'): 
	 	\sigma(a) \in A^* a' s 
	 	\wedge 
	 	\sigma(b)\ell \in p b' A^* \}
	\\
	\nonumber
	\text{(resp., } 
	E(u,\bw) &= 
		\{(a',b') \in A \times A 
		\mid 
		\exists (a,b) \in E(v,\bw'): 
	 	\ell\sigma(a) \in A^* a' s 
	 	\wedge 
	 	\sigma(b) \in p b' A^* \}\text{)}.		
\end{align}
In particular, if $u$ is a bispecial factor of $\bw$, then $v$ is a bispecial factor of $\bw'$.
\end{proposition}

Let $u$ and $v$ be as in Proposition~\ref{prop:def antecedent and bsp ext image} and assume that $u$ (and so $v$) is bispecial. 
The word $v$ is called the {\em bispecial antecedent of $u$ under $\sigma$} and $u$ is said to be a {\em bispecial extended image of $v$ under $\sigma$}. 
The next result is a direct application of Proposition~\ref{prop:def antecedent and bsp ext image} for morphisms of $\C'$.
We let $\varepsilon$ denote the empty word.

\begin{corollary}
\label{cor:antecedents}
Let $\bw,\bw' \in \A^\N$ be such that $\bw = \sigma(\bw')$ for some $\sigma \in \C'$ and let $u$ be a non-empty bispecial factor of $\bw$.
Then $u$ has a bispecial antecedent $v$ under $\sigma$ and we have one of the following. 
\begin{enumerate}
\item
    If $\sigma = c_{11}$, then $u = \sigma(v) 1$.
\item
    If $\sigma = c_{22}$, then $u = 3 \sigma(v)$.
\item
    If $\sigma = c_{122}$, then $u \in 2 \sigma(v) \{1,\varepsilon\}$.
\item
    If $\sigma = c_{211}$, then $u \in \{3,\varepsilon\} \sigma(v) 2$.
\item
    If $\sigma = c_{121}$, then $u \in \{2,\varepsilon\} \sigma(v) \{1,13\}$.
\item
    If $\sigma = c_{212}$, then $u \in \{3,13\} \sigma(v) \{2,\varepsilon\}$.
\end{enumerate}
\end{corollary}
\begin{proof}
To prove that $u$ has a bispecial antecedent, it suffices to show that $u$ has an occurrence of $\ell$, where $\ell$ is the letter such that $\sigma$ is strongly (left or right) proper for $\ell$.
We prove it for strongly left proper morphisms of $\C'$, the proof being symmetric for the strongly right proper ones.

If $\sigma = c_{11}$, it is immediate to check that $22,23,32,33$ do not belong to $\Lcal_\bw$.
Therefore, the letters $2$ and $3$ are neither left special, nor right special in $\bw$.
Since $u$ is a bispecial factor of $\bw$, it follows that $1$ is the first and the last letter of $u$. 
Since the antecedent $v$ of $u$ satisfies $u = s c_{11}(v) p$, where $s$ is a proper suffix of $c_{11}(a)$ and $p$ is a non-empty proper prefix of $c_{11}(b)1$ for some letters $a,b$, we must have $s=\varepsilon$ and $p = 1$, which shows the result.

If $\sigma = c_{211}$, the same reasoning shows that the first letter of $u$ belongs to $\{2,3\}$ and that the last letter is $2$ and we conclude in a similar way.

If $\sigma = c_{121}$, we again deduce that the first letter of $u$ belongs to $\{1,2\}$ and that the last letter belongs to $\{1,3\}$.
Since $23$ does not belong to $\mathcal{L}_\bw$, the word $u$ contains an occurrence of the letter $1$ and we conclude similarly.
\end{proof}

For an $\Scal$-adic word $\bw$ with $\Scal$-adic representation $((\sigma_n)_{n \in \N},(a_n)_{n \in \N})$, we set for each $n \in \mathbb{N}$, $\bw^{(n)} = \lim_{m \to +\infty} \sigma_n \sigma_{n+1} \cdots \sigma_{m-1}(a_{m})$, provided that the limit exists.
If $\bw$ is a $\C'$-adic word with directive sequence $(\sigma_n)_{n \in \N} \in \C'^\N$, then $\bw = \sigma_0(\bw')$, with $\bw' = \bw^{(1)}$.
Since the bispecial antecedent of a non-empty bispecial word is always shorter, for any bispecial factor $u$ of $\bw$, there is a unique sequence $(u_i)_{0 \leq i \leq n}$ such that 
\begin{itemize}
\item
$u_0 = u$, $u_n = \varepsilon$ and $u_i \neq \varepsilon$ for all $i < n$;
\item
for all $i < n$, $u_{i+1} \in \Lcal_{\bw^{(i+1)}}$ is the bispecial antecedent of $u_i$ under $\sigma_i$.
\end{itemize}
The factor $u$ is called a {\em bispecial descendant} of $\varepsilon$ in $\bw^{(n)}$.

As any bispecial factor of a primitive $\mathcal{C}$-adic word is a descendant of the empty word, to understand the extension sets of any bispecial word in $\bw$, we need to know the possible extension sets of $\varepsilon$ in $\bw^{(n)}$.
We will then use Equation~\eqref{eq:extensionsofufromv} to describe the extension sets of bispecial factors.

\begin{lemma}
\label{lemma: empty word}
If $\bw$ is a primitive $\mathcal{C'}$-adic word with directive sequence
$(\sigma_n)_{n \in \N}$, then the extension set $E(\varepsilon,\bw)$ is one of the following, depending on $\sigma_0$.

\medskip

\begin{center}
    \tabcolsep=5pt
\begin{tabular}{ccc}
    \tabcolsep=5pt
 \begin{tabular}{c|ccc}
 $\sigma_0 = c_{11}$ 
 & $1$ & $2$ & $3$ \\
 \hline
 $1$ & $\times$ & $\times$ & $\times$ \\
 $2$ & $\times$ &   &   \\
 $3$ & $\times$ &   &   \\
 \end{tabular}
 &\qquad
    \tabcolsep=5pt
 \begin{tabular}{c|ccc}
 $\sigma_0 = c_{122}$ 
 & $1$ & $2$ & $3$ \\
 \hline
 $1$ &   & $\times$ & $\times$ \\
 $2$ & $\times$ & $\times$ &   \\
 $3$ &   & $\times$ &   
 \end{tabular}
 &\qquad
    \tabcolsep=5pt
 \begin{tabular}{c|ccc}
 $\sigma_0 = c_{121}$ 
 & $1$ & $2$ & $3$ \\
 \hline
 $1$ &   & $\times$ & $\times$ \\
 $2$ & $\times$ &   &   \\
 $3$ & $\times$ & $\times$ &   \\
 \end{tabular}
    \\ \\
    \tabcolsep=5pt
 \begin{tabular}{c|ccc}
 $\sigma_0 = c_{22}$ 
 & $1$ & $2$ & $3$ \\
 \hline
 $1$ &   &   & $\times$ \\
 $2$ &   &   & $\times$ \\
 $3$ & $\times$ & $\times$ & $\times$ \\
 \end{tabular}
 &\qquad
    \tabcolsep=5pt
 \begin{tabular}{c|ccc}
 $\sigma_0 = c_{211}$ 
 & $1$ & $2$ & $3$ \\
 \hline
 $1$ &   &   & $\times$ \\
 $2$ & $\times$ & $\times$ & $\times$ \\
 $3$ &   & $\times$ &   \\
 \end{tabular}
 &\qquad
    \tabcolsep=5pt
 \begin{tabular}{c|ccc}
 $\sigma_0 = c_{212}$ 
 & $1$ & $2$ & $3$ \\
 \hline
 $1$ &   &   & $\times$ \\
 $2$ & $\times$ &   & $\times$ \\
 $3$ & $\times$ & $\times$ &   \\
 \end{tabular}
 \end{tabular}
 \end{center}
\end{lemma}
\begin{proof}
The directive sequence being primitive, all letters of $\A$ occur in $\bw^{(1)}$.
The extension set of $\varepsilon$ in $\bw$ is governed by the factors of length 2. Any factor $u = u_1u_2$ of length 2 of $\bw$ either occurs in some image of letter $\sigma_0(a)$ or $u_1$ is the last letter of some image of letter and $u_2$ is the first letter of some image of letter.
The result then follows from the fact that all morphisms $\sigma$ in $\mathcal{C}'$ are either left proper ($\sigma(\mathcal{A}) \subset a \mathcal{A}^*$ for some letter $a$) or right proper ($\sigma(\mathcal{A}) \subset \mathcal{A}^* a$ for some letter $a$). 
\end{proof}

Let us now explicitely show how Equation~\eqref{eq:extensionsofufromv} allows to compute the extensions of a bispecial factor $u$ from the extensions of its bispecial antecedent.
When $\sigma$ is strongly left proper for the letter $\ell$, the extensions of $u = s \sigma(v) p$ can be obtained as follows: 
\begin{enumerate}[1)]
\item
replace any left extensions $a$ by $\sigma(a)$ and any right extension $b$ by $\sigma(b)\ell$;
\item
remove the suffix $s$ from the left extensions whenever it is possible (otherwise, delete the row) and remove the prefix $p$ from the right extensions whenever it is possible (otherwise, delete the column);
\item
keep only the last letter of the left extensions and the first letter of the right extensions;
\item
permute and merge the rows and columns with the same label. 
\end{enumerate}
The case where $\sigma$ is strongly right proper is similar.

Let us make this more clear on an example and consider the extension set 
\[
E(v) = \{(1,3),(2,1),(2,2),(2,3),(3,2)\}.
\]
This extension set corresponds to the extension set of the empty word whenever the last applied substitution is $c_{211}$ (see Lemma~\ref{lemma: empty word}).
Using Equation~\eqref{eq:extensionsofufromv}, the extension sets of $2 c_{122}(v)$ and $2 c_{121}(v)1$ are obtained as follows (arrow labels indicate above step number):
\[
    \footnotesize
    \arraycolsep=4pt
\begin{array}{ccccccc}
\begin{array}{c}
E(v)\\
\begin{array}{r|ccc}
  & 1 & 2 & 3	\\
\hline
 1 &   		&   	 & \times 	\\
 2 & \times & \times & \times 	\\
 3 &   		& \times &   
\end{array}
\end{array}
&
\xrightarrow[]{1)}
&
\begin{array}{c}
E(c_{122}(v))\\
\begin{array}{r|ccc}
 & 12 & 132 & 2	\\
\hline
 212 	&   		&   	 & \times 	\\
 2132 	& \times 	& \times & \times 	\\
 22 	&   		& \times &   
\end{array}
\end{array}
& 
\xrightarrow[]{2) \text{ and }3)} 
&
\begin{array}{c}
E(2c_{122}(v))\\
\begin{array}{r|ccc}
& 1 & 1 & 2	\\
\hline
 1 	&   		&   	 & \times 	\\
 3 	& \times 	& \times & \times 	\\
 2	&   		& \times &   
\end{array}
\end{array}
&
\xrightarrow[]{4)}
&
\begin{array}{c}
E(2c_{122}(v))\\
\begin{array}{r|cc}
& 1 & 2 \\
\hline
 1 & 		& \times 		\\
 2 & \times &   			\\
 3 & \times	& \times    
\end{array}
\end{array}
\\
\\
\begin{array}{c}
E(v)\\
\begin{array}{r|ccc}
& 1 & 2 & 3	\\
\hline
 1 &   		&   	 & \times 	\\
 2 & \times & \times & \times 	\\
 3 &   		& \times &   
\end{array}
\end{array}
&
\xrightarrow[]{1)}
&
\begin{array}{c}
E(c_{121}(v))\\
\begin{array}{r|ccc}
& 131 & 1321 & 121	\\
\hline
 13 	&   		&   	 & \times 	\\
 132 	& \times 	& \times & \times 	\\
 12 	&   		& \times &   
\end{array}
\end{array}
& 
\xrightarrow[]{2) \text{ and }3)} 
&
\begin{array}{c}
E(2c_{121}(v)1)\\
\begin{array}{r|ccc}
& 3 & 3 & 2	\\
\hline
 .	&   		&        &   \\
 3 	& \times 	& \times & \times 	\\
 1	&   		& \times &   
\end{array}
\end{array}
&
\xrightarrow[]{4)}
&
\begin{array}{c}
E(2c_{121}(v)1)\\
\begin{array}{r|cc}
& 2 & 3 \\
\hline
 1 & 		& \times 		\\
 3 & \times	& \times    
\end{array}
\end{array}
\end{array}
\]
The proof of Theorem~\ref{thm:factorcomplexity} will essentially consist in describing how ordinary bispecial words occur.
The next lemma allows to understand when bispecial words have ordinary bispecial extended images.

\begin{lemma}
\label{lemma:ordinary preserved}
Let $\bw$ be a $\mathcal{C}'$-adic word with directive sequence $(\sigma_n)_{n \in \N} \in \mathcal{C}'^\N$.
    Let $u \in \Lcal_\bw$ be a non-empty bispecial factor and $v$ be its bispecial antecedent. We have the following.
\begin{enumerate}
\item \label{item:E(v)=E(u)}
If $\sigma_0 \in \{c_{11},c_{22}\}$, then $E(u)=E(v)$;
\item \label{item:v=epsilon}
if $v = \varepsilon$ and $\sigma_0 \in \{c_{121},c_{212}\}$, then $u$ is ordinary;
\item \label{item:E(v) C12}
if $ \sigma_0 \in \{c_{122},c_{121},c_{212}\}$, if $E(v) \subseteq (\A \times \{1,2\}) \cup \{(a,3)\}$ for some letter $a \in \A$ with $E(v) \cap \{(a,1), (a,2)\} \neq\emptyset$ and if $E(v) \setminus \{(a,3)\}$ is the extension set of an ordinary bispecial word, then $u$ is ordinary;
\item \label{item:E(v) L23}
if $ \sigma_0 \in \{c_{211},c_{121},c_{212}\}$, if $E(v) \subseteq (\{2,3\} \times \A) \cup \{(1,a)\}$ for some letter $a \in \A$ with $E(v) \cap \{(2,a), (3,a)\} \neq\emptyset$ and if $E(v) \setminus \{(1,a)\}$ is the extension set of an ordinary bispecial word, then $u$ is ordinary;
\item \label{item:ordinary preserved}
if $v$ is ordinary, then $u$ is ordinary
\end{enumerate}
\end{lemma}
\begin{proof}
Items~\ref{item:E(v)=E(u)} and~\ref{item:ordinary preserved} directly follow from Corollary~\ref{cor:antecedents} and Equation~\eqref{eq:extensionsofufromv}.
 Item~\ref{item:v=epsilon} can be checked by hand using Lemma~\ref{lemma: empty word} and Equation~\eqref{eq:extensionsofufromv}.
Let us prove Item~\ref{item:E(v) C12}, Item~\ref{item:E(v) L23} being symmetric
    (the symmetry consists in applying the reserval, exchanging letters $1$ and $3$
    and exchanging $c_1$ and $c_2$).

We say that two extension sets $E$ and $E'$ are equivalent whenever there exist two permutations $p_1$ and $p_2$ of $\A$ such that $E = \{(p_1(a),p_2(b)) \mid (a,b) \in E'\}$.
If $\sigma_0 = c_{122}$, then $u \in \{2 \sigma(v), 2\sigma(v)1\}$ by Corollary~\ref{cor:antecedents}. 
We make use of Equation~\eqref{eq:extensionsofufromv}. 
If $u = 2 \sigma(v)$, then the extension set of $u$ is equivalent to the one obtained from $E(v)$ by merging the columns with labels 1 and~2. 
If $u = 2 \sigma(v)1$, then the extension set of $u$ is equivalent to the one obtained from $E(v)$ by deleting the column with label 3.
In both cases, $u$ is ordinary.

The same reasoning applies when $\sigma_0 \in \{c_{121},c_{212}\}$: depending on the word $x$ such that $u \in\A^* \sigma(v)x$, either we delete the column with label 3, or we merge the columns with labels 1 and~2.
\end{proof}

\subsection{Factor complexity of $\C$-adic words}

\begin{proof}
    [\Proofof Theorem~\ref{thm:factorcomplexity}]
    Note that the three conditions on $(c_{i_n})_{n\in\N}$ in Theorem~\ref{thm:factorcomplexity} are mutually exclusive and cover all cases, so it is enough to prove that they are sufficient.

\ref{itemFC1} 
This directly follows from Lemma~\ref{lemma:existence of limit} and from the Morse-Hedlund theorem that states that an infinite word has bounded factor complexity if and only if it is eventually periodic~\cite{MR1507944}.

\ref{itemFC2} 
Let $N$ be such that $(c_{i_n})_{n \geq N}$ is in $\{c_1^2,c_2^2\}^\N$ with $c_1$ and $c_2$ occurring infinitely many times in $(c_{i_n})_{n \in \N}$.
Thus $\bw^{(N)}$ is a Sturmian sequence over the alphabet $\{1,3\}$.
As $\sigma_{[0,N)}$ is injective, $\bw$ has factor complexity $p_\bw(n) = n+k$ for some $k \geq 1$ for all large enough $n$~\cite[Proposition 8]{Cassaigne98sequenceswith}.

\ref{itemFC3} 
The sequence $(c_{i_n})_{n \in \N}$ being primitive, the word $\bw$ is uniformly recurrent.
Let us show that $\bw$ is dendric.
Recall that this implies that Equation~\eqref{eq:bilateral mult} holds for all bispecial factors, so that $p_\bw(n)=2n+1$ for all $n$.

To show that the extension graphs of all bispecial factors are trees, we make use of Lemma~\ref{lemma:ordinary preserved}.
If $u$ is a bispecial factor of $\bw$, it is a descendant of $\varepsilon \in \Lcal_{\bw^{(n)}}$ for some $n$.
If $\sigma_n \in \{c_{11},c_{22}\}$, then from Lemma~\ref{lemma: empty word} and Lemma~\ref{lemma:ordinary preserved}, all descendants of $\varepsilon$ are ordinary. 
The extension graph of $u$ is thus a tree.

For $\sigma_n \in\{c_{122},c_{211},c_{121},c_{212}\}$, we represent the extension sets of the descendants of $\varepsilon$ in the graphs represented in Figure~\ref{figure:descendants of c_{122}} and Figure~\ref{figure:descendants of c_{121}}.
Observe that the situation is symmetric for $c_{122}$ and $c_{211}$ and for $c_{121}$ and $c_{212}$ so we only represent the graphs for $c_{122}$ and $c_{121}$.
Furthermore, in these graphs, we do not represent the extension sets of ordinary bispecial factors as the property of being ordinary is preserved by taking bispecial extended images (Lemma~\ref{lemma:ordinary preserved}).
Given an extension set of some bispecial word $v$, if $u$ is a bispecial extended image of $v$ such that $u = s \sigma(v) p$, we label the edge from $E(v)$ to $E(u)$ by $s \cdot \sigma \cdot p$.
Finally, for all $v$, we have $E(c_{11}(v)1)=E(v)$ and $E(3c_{22}(v))=E(v)$, but for the sake of clarity, we do not draw the loops labeled by $c_{11}\cdot 1$ and by $3 \cdot c_{22}$.
We conclude by observing that the extension graphs of all descendants are trees.     
\end{proof}

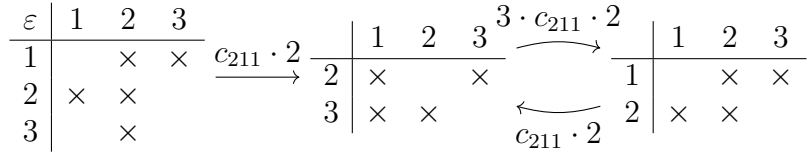
\begin{figure}[t]
\centering
\begin{tikzpicture}
\node(A) at (0,0){
$\begin{array}{c|ccc}
\varepsilon & 1 & 2 & 3 \\
 \hline
 1 &   		& \times & \times \\
 2 & \times & \times &   \\
 3 &   		& \times &  
\end{array}$
};
\node(B) at (4,0){
$\begin{array}{c|ccc}
 & 1 & 2 & 3 \\
 \hline
 2 & \times & 		 & \times  \\
 3 & \times	& \times &  
\end{array}$
};
\node(C) at (8,0){
$\begin{array}{c|ccc}
 & 1 & 2 & 3 \\
 \hline
 1 & 		& \times & \times  \\
 2 & \times	& \times &  
\end{array}$
};
\path [->] (A) edge node[above] {$c_{211}\cdot 2$} (B);
\path [->] (B) edge[bend left=15] node[above] {$3\cdot c_{211}\cdot 2$} (C);
\path [->] (C) edge[bend left=15] node[below] {$c_{211}\cdot 2$} (B);
\end{tikzpicture}
    \caption{Non-ordinary bispecial descendants of $\varepsilon \in \Lcal_{\bw^{(n)}}$ whenever $\sigma_n=c_{122}$.}
\label{figure:descendants of c_{122}}
\end{figure}

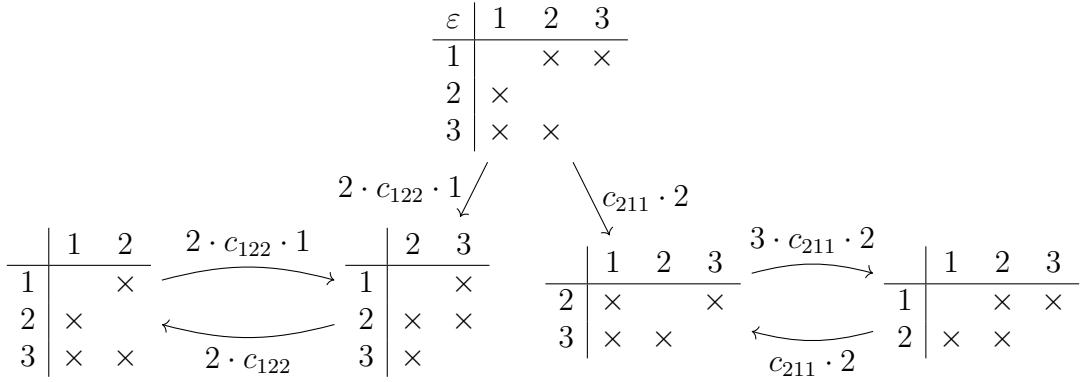
\begin{figure}[t]
\centering
\begin{tikzpicture}[scale=1.5]
\node(A) at (4,2){
$\begin{array}{c|ccc}
\varepsilon & 1 & 2 & 3 \\
 \hline
 1 &   		& \times 	& \times 	\\
 2 & \times &   		&   		\\
 3 & \times & \times 	&
\end{array}$
};
\node(B) at (3,0){
$\begin{array}{c|cc}
 & 2 & 3 \\
 \hline
 1 & 		& \times  \\
 2 & \times & \times  \\
 3 & \times	&   
\end{array}$
};
\node(B') at (0,0){
$\begin{array}{c|cc}
 & 1 & 2 \\
 \hline
 1 & 		& \times  \\
 2 & \times & 		  \\
 3 & \times	& \times  
\end{array}$
};
\node(C) at (5,0){
$\begin{array}{c|ccc}
 & 1 & 2 & 3 \\
 \hline
 2 & \times &		 & \times  \\
 3 & \times	& \times &  
\end{array}$
};
\node(C') at (8,0){
$\begin{array}{c|ccc}
 & 1 & 2 & 3 \\
 \hline
 1 & 		& \times & \times  \\
 2 & \times	& \times &  
\end{array}$
};
\path [->] (A) edge node[left] {$2\cdot c_{122}\cdot 1$} (B);
\path [->] (B) edge[bend left=15] node[below] {$2\cdot c_{122}$} (B');
\path [->] (B') edge[bend left=15] node[above] {$2 \cdot c_{122} \cdot 1$} (B);
\path [->] (A) edge node[right] {$c_{211} \cdot 2$} (C);
\path [->] (C) edge[bend left=15] node[above] {$3 \cdot c_{211} \cdot 2$} (C');
\path [->] (C') edge[bend left=15] node[below] {$c_{211} \cdot 2$} (C);
\end{tikzpicture}
    \caption{Non-ordinary bispecial descendants of $\varepsilon \in \Lcal_{\bw^{(n)}}$ whenever $\sigma_n=c_{121}$.}
\label{figure:descendants of c_{121}}
\end{figure}

\begin{proof}[\Proofof Theorem~\ref{maintheorem:caracterization 2n+1}]    
It directly follows from Theorem~\ref{thm:dim-sur-Q} and Theorem~\ref{thm:factorcomplexity}.    
\end{proof}

\begin{corollary}
Every $\C$-adic word has uniform word frequencies. 
Furthermore, if $\mu$ is a shift-invariant ergodic Borel probability measure on $\{1,2\}^\N$ satisfying $\mu([12121212])>0$, then
$\mu$-almost every $\C$-adic word is uniformly recurrent and balanced, has factor complexity $p_\bw(n)=2n+1$ for every $n$ and its vector of letter frequencies is totally irrational.
\end{corollary}

\begin{proof}
Uniform factor frequencies follows from Proposition~\ref{prop:unif frequencies}.
Since $\mu$ is ergodic and satisfies $\mu([12121212])>0$, then $\mu$-almost every directive sequence $(c_{i_n})_{n \in \N}$ is primitive. 
    The result then follows from 
    Theorem~\ref{maintheorem:caracterization 2n+1} and
    Theorem~\ref{maintheorem:combinatoire almost always}.
\end{proof}

\section{Conjugacy with a semi-sorted version of Selmer algorithm}
\label{sec:selmer}

The Selmer algorithm \cite{MR0130852,schweiger} (also called the \emph{GMA algorithm}~\cite{MR1156412}) is an algorithm which subtracts the smallest entry to the largest.
As observed in \cite{MR2413304},
Selmer algorithm is also conjugate on the absorbing simplex to
M\"onkemeyer's algorithm \cite{MR64084} which makes it an algorithm which have
been rediscovered many times.
We prove in this section that $F_C$ also belongs to this family.

As recalled in Section~\ref{sec:lyapunov}, the numerical computation of Lyapunov exponents~\cite{labbe_3-dimensional_2015} indicates that exponents for the Selmer algorithm and $F_C$ have
statistically equal values (the difference is at most $10^{-5}$). 
We confirm this observation by showing 
a relation between $F_C$ and the Selmer algorithm.
The map $F_C$ is not conjugate to the Selmer algorithm, however,
we show that $F_C$ is conjugate to a semi-sorted version of the Selmer
algorithm which keeps the largest entry at index $1$.
We also show that the application of this semi-sorted Selmer algorithm on its absorbing subset defines $\mathcal{S}$-adic subshifts that actually are images of $\C$-adic subshifts by a permutation of the alphabet.

On $\Theta=\{\bx=(x_1,x_2,x_3)\in\mathbb{R}^3_{\geq 0}\mid \max(x_2,x_3)\leq x_1\}$, the semi-sorted version of Selmer algorithm is defined by
\[
    F_S (x_1,x_2,x_3) = 
\begin{cases}
    (x_2, x_1-x_3, x_3) & \mbox{if } 
    x_1 \leq x_2+x_3\text{ and } x_2 \geq x_3,\\
    (x_3, x_2, x_1-x_2) & \mbox{if } 
    x_1 \leq x_2+x_3\text{ and } x_2 < x_3,\\
    (x_1-x_3, x_3, x_2) & \mbox{if } 
    x_1 > x_2+x_3\text{ and } x_2 \geq x_3,\\
    (x_1-x_2, x_3, x_2) & \mbox{if } 
    x_1 > x_2+x_3\text{ and } x_2 < x_3.
\end{cases}
\]
Like with $F_C$, we consider the partition 
\begin{align*}
	\Theta_1 &= \{ (x_1,x_2,x_3) \in \Theta \mid 
	x_1 \leq x_2+x_3\text{ and } x_2 \geq x_3 \}, \\
	\Theta_2 &= \{ (x_1,x_2,x_3) \in \Theta \mid 
	x_1 \leq x_2+x_3\text{ and } x_2 < x_3 \}, \\
	\Theta_3 &= \{ (x_1,x_2,x_3) \in \Theta \mid 
	x_1 > x_2+x_3\text{ and } x_2 \geq x_3 \}, \\
	\Theta_4 &= \{ (x_1,x_2,x_3) \in \Theta \mid 
	x_1 > x_2+x_3\text{ and } x_2 < x_3 \} \\	 
\end{align*}
and the matrices
\[
S_1 = \left(\begin{array}{rrr}
    0 & 1 & 1 \\
    1 & 0 & 0 \\
    0 & 0 & 1
\end{array}\right),
\quad
S_2 = \left(\begin{array}{rrr}
    0 & 1 & 1 \\
    0 & 1 & 0 \\
    1 & 0 & 0
\end{array}\right)
\quad
S_3 = \left(\begin{array}{rrr}
    1 & 1 & 0 \\
    0 & 0 & 1 \\
    0 & 1 & 0
\end{array}\right),
\quad
S_4 = \left(\begin{array}{rrr}
    1 & 0 & 1 \\
    0 & 0 & 1 \\
    0 & 1 & 0
\end{array}\right).
\]
The map $F_S$ is then defined by $F_S(\bx) = S_i^{-1}\bx$ whenever $\bx \in \Theta_i$.
The Selmer algorithm being weakly convergent~\cite{schweiger}, there is a continuous map $\pi:\{1,2,3,4\}^\N \to \Theta$ defined by 
\[
	\bigcap_{n \in \N} S_{i_0} S_{i_1} \cdots S_{i_n} \mathbb{R}^3_{\geq 0} = \mathbb{R}_{\geq 0} \pi((i_n)_{n \in \N}).
\] 
Note that if $\dim_\mathbb{Q}(\bx)=3$, then for all large enough $n$, $F_S^n(\mathbf x)$ belongs to $\Gamma = \Theta_1 \cup \Theta_2$.
Therefore, if $\mu$ is a shift-invariant ergodic measure on $\{1,2,3,4\}^\N$ such that $\pi_*\mu(\{\bx \in \Theta \mid \dim_\mathbb{Q}(\bx) = 3\})=1$, then $\mu([3]) = \mu([4]) = 0$.
To compute the Lyapunov exponents associated with such a measure, we may thus restrict the Selmer algorithm to the absorbing set $\Gamma$.
The next result shows that $F_C$ and $F_S$ (restricted to $\Gamma$) are conjugate, confirming the equality of their respective Lyapunov exponents.

\begin{proposition}
    The maps $F_C:\mathbb{R}^3_{\geq 0}\to\mathbb{R}^3_{\geq 0}$ and
    $F_S:\Gamma\to\Gamma$ 
    are conjugate, i.e.,
    there exists a linear homeomorphism $z:\mathbb{R}^3_{\geq 0}\to\Gamma$ 
    such that $z\circ F_C = F_S\circ z$.
    Furthermore, for all $\bx \in \mathbb{R}^3$, we have $\dim_\Q(\bx) = \dim_\Q(z(\bx))$.
\end{proposition}

\begin{proof}
Let $z:\mathbb{R}^3_{\geq 0}\to\Gamma$ be the homeomorphism defined by $\bx\mapsto Z\bx$ with
\[
Z= \left(\begin{array}{rrr}
    1 & 1 & 1 \\
    1 & 1 & 0 \\
    0 & 1 & 1
\end{array}\right).
\]
For $i=1,2$, we have $\bx \in \Lambda_i$ if and only if $Z \bx \in \Theta_i$ and $C_i$ is conjugate to $S_i$ through the matrix $Z$:
\[
S_1 Z=
\left(\begin{array}{rrr}
1 & 2 & 1 \\
1 & 1 & 1 \\
0 & 1 & 1
\end{array}\right)
= Z C_1
\qquad\text{and}\qquad
S_2 Z=
\left(\begin{array}{rrr}
    1 & 2 & 1 \\
    1 & 1 & 0 \\
    1 & 1 & 1
\end{array}\right)
= Z C_2.
\]
    Thus we have $S_i^{-1}z=zC_i^{-1}$ and
    $z\circ F_C = F_S\circ z$.
The equality $\dim_\Q(\bx) = \dim_\Q(z(\bx))$ directly follows from the definition of $z$.
\end{proof}

For example, orbits of the two algorithms are related like in the
following diagram:
\begin{center}
\begin{tikzpicture}[node distance=28mm,auto]
    \node             (a0) {$(3,15,22)$};
    \node[right of=a0] (a1) {$(15,3,19)$};
    \node[right of=a1] (a2) {$(3,15,4)$};
    \node[right of=a2] (a3) {$(15,3,1)$};
    \node[right of=a3] (a4) {$(14,1,3)$};
    \node[below of=a0,node distance=15mm] (b0) {$(40,18,37)$};
    \node[right of=b0] (b1) {$(37,18,22)$};
    \node[right of=b1] (b2) {$(22,18,19)$};
    \node[right of=b2] (b3) {$(19,18,4)$};
    \node[right of=b3] (b4) {$(18,15,4)$};

    \draw[->] (a0) to node {$F_C$} (a1);
    \draw[->] (a1) to node {$F_C$} (a2);
    \draw[->] (a2) to node {$F_C$} (a3);
    \draw[->] (a3) to node {$F_C$} (a4);
    \draw[->] (b0) to node {$F_S$} (b1);
    \draw[->] (b1) to node {$F_S$} (b2);
    \draw[->] (b2) to node {$F_S$} (b3);
    \draw[->] (b3) to node {$F_S$} (b4);
    \draw[->] (a0) to node {$z$} (b0);
    \draw[->] (a1) to node {$z$} (b1);
    \draw[->] (a2) to node {$z$} (b2);
    \draw[->] (a3) to node {$z$} (b3);
    \draw[->] (a4) to node {$z$} (b4);
\end{tikzpicture}
\end{center}

Like for the matrices $C_1$ and $C_2$, we associate with $S_1$ and $S_2$ the two substitutions 
\[
s_1
=
\left\{\begin{array}{l}
1 \mapsto 2\\
2 \mapsto 1\\
3 \mapsto 31
\end{array}\right.
\quad \text{and} \quad
s_2
=
\left\{\begin{array}{l}
1 \mapsto 3\\
2 \mapsto 12\\
3 \mapsto 1
\end{array}\right.,
\]
$S_i$ being the incidence matrix of $s_i$ for $i=1,2$.
Given a sequence $\boldsymbol{\sigma} = (\sigma_n)_{n \in \N} \in \{s_1,s_2\}^\N$ of substitutions and a sequence $(a_n)_{n \in \N} \in \A^\N$ of letters, the convergence of $(\sigma_{[0,n)}(a_n))_{n \in \N}$ to an infinite word is not as nicely described as with the substitutions $c_1$ and $c_2$ (see Lemma~\ref{lemma:existence of limit}).
We can however easily define the associated $\{s_1,s_2\}$-adic subshift 
\[
X_{\boldsymbol{\sigma}}
=
\{\bw \in \A^\N \mid u \in \Lcal_\bw \Rightarrow \exists a \in \A, n \in \N: u \in \Lcal_{\sigma_{[0,n)}(a)} \}.
\]
This subshift is minimal as soon as the sequence $\boldsymbol{\sigma}$ is primitive.
We will now show that such a subshift is actually the image of a $\C$-adic subshift under a permutation of the alphabet.

If $X$ is a subshift over some alphabet $A$ and if $\sigma: A^* \to A^*$ is a substitution, we define the {\em image of $X$ under $\sigma$} by
\[
	\sigma \cdot X
	=
	\left\{ S^i (\sigma(\bw)) 
    \mid 
    \bw = (w_n)_{n \in \N} \in X, 
    0 \leq i < |\sigma(w_0)| \right\}.
\]
It corresponds to the shift-orbit closure of $\sigma(X)$.

Let $z_l$ and $z_r$ be the substitutions:
\[
z_l :\left\{\begin{array}{l}
1 \mapsto 12\\
2 \mapsto 123\\
3 \mapsto 13
\end{array}\right.
\qquad\text{and}\qquad
z_r :\left\{\begin{array}{l}
1 \mapsto 21\\
2 \mapsto 231\\
3 \mapsto 31
\end{array}\right..
\]
Notice that $Z$ is the incidence matrix of both $z_l$ and $z_r$.
The substitution $z_l$ is left proper while $z_r$ is right proper.
Moreover they are conjugate through the equation
\[
    z_l(w)\cdot 1 = 1\cdot z_r(w)
\]
for every $w \in\mathcal{A}^*$.
In particular, for any word $\bw \in \A^\mathbb{N}$, we have
\[
	z_l(\bw) = 1 z_r(\bw) 
	\quad \text{and} \quad	
	z_r(\bw) = S z_l(\bw),
\]
where $S$ is the shift map.
For every minimal subshift $X \subset \A^\N$, we thus have
\[
	z_l \cdot X = z_r \cdot X.
\]

The substitutions $c_i$ are not conjugate to $s_i$ but are related through substitutions $z_l$ and $z_r$ for $i=1,2$:
\begin{align}
\label{eq:conjugacy1}
s_1\circ z_l = z_r\circ c_1 = (1 \mapsto 21, 2 \mapsto 2131, 3 \mapsto 231);
\\
\label{eq:conjugacy2}
s_2\circ z_r = z_l\circ c_2 = (1 \mapsto 123, 2 \mapsto 1213, 3 \mapsto 13).
\end{align}
This allows to prove the following result, where a minimal subshift is {\em dendric} if it is generated by a dendric word.

\begin{proposition}
For all $(i_n)_{n \in \N} \in \{1,2\}^\N$, the sequence $\bc = (c_{i_n})_{n \in \N}$ is primitive if and only if the sequence $\bs = (s_{i_n})_{n \in \N}$ is so.
Furthermore, in this case we have $X_\bs = \rho(c_{121} \cdot X_\bc)$, where $\rho$ is the permutation $(23)$.
In particular, $X_\bs$ is a minimal dendric subshift so it has factor complexity $2n+1$ for all $n$.
\end{proposition}

\begin{proof}
    By Proposition~\ref{prop:primitiveness for C adic}, we know that $\bc$ is not primitive if and only if there exists $N$ such that for all $n \geq 0$, $i_{N+2n} = i_{N+2n+1}$.
Since $s_1^2(1)=s_2^2(1) = 1$, we deduce that if $\bc$ is not primitive, then $\bs$ is not primitive either.
To prove that $\bs$ is primitive when so is $\bc$, we may proceed like in Proposition~\ref{prop:primitiveness for C adic}. 
We define graphs similar to those of Figure~\ref{figure: graph for primitive} and show that $\bs$ is primitive.

Now assume that $\bc$ is primitive.
By minimality of $X_\bc$, we get $z_l \cdot X_\bc = z_r \cdot X_\bc$ and using Equations~\ref{eq:conjugacy1} and~\ref{eq:conjugacy2}, we have $z_l \cdot X_\bc = X_\bs$.
To end the proof, it suffices to observe that $z_l = \rho \circ c_{121}$.
\end{proof}

\begin{corollary}
For every totally irrational vector $\bx \in \Gamma$, the application of the semi-sorted Selmer algorithm yields a $\{s_1,s_2\}$-adic subshift which is minimal and dendric. 
\end{corollary}

\section*{Appendix}

Figure~\ref{fig:venn-diagram} is an alternative representation
of Figure~\ref{fig:proof-contracting}.

\begin{figure}%
\begin{center}
    \includegraphics[width=.70\linewidth]{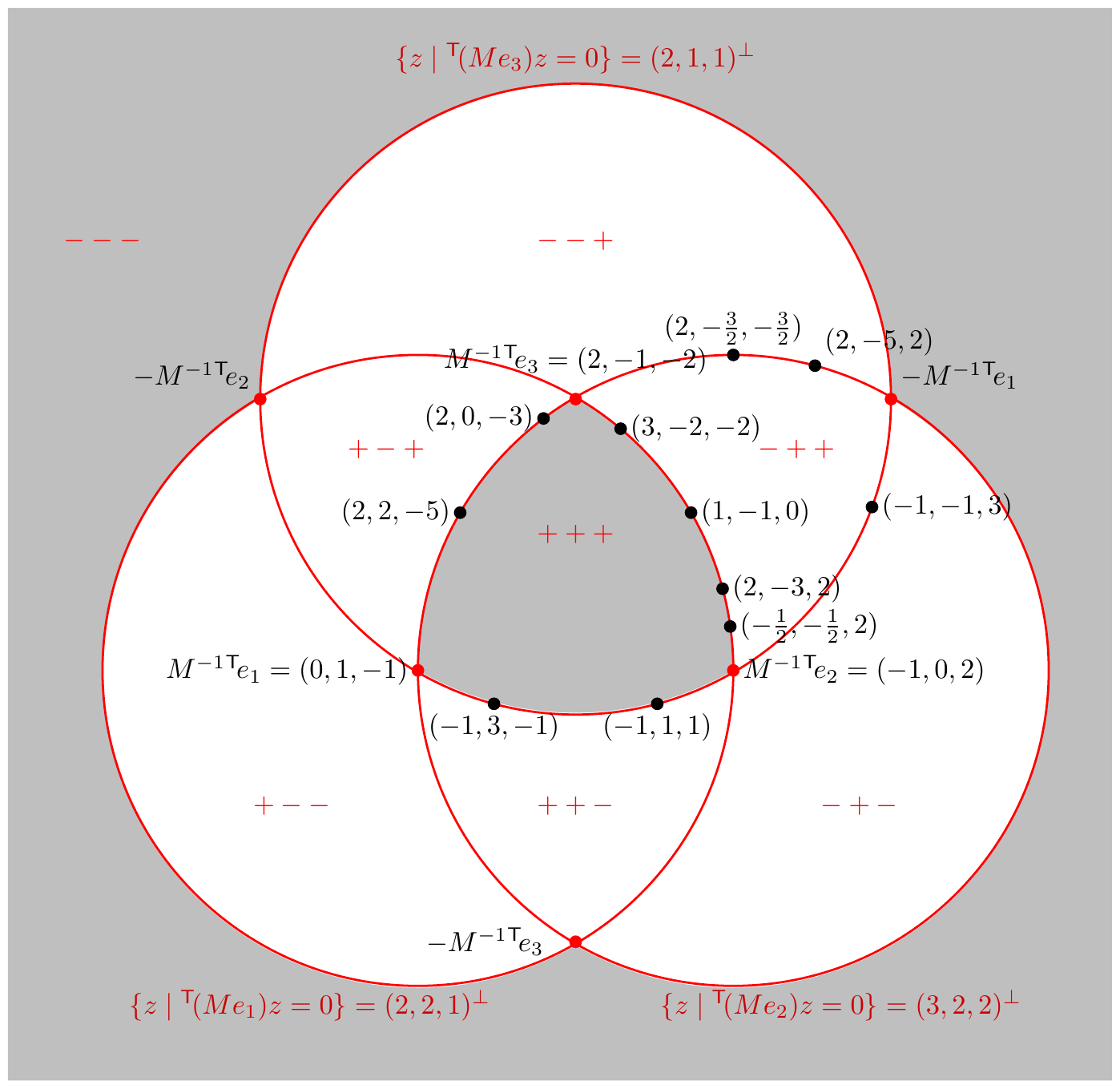}
\end{center}
    \caption{This illustration is an alternative representation
    of Figure~\ref{fig:proof-contracting}.
    Each plane orthogonal to $Me_1$, $Me_2$ or $Me_3$
    where $M=(C_1C_2)^3$
    passing through the origin intersects the sphere in a great circle which is
    represented as a circle in the figure.
    The grey regions represent the vectors $z\in
    \pm(\transposeENV{M} )^{-1}\,\R^3_{>0}$.
    The maximum of $\frac{\Vert \transposeENV{M} z\Vert_D}{\Vert z\Vert_D}$ is
    attained at $z=(2-3,2)$ with a value of $4/5$.}
\label{fig:venn-diagram}
\end{figure}

%% file: 36_cases_table.tex
\begin{tabular}{lllllll}
$u$ & $v$ & $z=u\wedge v$ & $\transpose{M}z$ & $z\in\R^3\setminus\pm (\transpose{M})^{-1}\R^3_{>0}$ & $\Vert z\Vert_D$ & $\Vert \transpose{M} z\Vert_D$ \\ \hline
$Me_1$ & $Me_2 $ & $\left(2,\,-1,\,-2\right)$ & $\left(0,\,0,\,1\right)$ & yes & $4$ & $1$ \\
$Me_1$ & $Me_3 $ & $\left(1,\,0,\,-2\right)$ & $\left(0,\,-1,\,0\right)$ & yes & $3$ & $1$ \\
$Me_1$ & $(e_1-e_3)$ & $\left(-2,\,3,\,-2\right)$ & $\left(0,\,-4,\,-3\right)$ & yes & $5$ & $4$ \\
$Me_1$ & $(e_1-e_2)$ & $\left(1,\,1,\,-4\right)$ & $\left(0,\,-3,\,-1\right)$ & yes & $5$ & $3$ \\
$Me_1$ & $(e_2-e_3)$ & $\left(-3,\,2,\,2\right)$ & $\left(0,\,-1,\,-2\right)$ & yes & $5$ & $2$ \\
$Me_1$ & $M(e_1-e_2)$ & $\left(-2,\,1,\,2\right)$ & $\left(0,\,0,\,-1\right)$ & yes & $4$ & $1$ \\
$Me_1$ & $M(e_2-e_3)$ & $\left(1,\,-1,\,0\right)$ & $\left(0,\,1,\,1\right)$ & yes & $2$ & $1$ \\
$Me_1$ & $M(e_1-e_3)$ & $\left(-1,\,0,\,2\right)$ & $\left(0,\,1,\,0\right)$ & yes & $3$ & $1$ \\
$Me_2 $ & $Me_3 $ & $\left(0,\,1,\,-1\right)$ & $\left(1,\,0,\,0\right)$ & yes & $2$ & $1$ \\
$Me_2 $ & $(e_1-e_3)$ & $\left(-2,\,5,\,-2\right)$ & $\left(4,\,0,\,-1\right)$ & yes & $7$ & $5$ \\
$Me_2 $ & $(e_1-e_2)$ & $\left(2,\,2,\,-5\right)$ & $\left(3,\,0,\,1\right)$ & yes & $7$ & $3$ \\
$Me_2 $ & $(e_2-e_3)$ & $\left(-4,\,3,\,3\right)$ & $\left(1,\,0,\,-2\right)$ & yes & $7$ & $3$ \\
$Me_2 $ & $M(e_1-e_2)$ & $\left(-2,\,1,\,2\right)$ & $\left(0,\,0,\,-1\right)$ & yes & $4$ & $1$ \\
$Me_2 $ & $M(e_2-e_3)$ & $\left(0,\,-1,\,1\right)$ & $\left(-1,\,0,\,0\right)$ & yes & $2$ & $1$ \\
$Me_2 $ & $M(e_1-e_3)$ & $\left(-2,\,0,\,3\right)$ & $\left(-1,\,0,\,-1\right)$ & yes & $5$ & $1$ \\
$Me_3 $ & $(e_1-e_3)$ & $\left(-1,\,3,\,-1\right)$ & $\left(3,\,1,\,0\right)$ & yes & $4$ & $3$ \\
$Me_3 $ & $(e_1-e_2)$ & $\left(1,\,1,\,-3\right)$ & $\left(1,\,-1,\,0\right)$ & yes & $4$ & $2$ \\
$Me_3 $ & $(e_2-e_3)$ & $\left(-2,\,2,\,2\right)$ & $\left(2,\,2,\,0\right)$ & yes & $4$ & $2$ \\
$Me_3 $ & $M(e_1-e_2)$ & $\left(-1,\,1,\,1\right)$ & $\left(1,\,1,\,0\right)$ & yes & $2$ & $1$ \\
$Me_3 $ & $M(e_2-e_3)$ & $\left(0,\,-1,\,1\right)$ & $\left(-1,\,0,\,0\right)$ & yes & $2$ & $1$ \\
$Me_3 $ & $M(e_1-e_3)$ & $\left(-1,\,0,\,2\right)$ & $\left(0,\,1,\,0\right)$ & yes & $3$ & $1$ \\
$(e_1-e_3)$ & $(e_1-e_2)$ & $\left(-1,\,-1,\,-1\right)$ & $\left(-5,\,-7,\,-4\right)$ & no & - & - \\
$(e_1-e_3)$ & $(e_2-e_3)$ & $\left(1,\,1,\,1\right)$ & $\left(5,\,7,\,4\right)$ & no & - & - \\
$(e_1-e_3)$ & $M(e_1-e_2)$ & $\left(0,\,2,\,0\right)$ & $\left(4,\,4,\,2\right)$ & no & - & - \\
$(e_1-e_3)$ & $M(e_2-e_3)$ & $\left(1,\,-2,\,1\right)$ & $\left(-1,\,1,\,1\right)$ & yes & $3$ & $2$ \\
$(e_1-e_3)$ & $M(e_1-e_3)$ & $\left(1,\,0,\,1\right)$ & $\left(3,\,5,\,3\right)$ & no & - & - \\
$(e_1-e_2)$ & $(e_2-e_3)$ & $\left(1,\,1,\,1\right)$ & $\left(5,\,7,\,4\right)$ & no & - & - \\
$(e_1-e_2)$ & $M(e_1-e_2)$ & $\left(1,\,1,\,-1\right)$ & $\left(3,\,3,\,2\right)$ & no & - & - \\
$(e_1-e_2)$ & $M(e_2-e_3)$ & $\left(-1,\,-1,\,2\right)$ & $\left(-2,\,-1,\,-1\right)$ & no & - & - \\
$(e_1-e_2)$ & $M(e_1-e_3)$ & $\left(0,\,0,\,1\right)$ & $\left(1,\,2,\,1\right)$ & no & - & - \\
$(e_2-e_3)$ & $M(e_1-e_2)$ & $\left(-1,\,1,\,1\right)$ & $\left(1,\,1,\,0\right)$ & yes & $2$ & $1$ \\
$(e_2-e_3)$ & $M(e_2-e_3)$ & $\left(2,\,-1,\,-1\right)$ & $\left(1,\,2,\,2\right)$ & no & - & - \\
$(e_2-e_3)$ & $M(e_1-e_3)$ & $\left(1,\,0,\,0\right)$ & $\left(2,\,3,\,2\right)$ & no & - & - \\
$M(e_1-e_2)$ & $M(e_2-e_3)$ & $\left(1,\,0,\,-1\right)$ & $\left(1,\,1,\,1\right)$ & no & - & - \\
$M(e_1-e_2)$ & $M(e_1-e_3)$ & $\left(1,\,0,\,-1\right)$ & $\left(1,\,1,\,1\right)$ & no & - & - \\
$M(e_2-e_3)$ & $M(e_1-e_3)$ & $\left(-1,\,0,\,1\right)$ & $\left(-1,\,-1,\,-1\right)$ & no & - & - \\
\end{tabular}

%% file: cassaigne_accelerated_ordre25.tikz
\begin{tikzpicture}
[scale=5, every node/.style={scale=0.8}]
\draw (0.6928, -0.5000) -- (0.7086, -0.5000);
\draw (0.7328, -0.5000) -- (0.7423, -0.5000);
\draw (0.7578, -0.3125) -- (0.5413, -0.3125);
\draw (0.7873, -0.3636) -- (0.7939, -0.3750);
\draw (-0.8660, -0.5000) -- (-0.4330, 0.2500);
\draw (0.7994, -0.3846) -- (0.6804, -0.3929);
\draw (0.0000, 1.0000) -- (-0.4330, 0.2500);
\draw (0.5774, -0.5000) -- (0.6186, -0.5000);
\draw (-0.8660, -0.5000) -- (0.0000, -0.5000);
\draw (0.7794, -0.3500) -- (0.6062, -0.3500);
\draw (0.6495, -0.3750) -- (0.7873, -0.3636);
\draw (0.0000, -0.5000) -- (0.2887, -0.5000);
\draw (0.6804, -0.3929) -- (0.6662, -0.3846);
\draw (0.6186, -0.5000) -- (0.6495, -0.5000);
\draw (0.7578, -0.3125) -- (0.7698, -0.3333);
\draw (0.0000, -0.5000) -- (-0.4330, 0.2500);
\draw (0.6495, -0.3750) -- (0.7086, -0.5000);
\draw (0.7217, -0.2500) -- (0.7423, -0.2857);
\draw (0.6495, -0.3750) -- (0.6662, -0.3846);
\draw (0.3464, -0.2000) -- (0.5196, -0.5000);
\draw (0.4330, -0.5000) -- (0.3464, -0.2000);
\draw (0.2165, -0.1250) -- (0.6495, -0.1250);
\draw (0.7086, -0.5000) -- (0.6298, -0.3636);
\draw (0.3464, -0.2000) -- (0.6495, -0.1250);
\draw (0.7794, -0.3500) -- (0.6298, -0.3636);
\draw (0.2165, -0.1250) -- (0.4330, -0.5000);
\draw (0.6736, -0.5000) -- (0.6062, -0.3500);
\draw (0.6928, -0.2000) -- (0.4330, -0.2500);
\draw (0.7423, -0.2857) -- (0.4949, -0.2857);
\draw (0.7994, -0.3846) -- (0.6662, -0.3846);
\draw (0.6495, -0.3750) -- (0.7939, -0.3750);
\draw (0.5774, -0.3333) -- (0.7698, -0.3333);
\draw (0.5413, -0.3125) -- (0.6186, -0.5000);
\draw (0.2165, -0.1250) -- (0.5774, 0.0000);
\draw (0.6804, -0.3929) -- (0.7328, -0.5000);
\draw (0.7217, -0.5000) -- (0.7328, -0.5000);
\draw (0.6928, -0.2000) -- (0.7217, -0.2500);
\draw (0.5774, -0.5000) -- (0.4949, -0.2857);
\draw (0.6495, -0.3750) -- (0.6298, -0.3636);
\draw (0.7794, -0.3500) -- (0.7873, -0.3636);
\draw (0.6298, -0.3636) -- (0.6062, -0.3500);
\draw (0.6495, -0.3750) -- (0.7217, -0.5000);
\draw (0.6804, -0.3929) -- (0.8042, -0.3929);
\draw (0.0000, 0.0000) -- (0.5774, 0.0000);
\draw (0.5774, 0.0000) -- (0.6495, -0.1250);
\draw (0.4330, 0.2500) -- (0.0000, 0.0000);
\draw (0.5196, -0.5000) -- (0.5774, -0.5000);
\draw (0.4330, -0.2500) -- (0.7217, -0.2500);
\draw (0.5413, -0.3125) -- (0.6495, -0.5000);
\draw (0.6804, -0.3929) -- (0.7423, -0.5000);
\draw (0.2887, -0.5000) -- (0.0000, 0.0000);
\draw (0.2165, -0.1250) -- (0.2887, -0.5000);
\draw (0.2165, -0.1250) -- (0.0000, 0.0000);
\draw (0.4330, -0.2500) -- (0.5774, -0.5000);
\draw (0.2887, -0.5000) -- (0.4330, -0.5000);
\draw (0.2165, -0.1250) -- (0.3464, -0.2000);
\draw (0.6928, -0.5000) -- (0.6298, -0.3636);
\draw (0.6736, -0.5000) -- (0.6928, -0.5000);
\draw (0.4330, -0.2500) -- (0.5196, -0.5000);
\draw (0.5774, -0.3333) -- (0.5413, -0.3125);
\draw (0.7578, -0.3125) -- (0.7423, -0.2857);
\draw (0.7794, -0.3500) -- (0.7698, -0.3333);
\draw (0.6662, -0.3846) -- (0.7939, -0.3750);
\draw (0.0000, -0.5000) -- (0.0000, 0.0000);
\draw (0.6662, -0.3846) -- (0.7328, -0.5000);
\draw (0.4330, -0.5000) -- (0.5196, -0.5000);
\draw (0.7217, -0.5000) -- (0.6662, -0.3846);
\draw (0.6928, -0.5000) -- (0.6062, -0.3500);
\draw (0.7086, -0.5000) -- (0.7217, -0.5000);
\draw (0.5774, -0.3333) -- (0.7578, -0.3125);
\draw (0.6298, -0.3636) -- (0.7873, -0.3636);
\draw (0.0000, 0.0000) -- (-0.4330, 0.2500);
\draw (0.5413, -0.3125) -- (0.7423, -0.2857);
\draw (0.6736, -0.5000) -- (0.6495, -0.5000);
\draw (0.7994, -0.3846) -- (0.7939, -0.3750);
\draw (0.6928, -0.2000) -- (0.3464, -0.2000);
\draw (0.6928, -0.2000) -- (0.6495, -0.1250);
\draw (0.4949, -0.2857) -- (0.6186, -0.5000);
\draw (0.6736, -0.5000) -- (0.5774, -0.3333);
\draw (0.5774, -0.3333) -- (0.6062, -0.3500);
\draw (0.4330, -0.2500) -- (0.3464, -0.2000);
\draw (0.4330, 0.2500) -- (0.0000, 1.0000);
\draw (0.5774, -0.3333) -- (0.6495, -0.5000);
\draw (0.7994, -0.3846) -- (0.8042, -0.3929);
\draw (0.5413, -0.3125) -- (0.4949, -0.2857);
\draw (0.4330, -0.2500) -- (0.4949, -0.2857);
\draw (0.7217, -0.2500) -- (0.4949, -0.2857);
\draw (0.7698, -0.3333) -- (0.6062, -0.3500);
\draw (0.4330, 0.2500) -- (0.5774, 0.0000);
\draw (0.4330, 0.2500) -- (-0.4330, 0.2500);
\node at (-0.4330, -0.2500) {$11$};
\node at (-0.1443, -0.0833) {$121$};
\node at (0.0962, -0.3333) {$12^{2}1$};
\node at (0.1684, -0.2083) {$12^{3}1$};
\node at (0.3127, -0.3750) {$12^{4}1$};
\node at (0.3320, -0.2750) {$12^{5}1$};
\node at (0.0000, 0.5000) {$22$};
\node at (0.0000, 0.1667) {$212$};
\node at (0.3368, 0.0833) {$21^{2}2$};
\node at (0.2646, -0.0417) {$21^{3}2$};
\node at (0.4811, -0.0833) {$21^{4}2$};
\node at (0.4041, -0.1500) {$21^{5}2$};
\end{tikzpicture}